%% file: main.tex
\documentclass[reqno,  oneside, 11pt]{amsart}
\usepackage{amsmath, mathrsfs, url, color}
\usepackage[margin=2cm, a4paper]{geometry}
\usepackage{enumitem}
\usepackage{bbm}
\usepackage{graphicx}
\usepackage{subcaption}
\allowdisplaybreaks
\usepackage{color}

\newtheorem{theorem}{Theorem}[section]
\newtheorem{cor}{Corollary}[section]
\newtheorem{lem}{Lemma}[section]
\newtheorem{prop}{Proposition}[section]
\newtheorem{rem}{Remark}[section]

\newtheorem{asump}{Assumption}[section]
\allowdisplaybreaks 
\DeclareMathOperator{\tr}{trace}

\definecolor{darkgreen}{rgb}{0,.6,0}

\title[Milstein-type schemes for McKean--Vlasov SDEs]{Milstein-type schemes for McKean--Vlasov SDEs driven by Brownian motion and Poisson random measure\\ (with super-linear coefficients)}

\usepackage[foot]{amsaddr}

\author{Sani Biswas$^{\mbox{\textasteriskcentered}}$}
\address{\textasteriskcentered$_{\normalfont{\mbox{Centro de Modelamiento Matem\'atico, Universidad de Chile, Chile. E-Mail: {sbiswas@cmm.uchile.cl}}}}$}

\author{Chaman Kumar$^{\mbox{\textdagger}}$}
\address{\textdagger$_{\normalfont{\mbox{Department of Mathematics, Indian Institute of Technology Roorkee, India. E-Mail: {chaman.kumar@ma.iitr.ac.in}}}}$}

\author{Christoph Reisinger$^{\mbox{\textdaggerdbl}}$}
\address{\textdaggerdbl$_{\normalfont{\mbox{Mathematical Institute, University of Oxford, United Kingdom. E-Mail: {christoph.reisinger@maths.ox.ac.uk}}}}$}

\author{Verena Schwarz$^{\mbox{\textsection}}$}
\address{\textsection$_{\normalfont{\mbox{Department of Statistics, University of Klagenfurt,   Austria. E-Mail: {verena.schwarz@aau.at}}}}$}

\begin{document}

	\begin{abstract}
		In this work, we present a  general Milstein-type scheme for McKean--Vlasov  stochastic differential equations (SDEs) driven by Brownian motion and Poisson random measure and the associated system of interacting particles where  drift, diffusion and jump coefficients may grow super-linearly in the state variable and linearly in the measure component.  
		The strong rate of $\mathcal{L}^2$-convergence of the proposed scheme is shown to be  arbitrarily close to one under appropriate regularity assumptions on the coefficients. 
		For the derivation of the Milstein scheme and to show its strong rate of convergence, we provide an It\^o formula for the interacting particle system connected with the McKean--Vlasov SDE driven by Brownian motion and Poisson random measure. 
		Moreover, we use  the notion of Lions derivative to examine our results.
		The two-fold challenges arising due to the presence of the empirical measure and super-linearity of the jump coefficient are resolved by identifying and exploiting an appropriate coercivity-type condition. 
	\end{abstract}
	
	\maketitle
	\noindent
	\textbf{Keywords.} McKean--Vlasov equation, L\'evy  noise, super-linear coefficient, interacting particle system,  Lions derivatives.
	\\ \\
	\textbf{AMS Subject Classifications.}  65C05, 65C30, 65C35, 60H35.
	\section{Introduction}
	\input{Newcommand_23-oct-2023}

	Consider a probability space   $(\Omega, \mathscr{F},\{\mathscr{F}_t\}_{t\geq 0}, {P})$ satisfying the usual conditions of completeness and right continuity, which supports a standard Brownian motion $w:=\{w_t\in\mathbb R^d\}_{t\geq 0}$ and an independent Poisson random measure $n_p(dt, dz)$ having compensator  $\nu(dz)  dt$, where $\nu(Z)<\infty$ is the L\'evy measure defined on a $\sigma-$finite measurable space $(Z,\mathscr Z)$. Further, define the compensated Poisson random measure by $\tilde n_p(dt,dz):=n_p(dt,dz)-\nu(dz) \, dt$.  
 Also, suppose that $ \mathscr P_2(\mathbb R^d)$ is the space of square integrable probability measures on the measurable space $(\mathbb R^d, \mathscr B(\mathbb R^d))$, equipped  with the $\mathcal{L}^2$-Wasserstein metric  
	$
	\displaystyle \mathcal W_2(\mu_1, \mu_2):=\inf_{\pi\in \Pi(\mu_1,\mu_2)}\Big[\int_{\mathbb R^d\times \mathbb R^d}|x-y|^2\pi(dx, dy)\Big]^{1/2}
	$ where  $\Pi(\mu_1,\mu_2)$ is the set of all couplings of  $\mu_1, \mu_2 \in \mathscr P_2(\mathbb R^d)$.
 Clearly, $ \mathscr P_2(\mathbb R^d)$ is a complete metric space under the metric $\mathcal W_2$. 
Let $b:\mathbb{R}^d\times \mathscr {P}_2(\mathbb{R}^d)
 \to \mathbb{R}^d$ and $\sigma: \mathbb{R}^d\times \mathscr {P}_2(\mathbb{R}^d) \to \mathbb{R}^{d\times m}$   be $\mathscr{B}(\mathbb R^d)\otimes \mathscr{B}(\mathscr{  P}_2(\mathbb R^d))$-measurable functions and let $\gamma:\mathbb{R}^d\times\mathscr {P}_2(\mathbb{R}^d)\times Z
 \to\mathbb{R}^{d}$  be a $\mathscr{B}(\mathbb R^d)\otimes \mathscr{B}(\mathscr{P}_2(\mathbb R^d))\otimes \mathscr Z$-measurable function.
	For a fixed constant  $ T>0$, consider the  McKean--Vlasov stochastic differential equation (SDE) driven by Brownian motion and compensated Poisson random measure as given below, 
	\begin{align} \label{eq:sde}
		x_t=x_0 + &\int_0^t b(x_s, \mu_s^x)ds +   \int_0^t \sigma(x_s,\mu_s^x)dw_s
		+\int_0^t\int_Z \gamma(x_s,\mu_s^x,z)\tilde{n}_p(ds, dz)
	\end{align}
	almost surely for any $t\in [0, T]$, where $ \mu^x_t$ denotes the law of $x_t$. 
The defining characteristic of McKean--Vlasov SDEs is the presence of the law $ \mu^x_t$ in the coefficients, which necessitates further investigation beyond that for classical SDEs.
	McKean--Vlasov equations such as \eqref{eq:sde} and its special cases are widely used in different fields, including finance, biological sciences, neuroscience, statistical physics and  machine learning,  see \cite{Baladron2012, Bolley2011, Bossy2015, Carmona2018-I, Carmona2018-II,  Carrillo2010,    Chiara2017, Delarue2015, Dreyer2011,  Goddard2012, Guhlke2018, Holm2006, Jin2020, Kolokolnikov2013,  Mehri2020,  Ullner2018} and the references therein. 
	Even more so than classical SDEs, McKean--Vlasov  SDEs too are often difficult to solve explicitly, as the law $ \mu^x_t$ is typically unknown, and this has led to keen interest by researchers in constructing numerical solutions of such equations.  
	Broadly speaking, most numerical schemes for McKean--Vlasov SDEs can be broken into two steps. 
	First, the  law $ \mu^x_t$ is approximated by the empirical law and then a time-discretization is performed.   
	Indeed,   consider $N$ \textit{i.i.d} copies $\{x_0^1,\ldots,x_0^N\}$ of $x_0$,  $\{w^1,\ldots,w^N\}$ of $w$, $\{\tilde{n}_p^1,\ldots,\tilde{n}_p^N\}$ of  $\tilde{n}_p$ and define the \textit{non-interacting  particle system} as
	\begin{align}
		x^i_t=&x^i_0+\int_0^tb(x^i_s, \mu_s^{x})ds+ \int_0^t\sigma (x^i_s,\mu_s^{x})dw^{i}_s+\int_0^t\int_Z\gamma(x^i_s,\mu_s^{x},z)\tilde{n}_p^i(ds,dz) \label{eq:non-int}
	\end{align}
	almost surely for any $t\in [0, T]$ and $i\in\{1,\ldots,N\}$ where  $\mu_t^{x^i}=\mu_t^x$.
	On estimating the law $\mu_t^x$ by the empirical law    
	$$
	\mu_t^{x,N}:= \frac{1}{N}\sum_{i=1}^{N}\delta_{x_t^{i,N}},
	$$
	one obtains the  following \textit{interacting particle system}
	\begin{align} \label{eq:int}
		x_t^{i, N}&=x_0^{i}+\int_0^t b(x_s^{i, N}, \mu_s^{x, N})ds + \int_0^t\sigma(x_s^{i, N},\mu_s^{x,N})dw_s^{ i} + \int_0^t \int_Z \gamma(x^{i,N}_s,\mu_s^{x,N},z)\tilde{n}_p^i(ds,dz)
	\end{align}
	almost surely for any $t\in [0, T]$ and $i\in\{1,\ldots,N\}$.
Equation \eqref{eq:non-int} can be considered to be the limiting case of \eqref{eq:int} when $N\to \infty$ and this phenomenon is known as the \textit{propagation of chaos} (PoC). 
Then, one can discretize \eqref{eq:int} to obtain fully discretized numerical schemes. 
The aim of this article is to construct a Milstein-type numerical scheme and investigate its rate of strong convergence. 
  
\subsection*{Existing Results}

{The strong well-posedness and PoC of McKean--Vlasov equations 
have been widely discussed in the literature.
When $\gamma\equiv 0$ and the coefficients $b$ and $\sigma$ are Lipschitz continuous in both state and measure variables, the well-posedness and PoC  are known from \cite{Carmona2018-I}. 
Further, \cite{Carmona2018-II} provides well-posedness and PoC for McKean--Vlasov SDE with common noise under  Lipschitz continuity of all coefficients in both the variables.
An extension of the results of \cite{Carmona2018-I} for super-linearly growing drift coefficient in the state variable is done in \cite{Salkeld2019}, while \cite{C.Kumar2021} extends the results of both \cite{Carmona2018-I, Salkeld2019} by including super-linear growth in the state variable of all the coefficients. Recently, \cite{Chen2024} proved well-posedness and PoC for McKean--Vlasov SDEs when the drift coefficient exhibits a convolution-type non-linearity in the measure, which is further extended to include super-linear diffusion coefficients of a similar type, with additional results on ergodicity, in \cite{Chen2023}.
Finally, \cite{Neelima2020} extends the results of \cite{C.Kumar2021} in the  absence of common noise to the McKean--Vlasov SDE driven by L\'evy noise by identifying a suitable monotonicity- and coercivity-type conditions on the coefficients, allowing all the coefficients (including the jump coefficient) to grow super-linearly. 
We refer the interested readers to the relevant references mentioned on the articles listed above for a more comprehensive discussion of the subject.} 

{The time discretization of the interacting particle system associated with the McKean--Vlasov SDE was initiated in \cite{Reis2019}, where an explicit (drift-tamed)  Euler-type scheme is proposed and its strong order is shown to be $1/2$. 
There, the drift coefficient is required to satisfy a one-sided Lipschitz condition and the diffusion coefficient is assumed  to be Lipschitz continuous in the state variable, while with respect to  the measure component, both of them  are taken to be  Lipschitz continuous.
Furthermore, an explicit first-order Milstein-type scheme for the interacting particle system related to McKean--Vlasov SDEs (with super-linear drift coefficient in the state variable) was first developed in \cite{bao2021-II, Kumar2021} (simultaneously but independently) with the help of the notion of Lions derivative. 
Moreover,  half- and first-order (tamed) schemes for such interacting particle systems with common noise are investigated in \cite{C.Kumar2021} under appropriate coercivity and monotonicity assumptions on the coefficients with respect to the state variable. 
Tamed schemes for McKean--Vlasov SDEs with super-linearly growing coefficients as mentioned above and hereafter are inspired by the corresponding  seminal works   \cite{hutzenthaler2015, hutzenthaler2012} on the tamed order-half  schemes for classical SDEs.
We also refer the interested reader to \cite{Reisinger2022} for an  adaptive Euler scheme for interacting particle systems with super-linearly growing coefficients in the state variable and its application to a mean-field FitzHugh--Nagumo model from neuroscience,  \cite{Li2022} for  Euler schemes for interacting particle systems under local Lipschitz conditions in the state component, \cite{bao2021-I} for McKean--Vlasov SDE with irregular coefficients, \cite{agarwal2023} for  numerical approximation of  McKean--Vlasov SDEs using stochastic gradient descent, which avoids the use of  interacting particle systems, \cite{Chen2021} for a flexible split‐step scheme for  McKean--Vlasov SDEs and its half order of convergence, \cite{Chen2024, Chen2023} for the numerical approximation of fully super-linear convolution-type McKean--Vlasov SDEs and the associated interacting particle systems, \cite{bao2021-III} for Milstein schemes and antithetic multi-level Monte Carlo sampling for McKean--Vlasov SDEs with delay and \cite{Biswas2022} for a randomized scheme for McKean--Vlasov SDE with common noise without assuming any derivative conditions on the drift coefficient,  along with the references contained therein.}

Recently, motivated by the techniques developed in \cite{Dareiotis2016, Kumar2017a} for classical SDEs with super-linear drift coefficient,  an explicit half-order Euler-type scheme for the interacting particle system coupled with McKean--Vlasov SDEs driven by L\'evy noise  is proposed and analysed in \cite{Neelima2020},
which also allows super-linear growth of the jump coefficient in the state variable under suitable coercivity  and monotonicity conditions.
To the best of our knowledge, there is no literature on higher order schemes such as Milstein-type schemes for McKean--Vlasov SDEs driven by Brownian motion and Poisson random measure, even for Lipschitz coefficients, which is the main objective of this article.

\subsection*{Challenges and Novelty}
Even though the interacting particle system in \eqref{eq:int} can be regarded as an $N \times d$-dimensional SDE driven by Brownian motion and Poisson random measure, known results on time-discretization schemes for  SDEs driven by Brownian motion and Poisson random measure or L\'evy noise -- see, for example, \cite{Dareiotis2016, hutzenthaler2015, hutzenthaler2012,  Kumar2017} and the references therein -- cannot be used because of the dependence of the constants appearing in their proofs on the dimension $N \times d$, which explodes in the limiting case $N \to \infty$.  
This is a general challenge for all time-discrete schemes for interacting particle systems and thus is present for our Milstein-type scheme.  
However, this difficulty becomes more pronounced in our case mainly due to the presence of additional terms in the numerical approximation scheme arising from the empirical law, including measure derivatives and differences of coefficients with slightly different measure components.
This is handled by a careful manipulation and a type of Taylor expansion given in Lemma \ref{lem:MVT}. 
Moreover, we allow the jump coefficient to grow super-linearly in the state variable. This has not been treated in the literature for higher-order schemes, even in the case of standard SDEs driven by Brownian motion and Poisson random measure. The approximation of SDEs driven by Brownian motion and Poisson random measure with a super-linear jump coefficient using an Euler--Maruyama scheme is studied in \cite{Chen2019}.
The super-linearity of the jump coefficient plus the interaction of the particles through the empirical measure brings additional complexities in comparison with existing results \cite{ Kumar2021a, Kumar2021, Kumar2017} on super-linear drift and diffusion coefficients. On the one hand, it requires additional taming of all the terms in the Milstein-type scheme that include the jump coefficient, which makes additional assumptions necessary. On the other hand, we have to identify appropriate conditions for our coefficients that allow for a super-linear jump coefficient. These conditions have not appeared in literature before, and thus, a different approach is required to investigate our Milstein-type scheme in this setting. The main difficulty in the proofs arises from the complexity of the scheme due to the versatile interaction of the coefficients in the approximation scheme. These make very careful estimates necessary and require us at certain points to use more advanced remainders in the Taylor like expansion of our coefficients than in the previous literature, see Lemma \ref{lem:MVT}. 

\subsection*{Our Contribution}
	 We propose a  Milstein-type scheme for the interacting particle system in \eqref{eq:int} associated with the McKean--Vlasov SDE \eqref{eq:sde}. 
	 The drift, diffusion and jump coefficients are assumed to satisfy a polynomial Lipschitz condition in the state variable and Lipschitz condition in the measure variable, \textit{i.e.} in the  $\mathcal{W}_2$-Wasserstein metric. 
	 New coercivity-type and growth conditions are identified to allow for a super-linear growth of the jump coefficient in the state variable.  
	 The first and second order derivatives of the coefficients with respect to the state component are assumed to be bounded, while we use the notion of Lions derivative (see \cite{Cardaliaquet2013}) to differentiate the coefficients with respect to the measure component. 
	 Further regularity on the first order derivatives of the coefficients with respect to the state and measure component are assumed. 
     In order to develop a Milstein-type scheme for the interacting particle system in \eqref{eq:int} and to investigate the rate of its strong convergence, we derive a version of It\^o's formula for the interacting particle system by the empirical projection method, which is different from \cite{pham2023, Li2018}. 
     To be able to treat the cases of super-linear coefficients in the state component, which make a taming of the Milstein-type scheme necessary, and Lipschitz continuous coefficients at once, we define a general approximation scheme in terms of a sequence. 
     In this setting we prove that the strong $\mathcal{L}^2$-convergence of the Milstein-type scheme \eqref{eq:scm} is arbitrarily close to one. 
     The general conditions allow the proposed class of taming schemes to be applied to a large variety of McKean--Vlasov equations, with the tamed coefficients adapted to the problem at hand. We defer concrete examples and their implementation to a subsequent paper. 

\subsection*{Notations} 
{The euclidean norm on $\mathbb R^d$ and the matrix norm  on $\mathbb R^{d \times m}$ are both expressed using the same notation, $ |\cdot|$. 
The Borel $\sigma$-algebra on a topological space  $\mathcal T$ is denoted by the symbol $ \mathscr B(\mathcal T) $.
The symbol  $A^u$ is used to signify the $u$-th element of a vector, $A \in \mathbb R^d$ for any $u\in\{1,\ldots,d\}$. 
We use $A^v$ and $A^{(u)}$ to denote the $v$-th  column and the $u$-th row of a matrix $A \in \mathbb R^{d \times m}$, respectively for any $u\in\{1,\ldots,d\}$ and $v\in\{1,\ldots,m\}$. 
Furthermore, we denote the $(u,v)$-th entry of a matrix, $A \in \mathbb R^{d \times m}$ by $A^{uv}$, for any $u\in\{1,\ldots,d\}$ and  $v\in\{1,\ldots,m\}$. The notation $A^*$ and $tr(A)$ are used to denote  transpose and trace of a matrix $A$.}
The Dirac measure with a center at $ a\in\mathbb R^d $ is denoted as $ \delta_a $, and the indicator function of a set $A$ is denoted as $1_A$.
	The gradient and Hessian matrix  of $f:\mathbb{R}^d\times\mathscr P_2(\mathbb R^d)\rightarrow\mathbb{R}$ with respect to the state variable are represented by $\partial_x f:\mathbb{R}^d\times\mathscr P_2(\mathbb R^d)\rightarrow\mathbb{R}^d$ and  $\partial_{x}^2 f:\mathbb{R}^d\times\mathscr P_2(\mathbb R^d)\rightarrow\mathbb{R}^{d\times d}$, respectively.  
Also, $\partial_\mu f:\mathbb{R}^d\times\mathscr P_2(\mathbb R^d)\times\mathbb{R}^d\rightarrow\mathbb{R}^d$ and $\partial_{\mu}^2 f:\mathbb{R}^d\times\mathscr P_2(\mathbb R^d)\times\mathbb{R}^d\times\mathbb{R}^d\rightarrow\mathbb{R}^{d\times d}$ are used to denote the first- and second- order Lions derivatives (derivative with respect to the measure component in the sense of \cite{Cardaliaquet2013}) of the function $f(x,\mu)$, respectively. 
Furthermore, $\partial_{x}\partial_{\mu} f:\mathbb{R}^d\times\mathscr P_2(\mathbb R^d)\times\mathbb{R}^d\rightarrow\mathbb{R}^{d\times d}$ and $\partial_{y}\partial_{\mu} f:\mathbb{R}^d\times\mathscr P_2(\mathbb R^d)\times\mathbb{R}^d\rightarrow\mathbb{R}^{d\times d}$ are used to denote the gradient   of $\partial_{\mu}f(x,\mu,y)$ with respect to $x$ and $y$, respectively. 
The notation $\mathcal C^2(\mathbb R^d)$ refers to the set of continuous functions $f:\mathbb R^d\to \mathbb R$ with continuous second-order partial derivatives everywhere.
\color{black}
Also, we use $\mathcal C^2(\mathbb R^d\times\mathscr{P}_2(\mathbb R^d))$ to denote the class of  continuous functions $f:\mathbb R^d\times\mathscr P_2(\mathbb R^d)\to \mathbb R$ having  continuous derivatives  $\partial _x f(x,\mu)$, $\partial _x^2 f(x,\mu)$, $\partial _\mu f(x,\mu,y)$, $\partial _{x}\partial _{\mu} f(x,\mu,y)$, $\partial _{y}\partial _{\mu} f(x,\mu,y)$ and $\partial _\mu^2 f(x,\mu,y,y')$, for any $\mu\in\mathscr{P}_2(\mathbb R^d)$ and $x,y,y'\in\mathbb R^d$. 
For  $f\in \mathcal C^2(\mathbb R^d\times\mathscr{P}_2(\mathbb R^d))$, we define the following operators
\begin{tabular}{ll}
$\mathfrak D_{x}^{b}f(x,\mu) := \partial_x f(x, \mu)b(x, \mu),$ & $\mathfrak D_{\mu}^{b}f(x,\mu, y):= \partial_\mu f(x, \mu, y)b(y, \mu),$ 
\\
$\mathfrak D_{x}^{\sigma^{\ell_1}}f(x,\mu):= \partial_x f(x, \mu)\sigma^{\ell_1}(x, \mu),$ & $\mathfrak D_{\mu}^{{\sigma^{\ell_1}}}f(x,\mu, y):= \partial_\mu f(x, \mu, y)\sigma^{\ell_1}(y, \mu),$
\\
 $\mathfrak D_{x}^{\gamma(z)}f(x,\mu) := \partial_x f(x, \mu) \gamma(x, \mu, z),$ & $ \mathfrak D_{\mu}^{\gamma(z)}f(x,\mu, y):= \partial_\mu f(x, \mu, y)\gamma(y, \mu,  z),$
 \\
 $\mathfrak D_{xx}^{\sigma}f(x,\mu):=tr[\partial_x^2 f(x,\mu)\sigma(x,\mu)\sigma^*(x,\mu)],$ & $\mathfrak D_{x\mu}^{\sigma}f(x,\mu):=tr[\partial_x \partial_\mu  f(x,\mu,x)\sigma(x,\mu)\sigma^*(x,\mu)], $
 \\
 $\mathfrak D_{y\mu}^{\sigma}f(x,\mu, y)\hspace{-0.01cm}:= tr[\partial_{y}\partial_{\mu} f(x,\mu,y)\sigma(y,\mu)\sigma^*(y,\mu)], $ & $\mathfrak D_{\mu\mu}^{\sigma}f(x,\mu, y):= tr[\partial_{\mu}^2 f(x,\mu,y,  y)\sigma(y,\mu)\sigma^*( y,\mu)] $
 \end{tabular}
for any $\ell_1\in \{1, \ldots, m\}$, $x, y \in \mathbb{R}^d$, $ z\in Z$ and $\mu \in \mathscr{P}_2(\mathbb{R}^d)$. 
The above operators are used in the It\^o's formula (\textit{i.e.}, Lemma \ref{lem:ito})  with $f(\cdot, \cdot)=F(\cdot, \cdot)$ and subsequently in the derivation  and investigation of the rate of convergence of the Milstein-type scheme  with $f(\cdot, \cdot)=b^u(\cdot, \cdot), \sigma^{u \ell}(\cdot, \cdot), \gamma^u(\cdot, \cdot, \bar z)$ for every $\bar z\in Z$, $u\in\{1, \ldots, d\}$, $\ell \in \{1, \ldots, m\}$.  
If $f$ is additionally a function of $Z$ (\textit{i.e.}, when $f(\cdot, \cdot)=\gamma^u(\cdot, \cdot, \bar z)$), then there is an additional dependence on $\bar z\in Z$ of these operators and this should be clear from the context.
{We write $EX$ for expectation of a random variable $X$ and $E_{t}X$ for the conditional expectation of $X$ given $\mathscr{F}_{t}$.}
The minimum of two real numbers $a$ and $b$ is symbolised by $a \wedge b$ and their maximum by $a \vee b$. 
We use the notation $K$ to denote the positive generic constant which may vary from place to place but never depends on $N$ and the discretization step-size. 
	\section{Main result and Relevant pre-requisites}  \label{sec:main_result}
	In this section, we first recall the well-posedness and moment stability of the McKean--Vlasov SDE \eqref{eq:sde} driven by Brownian motion and Poisson random measure and the associated interacting particle system \eqref{eq:int}, and a   \textit{propagation of chaos} result.
	Moreover, this article's  main result  on the rate of strong convergence of the proposed Milstein-type scheme \eqref{eq:scm} below for the interacting particle system \eqref{eq:int}  
 is given in   Theorem \ref{thm:mr}, with its proof in Section \ref{sec:rate}. 
 
\subsection{Well-posedness and Propagation of Chaos}
For well-posedness  of the McKean--Vlasov SDE \eqref{eq:sde} and the connected interacting particle system~\eqref{eq:int}, let us make the following assumptions. 

First,  fix $\bar p> 4$. 
	\begin{asump}\label{asum:ic}
		$E|x_0|^{\bar p}<\infty$.
	\end{asump}
	\begin{asump} \label{asum:lip}
		There exist   constants $C>0$ and $\alpha>1$  such that for any $x, x'\in \mathbb R^d$ and $\mu, \mu'\in~\mathscr P_2(\mathbb R^d)$ 
		\begin{align*}
			2(x-x')\big(b(x, \mu)-b(x', \mu'))+ \alpha|\sigma(x, \mu)-\sigma(x', \mu')|^2
			& + \alpha\int_Z |\gamma(x,\mu,z)-\gamma(x',\mu',z)|^2 \nu(dz)
			\\
			&   \leq  C\{|x-x'|^2+\mathcal W_2^2(\mu, \mu')\}. 
		\end{align*}
	\end{asump}

	\begin{asump}
		\label{asum:lin1}
There is a constant $C>0$   such that for any  $x\in \mathbb R^d$ and $\mu\in  \mathscr P_2(\mathbb R^d)$
		\begin{align*}
			2xb(x,\mu)+|\sigma(x,\mu)|^2 + \int_Z |\gamma(x,\mu,z)|^{2} \nu(dz) \leq  C\{1+|x|^2+\mathcal W_2^2(\mu, \delta_0)\}.
		\end{align*}
	\end{asump}

	\begin{asump} \label{asum:con}
		For every $\mu\in  \mathscr P_2(\mathbb R^d)$, the map $\mathbb R^d\ni x\mapsto b(x,\mu)$ is   continuous.
	\end{asump}

\begin{asump} \label{asum:lin**}
	There is  a constant $C>0$  such that for any  $x\in\mathbb R^d$, $\mu\in\mathscr{P}_2(\mathbb R^d)$
	\begin{align*}
		&2|x|^{{\bar p}-2}xb(x,\mu)
		+(\bar p-1)|x|^{{\bar p}-2}\big|\sigma(x,\mu)\big|^2  \notag
		\\
		&\quad 
		+2(\bar p-1)\int_Z\big| \gamma(x,\mu,z) \big|^2\int_{0}^1 (1-\theta)  \big |x+\theta\gamma(x,\mu,z) \big|^{\bar p-2} d\theta \nu(dz) \leq C\big\{1+|x|^{\bar p}+\mathcal{W}_2^{\bar p}(\mu,\delta_0)\big\}. 
	\end{align*}
\end{asump}

	The following proposition is proved  in   \cite[Theorem 2.1]{Neelima2020}. 
	\begin{prop} \label{prop:mb:mvsde}
		If Assumptions \mbox{\normalfont  \ref{asum:ic}} to  \mbox{\normalfont \ref{asum:con}}   are satisfied, then there exists a unique solution of  the McKean--Vlasov SDE \eqref{eq:sde} as well as of its associated interacting particle system \eqref{eq:int}. 
		Further, when Assumption \mbox{\normalfont \ref{asum:lin**}} also holds,  {then there exists a constant $K>0$ such that, for all $N\in \mathbb N$}
		\begin{align*}
		\sup_{t\in[0,T]}E|x_t|^{\bar p}\leq K \mbox{ and }\sup_{i\in\{1,\ldots,N\}}\sup_{t\in[0,T]}E|x_t^{i,N}|^{\bar p}\leq K.
		\end{align*}

	\end{prop}
	The asymptotic behaviour of the interacting particles system \eqref{eq:int} is shown in the following proposition, which is known in the literature as  \textit{propagation of chaos} and its proof  can be found  in \cite{Neelima2020}.
	\begin{prop} \label{prop:poc}
		If Assumptions \mbox{\normalfont  \ref{asum:ic}} to \mbox{\normalfont  \ref{asum:con}}    are satisfied, then the rate of convergence of the interacting particle system \eqref{eq:int} to the non-interacting particle system \eqref{eq:non-int}  is given~by
		\begin{align*}
			\sup_{i\in\{1,\ldots,N\}}\,\,\sup_{ t\in[0,T]}E|x_t^i-x_t^{i, N}|^2\leq K
			\begin{cases}
				N^{-1/2}, & \mbox{ if }  d <4,
				\\
				N^{-1/2}\ \log_2 N, & \mbox{ if } d=4,
				\\
				N^{-2/d}  &  \mbox{ if }d>4,
			\end{cases}
		\end{align*}
		where $K>0$ is a constant independent of $N$.
	\end{prop} 
	
\subsection{Milstein-type Scheme and Main Result} 
For the purpose of defining our scheme, we introduce the following sequences. 
	Given $x,y \in \mathbb{R}^d$ and $\mu \in \mathscr{P}_2(\mathbb{R}^d)$, consider 

		\begin{tabular}{lll}
			$ \big\{ \bbtamex \big\}_{n\in \mathbb N}$ & $ \big\{ \sigtamex \big\}_{n\in \mathbb N}$ & $  \big\{ \sigxsigtamex \big\}_{n\in \mathbb N}$
			\\
			$\big\{ \sigmusigtamex \big\}_{n\in \mathbb N}$ & $ \big\{ \sigxgamtamex \big\}_{n\in \mathbb N} $ & $  \big\{ \sigmugamtamex \big\}_{n\in \mathbb N} $
			\\
			$\big\{ \gamtamex \big\}_{n\in \mathbb N} $ & $\big\{ \gamxsigtamex \big\}_{n\in \mathbb N} $ & $  \big\{ \gammusigtamex \big\}_{n\in \mathbb N} $
			\\
			$ \big\{ \gamxgamtamex \big\}_{n\in \mathbb N} $ & $\big\{ \gammugamtamex \big\}_{n\in \mathbb N}$ &
		\end{tabular}
	
	\noindent
	{for any $\ell, \ell_1 \in \{1, \ldots, m\}$ and $u\in\{1, \ldots, d\}$}. 
	We drop the superscript `$u$'  from the above expressions for the ($d\times 1$) vector representation of the above expressions.  
 The interval $[0,T]$ is partitioned into $n\in\mathbb N$ sub-intervals $0=t_0^n  <t_1^n < \cdots <  t_{n-1}^n < t_n^n=T$ and  define for $n\in\mathbb N$, $k\in\{1,\ldots,n\}$ and $t\in[t_{k-1},t_k)$, $\kappa_n(t):=t_{k-1}$.   
 Now, we propose the following Milstein-type scheme for the interacting particle system \eqref{eq:int}
\begin{align}  \label{eq:scm}
	&x_t^{i, N,n}= x_{0}^{i}+\int_{0}^{t}  \Btame ds \notag
	\\
	&+\int_{0}^{t}  \Big\{\Sigtame + \sum_{\mathfrak q=1}^2\Sigq\Big\}dw_s^{i}    \notag
	\\
	& +\int_{0}^{t}\int_Z  \Big\{\Gamtame  
	+ \sum_{\mathfrak q=1}^2\Gamq \Big\}  \tilde n_p^i(ds,d\bar z) 
\end{align}
almost surely for any $t\in[ 0, T]$ and $i\in\{1,\ldots,N\}$ where  $(\widehat{\sigma}_{\mathfrak q})_{tam}^{\displaystyle n}$ for $\mathfrak{q}=1,2$ are $d \times m$ matrices whose $u\ell$-th elements are 
\begin{align}
&\sigone:= \sum_{\ell_1=1}^{m}\int_{\kappa_{n}(s)}^{s}{\sigxsigtame}dw_r^{i\ell_1} \notag
\\
&\qquad\qquad
+\frac{1}{N}\sum_{\ell_1=1}^{m}\sum_{k=1}^{N}\int_{\kappa_{n}(s)}^{s}{\sigmusigtame}dw_r^{k\ell_1}, \label{eq:sighat1}
\\
(\widehat\sigma_2^{u\ell}&)_{tam}^{\displaystyle{n}}\big(\kappa_n(s),s,x_{\kappa_n(s)}^{i,N,n},\mu^{x,N,n}_{\kappa_n(s)}\big) \notag
\\
&
:=\sum_{k=1}^N\int_{\kappa_n(s)}^{s}\int_Z\big\{(\widehat\sigma^{u\ell})_{tam}^{\displaystyle n}\big( x_{{\kappa_n(s)}}^{i,N,n}+1_{\{k=i\}}\Gamtamek,\tilmukappatame\big) \notag
\\
&\qquad\qquad\qquad\qquad\qquad\qquad\qquad
-\sigtame\big\}  n_p^k(d r,dz) \label{eq:sighat2}
\end{align}
respectively, and $(\widehat{\gamma}_{\mathfrak{q}})_{tam}^{\displaystyle n}$ for $\mathfrak{q}=1,2$ are $d\times 1$ vectors whose $u$-th components are 
\begin{align}
&\gamone:= \sum_{\ell_1=1}^{m}\int_{\kappa_{n}(s)}^{s}{\gamxsigtame}dw_r^{i\ell_1} \notag
\\
&\qquad\qquad
+\frac{1}{N}\sum_{\ell_1=1}^{m}\sum_{k=1}^{N}\int_{\kappa_{n}(s)}^{s}{\gammusigtame}dw_r^{k\ell_1}, \label{eq:gamhat1}
\\
(\widehat\gamma_2^u&)_{tam}^{\displaystyle{n}}\big(\kappa_n(s),s,x_{\kappa_n(s)}^{i,N,n},\mu^{x,N,n}_{\kappa_n(s)},\bar z\big)
\notag
\\
&
:=
\sum_{k=1}^N  \int_{\kappa_n(s)}^{s}\int_Z\big\{(\widehat\gamma^{u})_{tam}^{\displaystyle n}\big( x_{{\kappa_n(s)}}^{i,N,n}
+1_{\{k=i\}}\Gamtamek,\tilmukappatame,\bar z\big) \notag
\\
&\qquad\qquad\qquad\qquad\qquad\qquad\qquad
-\gamtame\big\} n_p^k(dr,dz)
\label{eq:gamhat2}
\end{align}
respectively, for any  $s\in[ 0, T]$, $u\in\{1,\ldots,d\}$, $\ell\in\{1,\ldots,m\}$ and   $\bar z\in Z$ where 
\begin{align} \label{eq:emperical}
	\mu^{x,N,n}_{\kappa_n(s)}:=\frac{1}{N}\sum_{j=1}^N \delta_{x_{\kappa_n(s)}^{j,N,n}},\quad \tilmukappatame:= \empiricaltame. 
\end{align}
\begin{rem}
		We do not restrict ourselves to any specific form of taming for the Milstein-type scheme in \eqref{eq:scm}. Instead, the scheme is defined through  sequences satisfying certain conditions, \textit{e.g.}, Assumptions \textnormal{\ref{asump:super}} to \textnormal{\ref{asump:sig:gam:tame}} and \textnormal{\ref{asum:convergence}} to \textnormal{\ref{asum:convergence2}} mentioned below.  
		Specifically, this includes the case when the coefficients are {Lipschitz continuous, \textit{i.e.}, of linear growth in both state and measure variables}. 
		It is worth mentioning that the tamed Milstein scheme for SDEs with L\'evy noise developed in \cite{Kumar2017, Kumar2021} can be considered as a special case of the Milstein-type scheme considered in this article when the law $\mu_t$ is known and the jump coefficient is linear. 
		Similarly, our scheme in \eqref{eq:scm} becomes the classical Milstein scheme for SDEs driven by L\'evy noise, see \cite{Platen2010}, when all the coefficients are of linear growth and the law $\mu_t$ is given. 
	\end{rem}			

Now, we present the convergence rate of the scheme  \eqref{eq:scm} for \eqref{eq:int} in Theorem \ref{thm:mr} below under the following set of assumptions. 
	Fix $\eta>0$.

	\begin{asump} \label{asump:super}
	For some $\bar p> 4$, there exists   a constant $C>0$  such that 
	\begin{align*}
		&2|x|^{{\bar p}-2}x(\widehat{b})_{tam}^{\displaystyle{n}}(x,\mu)
		+(\bar p-1)|x|^{{\bar p}-2}\big|(\widehat{\sigma})_{tam}^{\displaystyle{n}}(x,\mu)\big|^2 +2(\bar p-1)\int_Z\big| (\widehat{\gamma})_{tam}^{\displaystyle{n}}(x,\mu,z) \big|^2 \notag
		\\
		&\qquad  
		\times\int_{0}^1 (1-\theta)  \big |x+\theta(\widehat{\gamma})_{tam}^{\displaystyle{n}}(x,\mu,z) \big|^{\bar p-2} d\theta \nu(dz) \leq C\big\{1+|x|^{\bar p}+\mathcal{W}_2^{\bar p}(\mu,\delta_0)\big\}
	\end{align*}
	for any  $x\in\mathbb R^d$,  $\mu\in\mathscr{P}_2(\mathbb R^d)$ and $n\in\mathbb N$.
\end{asump}

\begin{asump} \label{asump:tame}
	There exists   a constant $C>0$  such that
	\begin{align*}
		|\bbtamex|&\leq C\min\Big\{n^{\frac{1}{3}}\big(1+|x|+\mathcal{W}_2(\mu,\delta_0)\big),|{b}(x,\mu)|\Big\},
		\\
		|\sigtamex|&\leq C\min\Big\{n^{\frac{1}{6}}\big(1+|x|+\mathcal{W}_2(\mu,\delta_0)\big),|{\sigma}(x,\mu)|\Big\},
		\\
		|\sigxsigtamex|&\leq C\min \Big\{n^{\frac{1}{6}}\big(1+|x|+\mathcal{W}_2(\mu,\delta_0)\big),\big|\sigxsigx\big|\Big\},
		\\
		|\sigmusigtamex|&\leq C\min \Big\{n^{\frac{1}{6}}\big(1+|x|+\mathcal{W}_2(\mu,\delta_0)\big),\big|\sigmusigx\big|\Big\},
		\\
		\int_Z|\gamtamexz|^{\bar p}\nu(dz)&\leq C\min\Big\{n^{\frac{1}{4}}\big(1+|x|+\mathcal{W}(\mu,\delta_0)\big)^{\bar p},\int_Z|{\gamma}(x,\mu,z)|^{\bar p}\nu(dz)\Big\},
		\\
		\int_Z|\gamxsigtamexz|^{\bar p}\nu(dz)&\leq C\min\Big\{n^{\frac{\bar p}{4}}\big(1+|x|+\mathcal{W}(\mu,\delta_0)\big)^{\bar p},\int_Z|\gamxsigxz|^{\bar p}\nu(dz)\Big\},
		\\
		\int_Z|\gammusigtamexz|^{\bar p}\nu(dz)&\leq C\min\Big\{n^{\frac{\bar p}{4}}\big(1+|x|+\mathcal{W}(\mu,\delta_0)\big)^{\bar p},\int_Z|\gammusigxz|^{\bar p}\nu(dz)\Big\}
	\end{align*}
	for any   $x,y\in\mathbb R^d$,  $\mu\in\mathscr{P}_2(\mathbb R^d)$,  $u\in\{1,\ldots,d\}$, $\ell, \ell_1\in\{1,\ldots,m\}$ and $n\in\mathbb N$.
\end{asump}
 \color{black}
 
\begin{asump} \label{asum:measure:derv:bound}
	{There exist   a constant $C>0$ and  $\bar C:Z \mapsto (0, \infty)$ with  $\displaystyle\int_Z| \bar{C}_z|^{\bar p}\nu(dz)<~\infty$ such that for any $x, y\in\mathbb R^d$, $\mu\in\mathscr P_2(\mathbb R^d)$,  $u\in\{1,\ldots,d\}$, $\ell\in\{1,\ldots,m\}$ and $z\in Z$ }
	\begin{align*}
		|\partial_\mu b^{u}(x,\mu,y)|+ |\partial_\mu \sigma^{u\ell}(x,\mu,y)| 	\leq \bar C \mbox{ and }	|\partial_\mu \gamma^{u}(x,\mu,y,z)|\leq 
 \bar C_z.  
	\end{align*}
\end{asump}

\begin{asump}  \label{asump:sig:gam:tame}
	There exists   a constant $C>0$  such that
	\begin{align*}
		|(\widehat{\sigma}^{u\ell})_{tam}^{\displaystyle{n}}(x,\mu)-(\widehat{\sigma}^{u\ell})_{tam}^{\displaystyle{n}}(x,\bar\mu)|&\leq C|\sigma^{u\ell}(x,\mu)-\sigma^{u\ell}(x,\bar\mu)|, 
		\\
		|(\widehat{\gamma}^{u})_{tam}^{\displaystyle{n}}(x,\mu, z)-(\widehat{\gamma}^{u})_{tam}^{\displaystyle{n}}(x,\bar\mu, z)|&\leq C|\gamma^{u}(x,\mu, z)-\gamma^{u}(x,\bar\mu, z)|, \, \mbox{ for $z\in Z$}
	\end{align*}
	for any  $x\in\mathbb R^d$, $\mu,\bar\mu\in\mathscr{P}_2(\mathbb R^d)$, $u\in\{1,\ldots,d\}$, $\ell\in\{1,\ldots,m\}$ and $n\in\mathbb N$.
\end{asump} 

\begin{asump} \label{asum:lip*}
	There exists   a constant $C>0$  such that for any $x, x'\in\mathbb R^d$ and $\mu, \mu'\in\mathscr P_2(\mathbb R^d)$ 
	\begin{align*}
		|b(x,\mu)-b(x',\mu')|
		&\leq C \{(1+|x|+|x'|)^{\eta}|x-x'|+\mathcal W_2(\mu, \mu')\}.
	\end{align*}
\end{asump}

\begin{asump} \label{asum:liningam}
	There exists   a  constant $C>0$ such that for any  $x\in\mathbb R^d$ and $\mu\in\mathscr{P}_2(\mathbb R^d)$
	\begin{align*}
		&\int_Z |\gamma(x,\mu, z)|^{\bar p} \nu(dz)   \leq  C\{1+|x|^{\frac{\eta}{2}+ \bar p}+\mathcal W_2^{\bar p}(\mu, \delta_0)\}. 
	\end{align*}
\end{asump}

\begin{asump} \label{asum:derv:liningam}
	There exists   a  constant $C>0$ such that for any  $x\in\mathbb R^d$ and $\mu\in\mathscr{P}_2(\mathbb R^d)$
	\begin{align*}
		&\int_Z |\partial_x\gamma(x,\mu, z)|^{\bar p} \nu(dz)   \leq  C\{1+|x|^{\frac{\eta}{2}}\}. 
	\end{align*}
\end{asump}

\begin{asump} \label{asum:dlip}
	{There exist   a constant $C>0$ and  $\bar C:Z \mapsto (0, \infty)$ with $\displaystyle\int_Z| \bar{C}_z|^{2}\nu(dz)<~\infty$ such that} for any $x,  x', y, y'\in \mathbb R^d$, $\mu, \mu'\in  \mathscr P_2(\mathbb R^d)$,  $u\in\{1,\ldots,d\}$, $\ell\in\{1,\ldots,m\}$  and $z\in Z$
	\begin{align*}
		|\partial_x b^{u}(x, \mu)-\partial_x b^{u}(x', \mu')|
		\leq &   C\{(1+|x|+|x'|)^{\eta-1}|x-x'|+\mathcal W_2(\mu, \mu')\},
		\\
		|\partial_x \sigma^{u\ell}(x, \mu)-\partial_x \sigma^{u\ell}(x', \mu')|
		\leq &   C\{(1+|x|+|x'|)^{\frac{\eta}{2}-1}|x-x'|+\mathcal W_2(\mu, \mu')\},
		\\
		\int_Z|\partial_x \gamma^{u}(x, \mu, z)-\partial_x \gamma^{u}(x', \mu', z)|^{2}\nu(dz) 
		\leq &  C\{(1+|x|+|x'|)^{\eta-2}|x-x'|^{2}+\mathcal W_2^{2}(\mu, \mu')\},
		\\
		|\partial_\mu b^{u}(x, \mu,y)-\partial_\mu b^{u}(x', \mu',y')|
		\leq &   C\{(1+|x|+|x'|)^{\eta}|x-x'|+|y-y'|+\mathcal W_2(\mu, \mu')\},
		\\
		|\partial_\mu \sigma^{u\ell}(x, \mu,y)-\partial_\mu \sigma^{u\ell}(x', \mu',y')|
		\leq &  C\{(1+|x|+|x'|)^{\frac{\eta}{2}}|x-x'|+|y-y'|+\mathcal W_2(\mu, \mu')\},
		\\
|\partial_\mu \gamma^{u}(x, \mu,y, z)-\partial_\mu \gamma^{u}(x', \mu',y', z)| 
\leq &  \bar C_z\{(1+|x|+|x'|)^{\frac{\eta}{2}}|x-x'|+|y-y'|+\mathcal W_2(\mu, \mu')\}. \color{black}
	\end{align*}	
\end{asump}

\begin{asump}\label{asum:second:measure:derv}
	There exists   a constant $C>0$  such that for any $x, y,y'\in\mathbb R^d$ and $\mu\in\mathscr P_2(\mathbb R^d)$
	\begin{align*}
		|\partial_\mu^2 b^{u}(x,\mu,y,y')|+ |\partial_\mu^2 \sigma^{u\ell}(x,\mu,y,y')| +\int_Z|\partial_\mu^2 \gamma^{u}(x,\mu,y,y',z)|^2\nu(dz)	&\leq C
	\end{align*}
	where $u\in\{1,\ldots,d\}$ and $\ell\in\{1,\ldots,m\}$. 
\end{asump}

\begin{asump} \label{asum:convergence}
	There exists an $\varepsilon\in(0,1)$ and a constant $C>0$ independent of $n, N$   such that
	\begin{align*}
		E\big|&\B -\Btame\big|^2
		+E\big|\sigma(x_{\kappa_n(s)}^{i,N,n},\mu_{\kappa_n(s)}^{x,N,n})-\Sigtame\big|^2
		\\
		&+\int_ZE\big|\Gamz-\Gamztame\big|^2\nu(dz)\leq C n^{-1-\frac{2}{\varepsilon+2}}
	\end{align*}
	for any  $s\in[0,T]$ and $i\in\{1,\ldots,N\}$. 
\end{asump}

\begin{asump} \label{asum:derv:convergence}
	There exists an $\varepsilon\in(0,1)$ and a constant $C>0$ independent of $n, N$   such that
	\begin{align*}
		E\big|\sigxsig -\sigxsigtame\big|^2
		\leq& C n^{-\frac{2}{\varepsilon+2}},
		\\
		E\big|\sigmusig -\sigmusigtame\big|^2
		\leq& C n^{-\frac{2}{\varepsilon+2}},
		\\
		\int_Z E\big|\gamzxsig-\gamzxsigtame\big|^2 \nu(dz) 	\leq& C n^{-\frac{2}{\varepsilon+2}},  	\notag
		\\
		\int_Z E\Big|\gamzmusig-\gamzmusigtame\Big|^2\nu(dz)   \leq& C n^{-\frac{2}{\varepsilon+2}}   	\notag 
	\end{align*}
	for any  $s\in[0,T]$, $u\in\{1,\ldots,d\}$, $\ell,\ell_1\in\{1,\ldots,m\}$ and $i,k\in\{1,\ldots,N\}$.
\end{asump}

\begin{asump} \label{asum:convergence1}
	There exists an $\varepsilon\in(0,1)$ and a constant $C>0$ independent of $n, N$   such that 
	\begin{align*}
		&\int_Z E\big|\sigma( x_{\kappa_n(s)}^{i,N,n},\tilmukappatamei) -(\widehat{\sigma})^{\displaystyle n}_{tam}( x_{\kappa_n(s)}^{i,N,n},\tilmukappatamei)\big|^2\nu(dz) \leq Cn^{-\frac{2}{\varepsilon+2}},\notag
		\\
		&\int_Z E\Big|\sigma(x_{\kappa_n(s)}^{i,N,n}+\Gamztame,\tilmukappatamei) \notag
		\\
		&\qquad-(\widehat{\sigma})^{\displaystyle n}_{tam}(x_{\kappa_n(s)}^{i,N,n}+\Gamztame,\tilmukappatamei)\Big|^2  \nu(dz)\leq Cn^{-\frac{2}{\varepsilon+2}},\notag
		\\
		& \int_Z \int_Z E\big|\gamma\big( x_{\kappa_n(s)}^{i,N,n},\tilmukappatameizbar,z\big)-(\widehat{\gamma})_{tam}^{\displaystyle n}\big( x_{\kappa_n(s)}^{i,N,n},\tilmukappatameizbar,z\big)\big|^2 \nu(d\bar z) \nu(dz)\leq Cn^{-\frac{2}{\varepsilon+2}},\notag
		\\
		&
		\int_Z \int_Z E\Big|\gamma\big(x_{\kappa_n(s)}^{i,N,n}+\Gamtame,\tilmukappatameizbar,z\big) 
		\\
		&\qquad-(\widehat{\gamma})_{tam}^{\displaystyle n}\big(x_{\kappa_n(s)}^{i,N,n}+\Gamtame,\tilmukappatameizbar,z\big)\Big|^2 \nu(d\bar z) \nu(dz)\leq Cn^{-\frac{2}{\varepsilon+2}}\notag
	\end{align*}
	for any  $s\in[0,T]$ and $i\in\{1,\ldots,N\}$.
\end{asump}

\begin{asump} \label{asum:convergence2}
	There exists an $\varepsilon\in(0,1)$ and a constant $C>0$ independent of $n, N$   such that
	\begin{align*}
		& \int_Z E\Big|\sigma(x_{\kappa_n(s)}^{i,N,n},\tilmukappatame) -\sigma( x_{\kappa_n(s)}^{i,N,n},\mu_{\kappa_{n}(s)}^{x,N,n})
		\\
		&
		\quad-(\widehat{\sigma})^{\displaystyle n}_{tam}(x_{\kappa_n(s)}^{i,N,n},\tilmukappatame) +(\widehat{\sigma})^{\displaystyle n}_{tam}( x_{\kappa_n(s)}^{i,N,n},\mu_{\kappa_{n}(s)}^{x,N,n}) \Big|^2 \nu(dz) \leq CN^{-2}n^{-\frac{2}{\varepsilon+2}},\notag
		\\
		& \int_Z\int_Z E\Big|\gamma\big(x_{\kappa_n(s)}^{i,N,n},\tilmukappatamezbar,z\big) -\gamma\big( x_{\kappa_n(s)}^{i,N,n},\mu_{\kappa_{n}(s)}^{x,N,n},z\big) \notag
		\\
		&
		\quad -(\widehat{\gamma})_{tam}^{\displaystyle n}(x_{\kappa_n(s)}^{i,N,n},\tilmukappatamezbar,z) +(\widehat{\gamma})_{tam}^{\displaystyle n}( x_{\kappa_n(s)}^{i,N,n},\mu_{\kappa_{n}(s)}^{x,N,n},z) \Big|^2 \nu(d\bar z) \nu(d z) \leq CN^{-2}n^{-\frac{2}{\varepsilon+2}}\notag
	\end{align*}
	for any  $s\in[0,T]$ and $i,k\in\{1,\ldots,N\}$.
\end{asump}

	The main result of this article is presented in  Theorem \ref{thm:mr}, which ascertains the strong sense convergence rate of the scheme described in Equation \eqref{eq:scm}. 
    One can find the proof at the end of Section \ref{sec:rate}.
	\begin{theorem} \label{thm:mr}
		Let Assumptions \mbox{\normalfont  \ref{asum:ic}} to \mbox{\normalfont  \ref{asum:lin1}} and \mbox{\normalfont  \ref{asum:lin**}} 
   to \mbox{\normalfont  \ref{asum:convergence2}} hold with  $\bar p\geq \max\{2,\eta/8+1\}(\varepsilon+2)(\eta+1)(\eta+2)/{\varepsilon}\color{black}$, $\varepsilon\in(0,1)$.
		  {Then there is a constant $K$, independent of $n$ and $N$, such that }
		\begin{align*}
			\sup_{i\in\{1,\ldots,N\}}\sup_{ t\in[0,T]}E|x_t^{i,N}-x_t^{i,N,n}|^{2}\leq K n^{-1-\frac{2}{\varepsilon+2}}.
		\end{align*}
	\end{theorem}
{By combining the results of Proposition \ref{prop:poc} and Theorem \ref{thm:mr}, we obtain the following corollary.
 \begin{cor}
     Let Assumptions \mbox{\normalfont  \ref{asum:ic}} to \mbox{\normalfont  \ref{asum:convergence2}} be satisfied  with  $\bar p\geq (\eta+1)\max\{4(\varepsilon+2)/\varepsilon,(\eta+2)(\varepsilon+2)/\varepsilon,4(\eta+1)\}$  where $\varepsilon\in(0,1)$. Then, there exists a $K>0$, independent of $n$ and $N$, such that
     \begin{align*}
			\sup_{i\in\{1,\ldots,N\}}\sup_{ t\in[0,T]}E|x_t^{i}-x_t^{i,N,n}|^{2}\leq K
			\begin{cases}
				N^{-\frac{1}{2}}+n^{-1-\frac{2}{\varepsilon+2}}, & \mbox{ if }  d <4,
				\\
				N^{-\frac{1}{2}}\ \log_2( N)+n^{-1-\frac{2}{\varepsilon+2}}\ , & \mbox{ if } d=4,
				\\
				N^{-\frac{2}{d}}+n^{-1-\frac{2}{\varepsilon+2}},  &  \mbox{ if }d>4.
			\end{cases}
		\end{align*}
 \end{cor}
     }
       	
\subsection{Relevant pre-requisites}
For the derivation of the  Milstein-type scheme of the interacting particle system driven by Brownian motion and Poisson random measure in \eqref{eq:int}, one requires It\^o's formula as given in Lemma \ref{lem:ito} below. 
It\^o's formula for continuous McKean--Vlasov SDEs is studied in \cite{Buckdahn2017}, which is extended for  L\'evy noise in \cite{Hao2016, Li2018}. 
For McKean--Vlasov SDEs with jumps in multiple dimensions, It\^o's formula is investigated in \cite{Shen2023}. 
In \cite{pham2023}, the authors provide an It\^o formula for the flow of probability measures of general semi-martingales, which is the most general result available in the literature to date. 
In Appendix \ref{sec:app}, for completeness, we provide a different and simpler proof of It\^o's formula (see Lemma \ref{lem:ito} below) specifically for \eqref{eq:int} using the empirical projection method, which is of independent interest.

\begin{lem}[\bf It\^o's formula for the interacting particle system \eqref{eq:int}] \label{lem:ito}
For any real-valued $F\in \mathcal C^2(\mathbb R^d\times\mathscr{P}_2(\mathbb R^d))$
\begin{align*} 
	F(&x_t^{i,N},\,\mu_t^{x,N})=F(x_0^{i,N},\mu_0^{x,N})+\int_{0}^{t}\mathfrak{D}_x^bF(x_s^{i,N},\mu_s^{x,N})ds
	-\int_{0}^{t}\int_Z \mathfrak D_{x}^{\gamma(z)} F(x_s^{i,N},\mu_s^{x,N}) \nu(dz)ds	\notag
	\\
	&
	+\frac{1}{N}\sum_{k=1}^N\int_{0}^{t}\mathfrak{D}_\mu^b F(x_s^{i,N},\mu_s^{x,N},x_s^{k,N})ds
	-\frac{1}{N}\sum_{k=1}^{N}\int_{0}^{t}\int_Z
	\mathfrak D_{\mu}^{\gamma(z)} F(x_s^{i,N},\mu_s^{x,N},x_s^{k,N})\nu(dz)ds
	\\
	&+\sum_{\ell=1}^m\int_{0}^{t}\mathfrak D_x^{\sigma^\ell}F(x_s^{i,N},\mu_s^{x,N})dw_s^{i\ell} \notag
	+\frac{1}{N}\sum_{\ell=1}^m\sum_{k=1}^N\int_{0}^{t}\mathfrak D_\mu^{\sigma^\ell}F(x_s^{i,N},\mu_s^{x,N},x_s^{k,N})dw_s^{k\ell} \notag
	\\
	&+\frac{1}{2}\int_{0}^{t} \mathfrak D_{xx}^{\sigma}F(x_s^{i,N},\mu_s^{x,N})ds 
	+\frac{1}{N}\int_{0}^{t} \mathfrak D_{x\mu}^{\sigma} F(x_s^{i,N},\mu_s^{x,N})ds   \notag
	\\
	&+\frac{1}{2N}\sum_{k=1}^N\int_{0}^{t} \mathfrak D_{y\mu}^{\sigma}F(x_s^{i,N},\mu_s^{x,N},x_s^{k,N})ds
	+\frac{1}{2N^2}\sum_{k=1}^N\int_{0}^{t} \mathfrak D_{\mu\mu}^{\sigma}F(x_s^{i,N},\mu_s^{x,N},x_s^{k,N})ds \notag
	\\
	&+ \sum_{k=1}^{N}\int_{0}^{t}\int_Z\Big\{F\big(x_{s-}^{i,N}+1_{\{k=i\}}\gamma(x_s^{k,N},\mu_{s}^{x,N},z),\mu_{s-}^{x+\gamma(k,z),N}\big)
	-F(x_{s-}^{i,N},\mu_{s-}^{x,N})
	\Big\}n_p^k(ds,dz) 
\end{align*}
almost surely for  $t \in [0,T]$ and $i\in\{1,\ldots,N\}$ where $\mu_{s-}^{x+\gamma(k,z),N}=\displaystyle \frac{1}{N}\sum_{j=1}^{N}\delta_{x_{s-}^{j,N}+1_{\{j=k\}}\gamma(x_s^{j,N}, \,\mu_{s}^{x,N}, \,z)}$.
\end{lem}
 \color{black}
 We also require the below mentioned elementary result. 
\begin{lem} \label{lem:mvt}
	For any $x,y\in\mathbb R^d$ and  $p>4$,
\begin{align*}
	 |x|^p-|y|^p-p |y|^{p-2}y (x-y) 
	\leq p(p-1) |x-y|^2  \int_0^1(1-\theta) \big|y+\theta(x-y)\big|^{p-2} d\theta.
\end{align*}
\end{lem}

	The proof of the following lemma is obtained by a similar argument as for \cite[Corollary 4.2]{Kumar2021}.
	\begin{lem} \label{lem:MVT}
		Let $f:\mathbb R^d\times \mathscr P_2(\mathbb R^d)\to\mathbb R$ be a function having  measure derivative $\partial_\mu f:\mathbb R^d\times \mathscr P_2(\mathbb R^d)\times\mathbb R^d\to\mathbb R^d$.
		Then, for any $z,x^j,y^j\in\mathbb R^d$
		\begin{align*}
			f\Big(z,\frac{1}{N}\sum_{j=1}^{N} \delta_{x^{j}}\Big) \hspace{-0.04cm}- \hspace{-0.04cm}f\Big(z,\frac{1}{N}\sum_{j=1}^{N} \delta_{y^{j}}\Big)
			\hspace{-0.03cm}=\hspace{-0.03cm} \frac{1}{N}\int_0^1\sum_{\widehat j=1}^{N}\partial_{\mu}f\Big(z,\frac{1}{N}\sum_{j=1}^{N} \delta_{y^j+\theta(x^j-y^j)},y^{\widehat j}+\theta(x^{\widehat j}\hspace{-0.02cm}-\hspace{-0.02cm}y^{\widehat j})\Big)(x^{\widehat j}\hspace{-0.02cm}-\hspace{-0.02cm}y^{\widehat j}) d\theta.
		\end{align*}
	\end{lem}
	
	The proof of the following lemma is given in \cite{Mikulevicius2012}. 

\begin{lem} \label{lem:martingale_ineq}
	Let $p\geq 2$. 
	For a real-valued $\mathscr{P}\otimes\mathscr{Z}$-measurable function $g$, there is a constant $K$, depending only on $p$, such that 
	\begin{align*}
		E \displaystyle \sup_{t\in[0,T]}\Big|\int_{0}^{t}&\int_Z g_s(z) \tilde n_p(ds,dz)\Big|^{p}
		\leq  K E\int_{0}^{T}\int_Z |g_t(z)|^p \nu(dz) dt<\infty.
	\end{align*}
\end{lem}
In this section, we list some of observations that can be made from the above mentioned assumptions in the form of the following remarks. These remarks  are used later in the proofs of the lemmas and the main result. 

\begin{rem}
Due to Assumption \textnormal{\ref{asum:lin1}}, $\displaystyle \int_Z |\gamma(0, \delta_0, z)|^2 \nu(dz)  \leq K$. 
\end{rem}
\begin{rem}\label{rem:poly:sig:gam}
For any $x, x' \in \mathbb{R}^d$ and $\mu, \mu'\in\mathscr{P}_2(\mathbb R^d)$, Assumptions \textnormal{\ref{asum:lip}}  and \textnormal{\ref{asum:lip*}} yield
\begin{align*}
|\sigma(x, \mu)-\sigma(x', \mu')|^2+ \int_Z |\gamma(x,\mu,z)-\gamma(x',\mu',z)|^2 \nu(dz)\leq K \{(1+|x|+|x'|)^{\eta}|x-x'|^2+\mathcal W_2^2(\mu, \mu')\}. 
\end{align*}  
\end{rem}

\begin{rem} \label{rem:super_linear}
By Assumption \textnormal{\ref{asum:lip*}} and Remark \textnormal{\ref{rem:poly:sig:gam}},  for any $x\in\mathbb R^d$ and $\mu\in\mathscr P_2(\mathbb R^d)$
\begin{align*}
|b(x,\mu)|& \leq K\{(1+|x|)^{\eta+1}+\mathcal W_2(\mu, \delta_0)\},
\\
|\sigma(x,\mu)|^2  + \int_Z |\gamma(x,\mu,z) |^2 \nu(dz) &  \leq K \{(1+|x|)^{\eta+2}+\mathcal W_2^2(\mu, \delta_0)\},
\\
 |\partial_x b^{u}(x,\mu)|+|\partial_x \sigma^{u\ell}(x,\mu)|^2+\int_Z|\partial_x \gamma^{u}(x,\mu,z)|^{2}\nu(dz)& \leq K(1+|x|)^{\eta}
\end{align*}
where $u \in \{1, \ldots, d\}$ and  $\ell\in\{1, \ldots, m\}$. 
Similarly,   for any $x, y\in\mathbb{R}^d$, $\mu\in\mathscr{P}_2(\mathbb{R}^d)$, $u \in \{1, \ldots, d\}$ and  $\ell\in\{1, \ldots, m\}$, by Assumption~\textnormal{\ref{asum:dlip}}
\begin{align*}
|\partial_{x}^2 b^{u}(x,\mu)| &\leq K (1+|x|)^{\eta-1},
\\
|\partial_{x}^2 \sigma^{u\ell}(x,\mu)|^2+ \int_Z|\partial_{x}^2\gamma^{u}(x,\mu,z)|^{2}\nu(dz) & \leq K(1+|x|)^{\eta-2},
  \\
 |\partial_{x} \partial_{\mu} b^{u}(x,\mu,y)|+|\partial_{x} \partial_{\mu} \sigma^{u\ell}(x,\mu,y)|^2+ \int_Z|\partial_{x} \partial_{\mu}\gamma^{u}(x,\mu,y, z)|^{2}\nu(dz)& \leq K(1+|x|)^{\eta},
\\
			|\partial_{y} \partial_{\mu} b^{u}(x,\mu,y)|
			+|\partial_{y} \partial_{\mu} \sigma^{u\ell}(x,\mu,y)|
		+\int_Z|\partial_{y} \partial_{\mu} \gamma^{u}(x,\mu,y,z)|^{2}\nu(dz)& \leq K.
		\end{align*}
\end{rem}
We use the following estimate for any $x^j,y^j\in\mathbb R^d$  which can be found in \cite[Section 2]{Reis2019}
\begin{align}
	\mathcal{W}_2^{2}\Big(\frac{1}{N}\sum_{j=1}^N\delta_{x^{j}},\frac{1}{N}\sum_{j=1}^N\delta_{y^{j}}\Big)
	&\leq \frac{1}{N}\sum_{j=1}^N |x^{j}-y^{j}|^2. \label{eq:empirical_dif}
\end{align}
	\section{Moment bound of the  Milstein-type scheme} 
	In this section, we establish the moment bound  (in $\mathcal{L}^p$-sense)   of the tamed Milstein scheme (see Lemma \ref{lem:scm:mb} below).
	We first prove the following auxiliary results.

	\begin{lem} \label{lem:sig_gam}
	If Assumptions     \mbox{\normalfont  \ref{asump:tame}} to \mbox{\normalfont  \ref{asump:sig:gam:tame}} hold, then  for any $2 \leq p_0\leq \bar p$, there exists a constant $K>0$  independent of $N$ and $n$ such that for any $s\in[0,T]$,  $u\in\{1,\ldots,d\}$  and $i\in\{1,\ldots,N\}$
	\begin{align*} 
		&\sum_{\mathfrak{q=1}}^{2} E_{\kappa_{n}(s)} \big|\sigq\big|^{p_0} \leq
		Kn^{\frac{2p_0-9}{12}} \big(1+|x_{\kappa_n(s)}^{i,N,n}|^{p_0}+\frac{1}{N} \sum_{k=1}^{N}|x_{\kappa_n(s)}^{k,N,n}|^{p_0}\big),
				\\
				& \int_ZE_{\kappa_{n}(s)}\big|\gamone\big|^{p_0}\nu(d\bar z) 
				\leq Kn^{-\frac{p_0}{4}} \big(1+|x_{\kappa_n(s)}^{i,N,n}|^{p_0}+\frac{1}{N} \sum_{k=1}^{N}|x_{\kappa_n(s)}^{k,N,n}|^{p_0}\big) ,
				\\
				 &\int_Z \Big(E_{\kappa_{n}(s)} \big|\gamtwo\big|^{p_0}\Big)^{p_1} \nu(d\bar z)
				\leq Kn^{-p_1+\frac{1}{2}} \big(1+|x_{\kappa_n(s)}^{i,N,n}|^{p_0p_1}+\frac{1}{N} \sum_{k=1}^{N}|x_{\kappa_n(s)}^{k,N,n}|^{p_0p_1}\big)
	\end{align*}
	 where the last inequality  holds for some   $p_1\geq 1$  with  $p_0p_1\leq \bar p$.
\end{lem}
\begin{proof}
Let us recall the definition of $(\widehat{\sigma}_1)_{tam}^{\displaystyle n}$ from Equation \eqref{eq:sighat1} and write
\begin{align*}
E_{\kappa_n(s)}\big|&\sigone\big|^{p_0}\leq K  E_{\kappa_n(s)}\Big|\sum_{\ell_1=1}^{m}\int_{\kappa_{n}(s)}^{s}{\sigxsigtame}dw_r^{i\ell_1}\Big|^{p_0} \notag
\\
&\qquad
+\frac{K}{N}\sum_{k=1}^{N}E_{\kappa_n(s)} \Big|\sum_{\ell_1=1}^{m}\int_{\kappa_{n}(s)}^{s}{\sigmusigtame}dw_r^{k\ell_1}\Big|^{p_0}
\end{align*}
which on applying    Assumption  \ref{asump:tame} and Equation \eqref{eq:empirical_dif} yields 
\begin{align} 
		E_{\kappa_n(s)}\big| & \sigone\big|^{p_0} \leq K n^{-\frac{p_0}{2}} \sum_{\ell_1=1}^{m} \big|\sigxsigtame\big|^{p_0}  \notag
		\\
		& \qquad \qquad + \frac{K}{N} n^{-\frac{p_0}{2}} \sum_{k=1}^{N} \sum_{\ell_1=1}^{m}\big|\sigmusigtame\big|^{p_0} \notag
		\\
		& \leq K n^{-\frac{p_0}{3}} \big(1+|x_{\kappa_n(s)}^{i,N,n}|^{p_0}+\frac{1}{N} \sum_{k=1}^{N}|x_{\kappa_n(s)}^{k,N,n}|^{p_0}+\mathcal{W}_2^{p_0}(\mu_{\kappa_n(s)}^{x,N,n},\delta_0)\big)
 \notag
 \\
 &\leq Kn^{-\frac{p_0}{3}} \big(1+|x_{\kappa_n(s)}^{i,N,n}|^{p_0}+\frac{1}{N} \sum_{k=1}^{N}|x_{\kappa_n(s)}^{k,N,n}|^{p_0}\big) \label{eq:sig1}
	\end{align}
for any $s\in[0,T]$, $u\in\{1,\ldots,d\}$, $\ell\in\{1,\ldots,m\}$ and $i\in\{1,\ldots,N\}$.
For the term involving $(\widehat{\sigma}_2)_{tam}^{\displaystyle n}$, {from  Equation \eqref{eq:sighat2}} 
\begin{align} \label{eq:sig21}
	&\sigtwo \notag
	\\
	&= \sum_{k=1}^{N}\int_{\kappa_n(s)}^{s} \int_Z (\widehat\sigma^{u\ell})_{tam}^{\displaystyle n}\big({x}_{\kappa_{n}(s)}^{i,N,n}+1_{\{k=i\}}\Gamtamek,\tilmukappatame\big)  n^k_p(dr, dz) \notag
	\\
	&\qquad-\sum_{k=1}^{N} \int_{\kappa_n(s)}^{s} \int_Z(\widehat\sigma^{u\ell})_{tam}^{\displaystyle n}\big({x}_{\kappa_{n}(s)}^{i,N,n},\mu_{\kappa_{n}(s)}^{x,N,n}\big) n^k_p(dr, dz) \notag
	\\
	& =\int_{\kappa_n(s)}^{s} \int_Z \Big\{(\widehat\sigma^{u\ell})_{tam}^{\displaystyle n}\big({x}_{\kappa_{n}(s)}^{i,N,n}+\Gamtamei,\tilmukappatamei\big) \notag
	\\
	&
	\qquad\qquad\qquad\qquad\qquad\qquad- (\widehat\sigma^{u\ell})_{tam}^{\displaystyle n}\big({x}_{\kappa_{n}(s)}^{i,N,n},\tilmukappatamei\big)\Big\} n^i_p(dr,dz) \notag
	\\
	&\qquad+\int_{\kappa_n(s)}^{s} \int_Z  (\widehat\sigma^{u\ell})_{tam}^{\displaystyle n}\big({x}_{\kappa_{n}(s)}^{i,N,n},\tilmukappatamei\big)n^i_p(dr,dz)  \notag
	\\
	&\qquad+\sum_{k\neq i}^{N}\int_{\kappa_n(s)}^{s} \int_Z (\widehat\sigma^{u\ell})_{tam}^{\displaystyle n}\big({x}_{\kappa_{n}(s)}^{i,N,n},\tilmukappatame\big) n^k_p(dr,dz) \notag
	\\
	&\qquad 
	-\sum_{k=1}^{N}\int_{\kappa_n(s)}^{s} \int_Z(\widehat\sigma^{u\ell})_{tam}^{\displaystyle n}\big({x}_{\kappa_{n}(s)}^{i,N,n},\mu_{\kappa_{n}(s)}^{x,N,n}\big) n^k_p(dr,dz) \notag
	\\
	&=\int_{\kappa_n(s)}^{s} \int_Z \Big\{(\widehat\sigma^{u\ell})_{tam}^{\displaystyle n}\big({x}_{\kappa_{n}(s)}^{i,N,n}+\Gamtamei,\tilmukappatamei\big) \notag
	\\
	&\qquad\qquad\qquad\qquad\qquad\qquad-(\widehat\sigma^{u\ell})_{tam}^{\displaystyle n}\big({x}_{\kappa_{n}(s)}^{i,N,n},\tilmukappatamei\big)\Big\}n_p^i(dr,dz) \notag
	\\
	&\quad+\sum_{k=1}^{N}\int_{\kappa_n(s)}^{s} \int_Z  \Big\{(\widehat\sigma^{u\ell})_{tam}^{\displaystyle n}\big({x}_{\kappa_{n}(s)}^{i,N,n},\tilmukappatame\big)
	-(\widehat\sigma^{u\ell})_{tam}^{\displaystyle n}\big({x}_{\kappa_{n}(s)}^{i,N,n},\mu_{\kappa_{n}(s)}^{x,N,n}\big)\Big\} n_p^k(dr,dz) \notag
\\
& = \int_{\kappa_n(s)}^{s} \int_Z  \Big\{(\widehat\sigma^{u\ell})_{tam}^{\displaystyle n}\big({x}_{\kappa_{n}(s)}^{i,N,n}+\Gamtamei,\tilmukappatamei\big) \notag
\\
&\qquad\qquad\qquad\qquad\qquad\qquad-(\widehat\sigma^{u\ell})_{tam}^{\displaystyle n}\big({x}_{\kappa_{n}(s)}^{i,N,n},\tilmukappatamei\big)\Big\}\tilde{ n}_p^i(dr,dz)
\notag
\\
& \quad+ \int_{\kappa_n(s)}^{s} \int_Z  \Big\{(\widehat\sigma^{u\ell})_{tam}^{\displaystyle n}\big({x}_{\kappa_{n}(s)}^{i,N,n}+\Gamtamei,\tilmukappatamei\big) \notag
\\
&\qquad\qquad\qquad\qquad\qquad\qquad-(\widehat\sigma^{u\ell})_{tam}^{\displaystyle n}\big({x}_{\kappa_{n}(s)}^{i,N,n},\tilmukappatamei\big)\Big\}\nu(dz)dr \notag
\\
&\quad +\sum_{k=1}^{N}\int_{\kappa_n(s)}^{s}\int_Z\Big\{(\widehat\sigma^{u\ell})_{tam}^{\displaystyle n}\big({x}_{\kappa_{n}(s)}^{i,N,n},\tilmukappatame\big)
-(\widehat\sigma^{u\ell})_{tam}^{\displaystyle n}\big({x}_{\kappa_{n}(s)}^{i,N,n},\mu_{\kappa_{n}(s)}^{x,N,n}\big)\Big\} \tilde{ n}_p^k(dr,dz)\notag
\\
& \quad+\sum_{k=1}^{N}\int_{\kappa_n(s)}^{s}\int_Z\Big\{(\widehat\sigma^{u\ell})_{tam}^{\displaystyle n}\big({x}_{\kappa_{n}(s)}^{i,N,n},\tilmukappatame\big)
-(\widehat\sigma^{u\ell})_{tam}^{\displaystyle n}\big({x}_{\kappa_{n}(s)}^{i,N,n},\mu_{\kappa_{n}(s)}^{x,N,n}\big)\Big\} \nu(dz) dr
\end{align}
for any $s\in[0,T]$, $u\in\{1,\ldots,d\}$, $\ell\in\{1,\ldots,m\}$ and $i\in\{1,\ldots,N\}$.
To estimate the $p_0$-th moment of the above equation, one  uses Lemma \ref{lem:martingale_ineq} with $p_0\geq 2$ and H\"older's inequality  to obtain
\begin{align} \label{eq:sig22}
	&E_{\kappa_n(s)}\big|\sigtwo\big|^{p_0} \notag
	\\ 
	& \leq K
E_{\kappa_n(s)}\int_{\kappa_n(s)}^{s} \int_Z  \big|(\widehat\sigma^{u\ell})_{tam}^{\displaystyle n}\big({x}_{\kappa_{n}(s)}^{i,N,n}+\Gamtamei,\tilmukappatamei\big) \notag
	\\
	&\qquad\qquad\qquad\qquad\qquad\qquad-(\widehat\sigma^{u\ell})_{tam}^{\displaystyle n}\big({x}_{\kappa_{n}(s)}^{i,N,n},\tilmukappatamei\big)\big|^{p_0} \nu(dz)dr \notag
	\\
	&\,\,\,\,+KN^{p_0-1}\sum_{k=1}^{N}E_{\kappa_n(s)} \hspace{-0.07cm}\int_{\kappa_n(s)}^{s}\hspace{-0.07cm}\int_Z\big|(\widehat\sigma^{u\ell})_{tam}^{\displaystyle n}\big({x}_{\kappa_{n}(s)}^{i,N,n},\tilmukappatame\big)
\hspace{-0.07cm}-\hspace{-0.07cm}(\widehat\sigma^{u\ell})_{tam}^{\displaystyle n}\big({x}_{\kappa_{n}(s)}^{i,N,n},\mu_{\kappa_{n}(s)}^{x,N,n}\big)\big|^{p_0}\nu(dz)dr \notag
\\
&\leq  K n^{-1}\int_Z  \big|(\widehat\sigma^{u\ell})_{tam}^{\displaystyle n}\big({x}_{\kappa_{n}(s)}^{i,N,n}+\Gamtamei,\tilmukappatamei\big) \notag
	\\
	&\qquad\qquad\qquad\qquad\qquad\qquad-(\widehat\sigma^{u\ell})_{tam}^{\displaystyle n}\big({x}_{\kappa_{n}(s)}^{i,N,n},\tilmukappatamei\big)\big|^{p_0} \nu(dz) \notag
	\\
	&\,\,\,\, +KN^{p_0-1}n^{-1}\sum_{k=1}^{N} \int_Z\big|(\widehat\sigma^{u\ell})_{tam}^{\displaystyle n}\big({x}_{\kappa_{n}(s)}^{i,N,n},\tilmukappatame\big)-(\widehat\sigma^{u\ell})_{tam}^{\displaystyle n}\big({x}_{\kappa_{n}(s)}^{i,N,n},\mu_{\kappa_{n}(s)}^{x,N,n}\big)\big|^{p_0}\nu(dz) \notag
	\\
	& =:T_1+T_2
\end{align}
for any $s\in[0,T]$, $u\in\{1,\ldots,d\}$, $\ell\in\{1,\ldots,m\}$ and $i\in\{1,\ldots,N\}$.
To estimate $T_1$, one applies Assumption \ref{asump:tame},  Equations \eqref{eq:emperical} and \eqref{eq:empirical_dif} as follows
\begin{align} \label{eq:t1}
	T_1&:=   K n^{-1}  \int_Z  \big|(\widehat\sigma^{u\ell})_{tam}^{\displaystyle n}\big({x}_{\kappa_{n}(s)}^{i,N,n}+\Gamtamei,\tilmukappatamei\big) \notag
	\\
	&\qquad\qquad\qquad\qquad\qquad\qquad-(\widehat\sigma^{u\ell})_{tam}^{\displaystyle n}\big({x}_{\kappa_{n}(s)}^{i,N,n},\tilmukappatamei\big)\big|^{p_0} \nu(dz) \notag
	\\
	\leq& Kn^{-1} \int_Z  n^{\frac{p_0}{6}} \Big\{1+|x_{\kappa_n(s)}^{i,N,n}|+|\Gamztame| 
	+\mathcal{W}_2\big(\tilmukappatamei,\delta_0\big)\Big\}^{p_0} \nu(dz) \notag
   \\
   \leq& Kn^{\frac{p_0-6}{6}} \Big(1+|x_{\kappa_n(s)}^{i,N,n}|^{p_0}+\int_Z\big|\Gamztame\big|^{p_0}\nu(dz)+\frac{1}{N}\sum_{j=1}^N|x_{\kappa_n(s)}^{j,N,n}|^{p_0}\Big) \notag
   \\
   \leq& Kn^{\frac{2p_0-9}{12}}\big(1+|x_{\kappa_n(s)}^{i,N,n}|^{p_0}+\mathcal{W}_2^{p_0}(\mu_{\kappa_n(s)}^{x,N,n},\delta_0)+\frac{1}{N} \sum_{j=1}^{N}|x_{\kappa_n(s)}^{k,N,n}|^{p_0}\big)  \notag
   \\
   \leq&
   Kn^{\frac{2p_0-9}{12}} \big(1+|x_{\kappa_n(s)}^{i,N,n}|^{p_0}+\frac{1}{N} \sum_{j=1}^{N}|x_{\kappa_n(s)}^{k,N,n}|^{p_0}\big)
\end{align}
for any $s\in[0,T]$, $u\in\{1,\ldots,d\}$, $\ell\in\{1,\ldots,m\}$ and $i\in\{1,\ldots,N\}$.
For $T_2$, one uses  Assumption \ref{asump:sig:gam:tame}, Lemma  \ref{lem:MVT},  Assumptions  \ref{asum:measure:derv:bound}, \ref{asump:tame} and Equation  \eqref{eq:empirical_dif} as follows
\begin{align} \label{eq:t2}
	T_2 &:= KN^{p_0-1}n^{-1}\sum_{k=1}^{N}\int_Z\big|(\widehat\sigma^{u\ell})_{tam}^{\displaystyle n}\big({x}_{\kappa_{n}(s)}^{i,N,n},\tilmukappatame\big)-(\widehat\sigma^{u\ell})_{tam}^{\displaystyle n}\big({x}_{\kappa_{n}(s)}^{i,N,n},\mu_{\kappa_{n}(s)}^{x,N,n}\big)\big|^{p_0}\nu(dz) \notag
	\\
	\leq & KN^{p_0-1}n^{-1}\sum_{k=1}^{N}\int_Z\big|\sigma^{u\ell}\big({x}_{\kappa_{n}(s)}^{i,N,n},\tilmukappatame\big)-\sigma^{u\ell}\big({x}_{\kappa_{n}(s)}^{i,N,n},\mu_{\kappa_{n}(s)}^{x,N,n}\big)\big|^{p_0}\nu(dz) \notag
	\\
	\leq & KN^{p_0-1}n^{-1}\sum_{k=1}^{N}\int_Z{\int_0^1}\Big|\frac{1}{N}\sum_{\widehat j=1}^{N}\partial_\mu \sigma^{u\ell}\Big( x_{{\kappa_n(s)}}^{i,N,n}, \empiricalthetaonetame, \notag
	\\
	&\qquad\qquad\qquad   x_{\kappa_n(s)}^{\widehat j,N,n}+1_{\{\widehat j=k\}}\theta\Gamtamejhat \Big)  1_{\{\widehat j=k\}}\Gamtamejhat\Big|^{p_0}d\theta\nu(dz) \notag
	\\
	\leq& \frac{K}{N}n^{-1} \sum_{k=1}^{N}\int_Z\big|\Gamtamek\big|^{p_0}\nu(dz)
	\leq Kn^{-\frac{3}{4}} \big(1+\frac{1}{N} \sum_{k=1}^{N}|x_{\kappa_n(s)}^{k,N,n}|^{p_0}+\mathcal{W}_2^{p_0}(\mu_{\kappa_n(s)}^{x,N,n},\delta_0)\big) \notag
	\\
	& \qquad\qquad\qquad  \qquad\qquad\qquad  
	\leq
	Kn^{-\frac{3}{4}} \big(1+\frac{1}{N} \sum_{k=1}^{N}|x_{\kappa_n(s)}^{k,N,n}|^{p_0}\big)
\end{align}
for any $s\in[0,T]$, $u\in\{1,\ldots,d\}$, $\ell\in\{1,\ldots,m\}$ and $i\in\{1,\ldots,N\}$.
Then, by   using    Equations \eqref{eq:t1} and \eqref{eq:t2}, one gets
\begin{align} \label{eq:sig2}
	&E_{\kappa_{n}(s)}\big|\sigtwo\big|^{p_0}  	\leq
	Kn^{\frac{2p_0-9}{12}} \big(1+|x_{\kappa_n(s)}^{i,N,n}|^{p_0}+\frac{1}{N} \sum_{k=1}^{N}|x_{\kappa_n(s)}^{k,N,n}|^{p_0}\big) 
\end{align}
for any $s\in[0,T]$, $u\in\{1,\ldots,d\}$, $\ell\in\{1,\ldots,m\}$ and $i\in\{1,\ldots,N\}$. First inequality of the lemma follows by combining  Equations \eqref{eq:sig1} and \eqref{eq:sig2}.
To estimate the term involving $(\widehat{\gamma}_1^{u\ell})_{tam}^{\displaystyle n}$, one uses  Equation \eqref{eq:gamhat1} to obtain
\begin{align*}
	 &E_{\kappa_{n}(s)}\big|\gamone\big|^{p_0}
	\\
	& \leq K   E_{\kappa_{n}(s)}\Big|\sum_{\ell_1=1}^{m}\int_{\kappa_{n}(s)}^{s}{\gamxsigtame}dw_r^{i\ell_1}\Big|^{p_0} \notag
	\\
	&\qquad
	+\frac{K}{N}\sum_{k=1}^{N} E_{\kappa_{n}(s)} \Big|\sum_{\ell_1=1}^{m}\int_{\kappa_{n}(s)}^{s}{\gammusigtame}dw_r^{k\ell_1}\Big|^{p_0} 
\end{align*}
for any $\bar z\in Z$ which on applying    Assumption  \ref{asump:tame} and Equation \eqref{eq:empirical_dif} yield
\begin{align} 
	&\int_ZE_{\kappa_{n}(s)}\big|  \gamone\big|^{p_0} \nu(d\bar z) \notag 
	\\
	&\leq K n^{-\frac{p_0}{2}} \sum_{\ell_1=1}^{m} \int_Z\big|\gamxsigtame\big|^{p_0} \nu(d\bar z)\notag
	\\
	& \qquad + \frac{K}{N} n^{-\frac{p_0}{2}} \sum_{k=1}^{N} \sum_{\ell_1=1}^{m}\int_Z\big|\gammusigtame\big|^{p_0} \nu(d\bar z) \notag
	\\
	& \leq K n^{-\frac{p_0}{4}} \big(1+|x_{\kappa_n(s)}^{i,N,n}|^{p_0}+\frac{1}{N} \sum_{k=1}^{N}|x_{\kappa_n(s)}^{k,N,n}|^{p_0}+\mathcal{W}_2^{p_0}(\mu_{\kappa_n(s)}^{x,N,n},\delta_0)\big)
	\notag
	\\
	& \leq Kn^{-\frac{p_0}{4}} \big(1+|x_{\kappa_n(s)}^{i,N,n}|^{p_0}+\frac{1}{N} \sum_{k=1}^{N}|x_{\kappa_n(s)}^{k,N,n}|^{p_0}\big) \notag
\end{align}
for any $s\in[0,T]$, $u\in\{1,\ldots,d\}$ and $i\in\{1,\ldots,N\}$.
This completes the proof of  the second inequality of the lemma.
For the term involving $(\widehat{\gamma}_2)_{tam}^{\displaystyle n}$,  one  recalls  Equation \eqref{eq:gamhat2} and applies similar arguments used for $(\widehat{\sigma}_2)_{tam}^{\displaystyle n}$ in  Equations \eqref{eq:sig21}, \eqref{eq:sig22} to get
\begin{align*}  
	&E_{\kappa_{n}(s)}\big|\gamtwo\big|^{p_0}\notag
	\\
	&= E_{\kappa_{n}(s)}\Big|\sum_{k=1}^{N}\int_{\kappa_n(s)}^{s}\int_Z (\widehat\gamma^{u})_{tam}^{\displaystyle n}\big({x}_{\kappa_{n}(s)}^{i,N,n}+1_{\{k=i\}}\Gamtamek,\tilmukappatame,\bar z\big) \notag
	\\
	&\qquad\qquad\qquad\qquad\qquad\qquad
	-(\widehat\gamma^{u})_{tam}^{\displaystyle n}\big({x}_{\kappa_{n}(s)}^{i,N,n},\mu_{\kappa_{n}(s)}^{x,N,n},\bar z\big)n_p^k(dr,dz) \Big|^{p_0}   \notag
	\\
	&\leq Kn^{-1} \int_Z \big| (\widehat\gamma^{u})_{tam}^{\displaystyle n}\big({x}_{\kappa_{n}(s)}^{i,N,n}+\Gamztame,\tilmukappai,\bar z\big) \notag
	\\
	&\qquad\qquad\qquad\qquad\qquad\qquad
	-(\widehat\gamma^{u})_{tam}^{\displaystyle n}\big({x}_{\kappa_{n}(s)}^{i,N,n},\tilmukappai,\bar z\big)\big|^{p_0} \nu(dz)  \notag
	\\
	&\quad +N^{p_0-1}n^{-1}\sum_{k=1}^{N}\int_Z\big|(\widehat\gamma^{u})_{tam}^{\displaystyle n}\big({x}_{\kappa_{n}(s)}^{i,N,n},\tilmukappatame,\bar z\big)-(\widehat\gamma^{u})_{tam}^{\displaystyle n}\big({x}_{\kappa_{n}(s)}^{i,N,n},\mu_{\kappa_{n}(s)}^{x,N,n},\bar z\big)\big|^{p_0}\nu(dz) \notag
\end{align*}
for any $\bar z\in Z$ which on applying  H\"older's inequality	yields
\begin{align} \label{eq:gam**}
	&\int_Z \Big(E_{\kappa_{n}(s)}\big|\gamtwo\big|^{p_0} \Big)^{p_1}\nu(d\bar z) \notag
	\\
	&\leq Kn^{-p_1}\int_Z  \int_Z \big| (\widehat\gamma^{u})_{tam}^{\displaystyle n}\big({x}_{\kappa_{n}(s)}^{i,N,n}+\Gamztame,\tilmukappai,\bar z\big) \notag
	\\
	&\qquad\qquad\qquad\qquad\qquad\qquad
	-(\widehat\gamma^{u})_{tam}^{\displaystyle n}\big({x}_{\kappa_{n}(s)}^{i,N,n},\tilmukappai,\bar z\big)\big|^{p_0p_1} \nu(dz)\nu(d\bar z) \notag
	\\
	&\quad +N^{p_0p_1-1}n^{-p_1}\sum_{k=1}^{N}\int_Z\int_Z\big|(\widehat\gamma^{u})_{tam}^{\displaystyle n}\big({x}_{\kappa_{n}(s)}^{i,N,n},\tilmukappatame,\bar z\big) \notag
	\\
	&\qquad\qquad\qquad\qquad\qquad\qquad
 -(\widehat\gamma^{u})_{tam}^{\displaystyle n}\big({x}_{\kappa_{n}(s)}^{i,N,n},\mu_{\kappa_{n}(s)}^{x,N,n},\bar z\big)\big|^{p_0p_1}\nu(dz)\nu(d\bar z)  
	:=\widetilde{T}_1+\widetilde{T}_2
\end{align}
for any $s\in[0,T]$, $u\in\{1,\ldots,d\}$ and $i\in\{1,\ldots,N\}$.
For $\widetilde T_1$, one uses Assumption \ref{asump:tame},  Equations \eqref{eq:emperical} and \eqref{eq:empirical_dif} as follows
\begin{align} \label{eq:til:t1}
	\widetilde T_1:=&   K n^{-p_1}  \int_Z \int_Z  \big|(\widehat\gamma^{u})_{tam}^{\displaystyle n}\big({x}_{\kappa_{n}(s)}^{i,N,n}+\Gamtamei,\tilmukappatamei,\bar z\big) \notag
	\\
	&\qquad\qquad\qquad\qquad\qquad\qquad-(\widehat\gamma^{u})_{tam}^{\displaystyle n}\big({x}_{\kappa_{n}(s)}^{i,N,n},\tilmukappatamei,\bar z\big)\big|^{p_0p_1} \nu(dz)\nu(d\bar z) \notag
	\\
	\leq& Kn^{-p_1} \int_Z  n^{\frac{1}{4}} \Big\{1+|x_{\kappa_n(s)}^{i,N,n}|+|\Gamztame| 
	+\mathcal{W}_2\big(\tilmukappatamei,\delta_0\big)\Big\}^{p_0p_1} \nu(dz) \notag
	\\
	\leq& Kn^{-p_1+\frac{1}{4}} \Big(1+|x_{\kappa_n(s)}^{i,N,n}|^{p_0p_1}+\int_Z\big|\Gamztame\big|^{p_0p_1}\nu(dz)+\frac{1}{N}\sum_{j=1}^N|x_{\kappa_n(s)}^{j,N,n}|^{p_0p_1}\Big) \notag
	\\
	\leq& Kn^{-p_1+\frac{1}{2}}\big(1+|x_{\kappa_n(s)}^{i,N,n}|^{p_0p_1}+\mathcal{W}_2^{p_0p_1}(\mu_{\kappa_n(s)}^{x,N,n},\delta_0)+\frac{1}{N} \sum_{j=1}^{N}|x_{\kappa_n(s)}^{k,N,n}|^{p_0p_1}\big)  \notag
	\\
	\leq&
	Kn^{-p_1+\frac{1}{2}} \big(1+|x_{\kappa_n(s)}^{i,N,n}|^{p_0p_1}+\frac{1}{N} \sum_{j=1}^{N}|x_{\kappa_n(s)}^{k,N,n}|^{p_0p_1}\big)
\end{align}
for any $s\in[0,T]$, $u\in\{1,\ldots,d\}$ and $i\in\{1,\ldots,N\}$.
To estimate $\widetilde T_2$, one applies Assumption \ref{asump:sig:gam:tame},   Lemma  \ref{lem:MVT}, Assumptions \ref{asum:measure:derv:bound}, \ref{asump:tame} and Equation  \eqref{eq:empirical_dif} as follows
\begin{align} \label{eq:til:t2}
	\widetilde T_2 :=& KN^{p_0p_1-1}n^{-p_1}\sum_{k=1}^{N}\int_Z\int_Z\big|(\widehat\gamma^{u})_{tam}^{\displaystyle n}\big({x}_{\kappa_{n}(s)}^{i,N,n},\tilmukappatame,\bar z\big)\notag
	\\
	&\qquad\qquad\qquad\qquad\qquad\qquad
 -(\widehat\gamma^{u})_{tam}^{\displaystyle n}\big({x}_{\kappa_{n}(s)}^{i,N,n},\mu_{\kappa_{n}(s)}^{x,N,n},\bar z\big)\big|^{p_0p_1}\nu(dz)\nu(d\bar z) \notag
	\\
	\leq & KN^{p_0p_1-1}n^{-p_1}\sum_{k=1}^{N}\int_Z\int_Z\big|\gamma^{u}\big({x}_{\kappa_{n}(s)}^{i,N,n},\tilmukappatame,\bar z\big)-\gamma^{u}\big({x}_{\kappa_{n}(s)}^{i,N,n},\mu_{\kappa_{n}(s)}^{x,N,n},\bar z\big)\big|^{p_0p_1}\nu(dz)  \nu(d\bar z)\notag
	\\
	\leq &  KN^{p_0p_1-1}n^{-p_1}\sum_{k=1}^{N}\int_Z\int_Z{\int_0^1}\Big|\frac{1}{N}\sum_{\widehat j=1}^{N}\partial_\mu \gamma^{u}\Big( x_{{\kappa_n(s)}}^{i,N,n}, \empiricalthetaonetame, \notag
	\\
	&\quad  x_{{\kappa_n(s)}}^{\widehat j,N,n}+1_{\{\widehat j=k\}}\Gamtamejhat,\bar z\Big)  1_{\{\widehat j=k\}}\Gamtamejhat\Big|^{p_0p_1}d\theta\nu(dz)\nu(d\bar z) \notag
	\\
	\leq& \frac{K}{N}n^{-p_1} \sum_{k=1}^{N}\int_Z \int_Z{|\bar C_{\bar z}|^{p_0p_1}}   \big|\Gamtamek\big|^{p_0p_1}\nu(dz)\nu(d\bar z) \notag
	\\
 \leq &Kn^{-p_1+\frac{1}{4}} \big(1+\frac{1}{N} \sum_{k=1}^{N}|x_{\kappa_n(s)}^{k,N,n}|^{p_0p_1}+\mathcal{W}_2^{p_0p_1}(\mu_{\kappa_n(s)}^{x,N,n},\delta_0)\big) 
	\leq
	Kn^{-p_1+\frac{1}{4}} \big(1+\frac{1}{N} \sum_{k=1}^{N}|x_{\kappa_n(s)}^{k,N,n}|^{p_0p_1}\big)
\end{align}
for any $s\in[0,T]$, $u\in\{1,\ldots,d\}$ and $i\in\{1,\ldots,N\}$.
Then, by   using    Equations \eqref{eq:til:t1} and \eqref{eq:til:t2} in Equation \eqref{eq:gam**}, one obtains for any $s\in[0,T]$,  $u\in\{1,\ldots,d\}$  and $i\in\{1,\ldots,N\}$ as follows
\begin{align*} 
	&\int_Z \Big(E_{\kappa_{n}(s)} \big|\gamtwo\big|^{p_0}\Big)^{p_1} \nu(d\bar z)
	\\
 &\qquad\leq 
	Kn^{-p_1+\frac{1}{2}} \big(1+|x_{\kappa_n(s)}^{i,N,n}|^{p_0p_1}+\frac{1}{N} \sum_{k=1}^{N}|x_{\kappa_n(s)}^{k,N,n}|^{p_0p_1}\big)
\end{align*}
which completes the proof of the last inequality of the lemma.
\end{proof}

	\begin{lem} \label{lem:dif:scm}
		If Assumptions     \mbox{\normalfont  \ref{asump:tame}} to \mbox{\normalfont  \ref{asump:sig:gam:tame}}   are satisfied, then for any $p_0\geq 2$, $t\in[0,T]$ and $i\in\{1,\ldots,N\}$
			\begin{align*}
				E_{\kappa_{n}(t)}|x_t^{i,N,n}-x_{\kappa_{n}(t)}^{i,N,n}|^{p_0}\leq K
					(n^{-\frac{p_0}{3}}+n^{-\frac{3}{4}})\big(1+|x_{\kappa_n(t)}^{i,N,n}|^{p_0}+\frac{1}{N} \displaystyle\sum_{k=1}^{N}|x_{\kappa_n(t)}^{k,N,n}|^{p_0}\big)
			\end{align*}
almost surely where  $K>0$ is a constant independent of $N$ and $n$.
	\end{lem}
 \color{black}
 
	\begin{proof} 
		First,  by using Equation    \eqref{eq:scm},
  Lemma \ref{lem:martingale_ineq} and H\"older's inequality, one obtains
		\begin{align*} 
			&E_{\kappa_{n}(t)}|x_t^{i,N,n}-x_{\kappa_{n}(t)}^{i,N,n}|^{p_0}			
			\leq Kn^{-{p_0}+1}E_{\kappa_{n}(t)}\int_{\kappa_{n}(t)}^t \big|\Btame \big|^{p_0}ds 
			\\
			&\quad +Kn^{-\frac{p_0}{2}+1}E_{\kappa_{n}(t)}\int_{\kappa_{n}(t)}^{t}\Big|\Sigtame+\sum_{\mathfrak{q}=1}^{2}\Sigq\Big|^{p_0}  ds 
			\\
			&\quad
			+Kn^{-\frac{p_0}{2}+1} E_{\kappa_{n}(t)}\int_{\kappa_{n} \hspace{-0.1cm}(t)}^{t}\hspace{-0.1cm}\int_Z\Big|\Gamtame \hspace{-0.05cm}+ \hspace{-0.05cm}\sum_{\mathfrak{q}=1}^{2}\Gamq \Big|^{p_0} \nu(d\bar z)ds
			\\
			&\quad +KE_{\kappa_{n}(t)}\int_{\kappa_{n}(t)}^{t}\int_Z\Big|\Gamtame
		 +\sum_{\mathfrak{q}=1}^{2}\Gamq \Big|^{p_0}\nu(d\bar z)ds
		\end{align*}
		for any  $t\in[0,T]$ and $i\in\{1,\ldots,N\}$. 
		Also, the application of Assumption  \ref{asump:tame} and   Lemma \ref{lem:sig_gam} yields
		\begin{align*}
		&E_{\kappa_{n}(t)}|x_t^{i,N,n}-x_{\kappa_{n}(t)}^{i,N,n}|^{p_0}	
	\leq K\big(
			n^{-\frac{2p_0}{3}}
			+n^{-\frac{p_0}{3}}
			+n^{-\frac{4p_0+9}{12}}
			+n^{-\frac{3}{4}}
			+n^{-\frac{p_0+4}{4}}
			+n^{-\frac{3}{2}}\big)
			\\
			&\qquad \times 
			\big(1+|x_{\kappa_n(t)}^{i,N,n}|^{p_0}+\frac{1}{N} \sum_{k=1}^{N}|x_{\kappa_n(t)}^{k,N,n}|^{p_0}\big)\leq K(n^{-\frac{p_0}{3}}+n^{-\frac{3}{4}})\big(1+|x_{\kappa_n(t)}^{i,N,n}|^{p_0}+\frac{1}{N} \sum_{k=1}^{N}|x_{\kappa_n(t)}^{k,N,n}|^{p_0}\big)
		\end{align*}
for any  $t\in[0,T]$ and $i\in\{1,\ldots,N\}$.
This completes the proof. 
\end{proof}

As the interacting particle system \eqref{eq:int} is regarded as a $d \times N$ dimensional SDEs, its Milstein-type scheme  \eqref{eq:scm} can also be considered as  $d \times N$ dimensional SDEs. 
By Assumption \ref{asump:tame}, we notice that the coefficients of the Milstein-type scheme \eqref{eq:scm} satisfy linear growth type conditions, although the constants depend on the number of sub-intervals $n$ and of particles $N$. 
As a consequence of this and the finiteness of the discretisation points, one can immediately notice that 
\begin{align*}
			\sup_{i\in\{1,\ldots,N\}}\sup_{ t\in[0,T]}E|x_t^{i,N,n}|^{\bar p}< K
		\end{align*}
  where $K>0$ depends on $n$ and $N$ which explode when $n$ and $N$ tend to infinity. 
  In the following lemma, we establish that $K$ is independent of $n$ and $N$ and thus the Milstein-type scheme \eqref{eq:scm} has bounded moments. 

	\begin{lem} \label{lem:scm:mb}
		If Assumptions  \mbox{\normalfont  \ref{asum:ic}} and \mbox{\normalfont  \ref{asump:super}} to \mbox{\normalfont  \ref{asump:sig:gam:tame}}  are satisfied, then the following holds
		\begin{align*}
			\sup_{i\in\{1,\ldots,N\}}\sup_{ t\in[0,T]}E|x_t^{i,N,n}|^{\bar p}\leq K
		\end{align*}
		where $K>0$ is a constant independent of $N$ and $n$.
	\end{lem}
	\begin{proof}
		Let us recall the Milstein-type scheme   \eqref{eq:scm}  and  use   It\^o's formula (see \cite[Theorem 94]{Situ2006}) to obtain
			\begin{align*} 
			&|x_t^{i,N,n}|^{\bar p}  
   =
   |x_0^{i}|^{\bar p}
			 +{\bar p}\int_{0}^{t}|x_s^{i,N,n}|^{{\bar p}-2}x_s^{i,N,n}\Btame 
			  ds \notag
			\\
			&\quad +{\bar p}\int_{0}^{t}|x_s^{i,N,n}|^{{\bar p}-2}x_s^{i,N,n}\Big\{\Sigtame+ \sum_{\mathfrak q=1}^2\Sigq\Big\}   dw_s^i \notag
			\\
			&\quad +\frac{\bar p(\bar p-1)}{2}\int_{0}^{t}|x_s^{i,N,n}|^{{\bar p}-2}\big| \Sigtame+ \sum_{\mathfrak q=1}^2\Sigq \big|^2 ds \notag
			\\
			&\quad+{\bar p}\hspace{-0.1cm}\int_{0}^{t} \hspace{-0.1cm}\int_Z\hspace{-0.01cm}|x_s^{i,N,n}|^{{\bar p}-2}x_s^{i,N,n}\hspace{-0.07cm}\Big\{\Gamtame \hspace{-0.07cm}+\hspace{-0.07cm} \sum_{\mathfrak q=1}^2\Gamq \Big\} \tilde n_p^i(ds,d\bar z)  \notag
			\\
			&\quad+\int_{0}^{t}\int_Z \Big\{\big|x_s^{i,N,n}+\Gamtame+\sum_{\mathfrak q=1}^2\Gamq\big|^{\bar p}-|x_s^{i,N,n}|^{\bar p}
			 \\
			&
			\quad -{\bar p}\,|x_s^{i,N,n}|^{{\bar p}-2}x_s^{i,N,n}\Big(\Gamtame+\sum_{\mathfrak q=1}^2\Gamq\Big)\Big\}n_p^i(ds,d\bar z)
		\end{align*}
		almost surely for any $t\in[ 0, T]$ and $i\in\{1,\ldots,N\}$ which on  talking expectation and using Assumption \ref{asump:super} for any $t\in[ 0, T]$ and $i\in\{1,\ldots,N\}$ yield
		\begin{align}  \label{eq:mb_est}
		&E|x_t^{i,N,n}|^{\bar p}  \leq  E|x_0^{i}|^{\bar p}+\frac{\bar p}{2}E\int_{0}^{t}\Big\{2|x_{\kappa_n(s)}^{i,N,n}|^{{\bar p}-2}x_{\kappa_n(s)}^{i,N,n}\Btame  \notag 
		\\
		&\,\,\, +(\bar p-1)|x_{\kappa_n(s)}^{i,N,n}|^{{\bar p}-2}\big|\Sigtame\big|^2 \notag
		\\
		&\,\,\,  
		+2(\bar p-1)\int_Z\big| \Gamtame \big|^2
		\int_{0}^1 (1-\theta)  \big |x_{\kappa_{n}(s)}^{i,N,n}+\theta\Gamtame \big|^{\bar p-2} d\theta \nu(d\bar z)  \notag\Big\} ds \notag
		\\
		&+{\bar p}E\int_{0}^{t}\big\{|x_s^{i,N,n}|^{{\bar p}-2}x_s^{i,N,n}-|x_{\kappa_n(s)}^{i,N,n}|^{{\bar p}-2}x_{\kappa_n(s)}^{i,N,n}\big\}\Btame  ds \notag
		\\
		&+\frac{\bar{p}(\bar{p}-1)}{2}E\int_{0}^{t}\big\{|x_s^{i,N,n}|^{{\bar p}-2}-|x_{\kappa_n(s)}^{i,N,n}|^{{\bar p}-2}\big\}\big|\Sigtame\big|^2  ds \notag
		\\ 
		&+\frac{\bar{p}(\bar{p}-1)}{2}E\int_{0}^{t}|x_s^{i,N,n}|^{{\bar p}-2}\Big\{\big| \Sigtame+ \sum_{\mathfrak q=1}^2\Sigq\big|^2 \notag
		\\
		&\qquad\qquad\qquad\qquad\qquad\qquad\qquad
		-|\Sigtame|^2 \Big\}ds \notag
		\\
		&+E\int_{0}^{t}\int_Z \Big\{\big|x_s^{i,N,n}+\Gamtame+\sum_{\mathfrak q=1}^2\Gamq\big|^{\bar p}-|x_s^{i,N,n}|^{\bar p} \notag
		\\
		&
		\quad-{\bar p}\,|x_s^{i,N,n}|^{{\bar p}-2}x_s^{i,N,n}\Big(\Gamtame+\sum_{\mathfrak q=1}^2\Gamq\Big) \notag
		\\
		&\quad-\bar p(\bar p-1)\big| \Gamtame \big|^2
		\int_{0}^1 (1-\theta)  \big |x_{\kappa_{n}(s)}^{i,N,n}+\theta\Gamtame \big|^{\bar p-2} d\theta \Big\}\nu(d\bar z)ds \notag
	\\
	& \qquad \qquad \leq  E|x_0^{i}|^{\bar p}+KE\int_0^t\big\{1+|x_{\kappa_n(s)}^{i,N,n}|^{p_0}+\mathcal{W}_2^{p_0}(\mu_{\kappa_n(s)}^{x,N,n},\delta_0)\big\}ds+\sum_{\mathfrak{q}=1}^{4} \Pi_{\mathfrak{q}}.
\end{align}

	To estimate $\Pi_1$,  one can write
	\begin{align*}
	\Pi_1 :&={\bar p}E\int_{0}^{t}\big\{|x_s^{i,N,n}|^{{\bar p}-2}x_s^{i,N,n}-|x_{\kappa_n(s)}^{i,N,n}|^{{\bar p}-2}x_{\kappa_n(s)}^{i,N,n}\big\}\Btame  ds \notag
		\\
		&={\bar p}E\int_{0}^{t}\Big\{\big(|x_s^{i,N,n}|^{{\bar p}-2}-|x_{\kappa_n(s)}^{i,N,n}|^{{\bar p}-2}\big)x_s^{i,N,n}+|x_{\kappa_n(s)}^{i,N,n}|^{{\bar p}-2}\big(x_s^{i,N,n}-x_{\kappa_n(s)}^{i,N,n}\big)\Big\}\Btame  ds \notag
		\\
		 &\leq   KE\int_{0}^{t}\Big\{\big(|x_{\kappa_n(s)}^{i,N,n}|^{\bar p-3}+|x_s^{i,N,n}-x_{\kappa_n(s)}^{i,N,n}|^{\bar p-3}\big) \big|x_s^{i,N,n}-x_{\kappa_n(s)}^{i,N,n}\big|\big(|x_{\kappa_n(s)}^{i,N,n}| +|x_s^{i,N,n}-x_{\kappa_n(s)}^{i,N,n}|\big) \notag
		\\
		& \qquad  +|x_{\kappa_n(s)}^{i,N,n}|^{{\bar p}-2}|x_s^{i,N,n}-x_{\kappa_n(s)}^{i,N,n}|\Big\}\big|\Btame\big|
		ds \notag
	\end{align*}
for any 	$t\in[ 0, T]$ and $i\in\{1,\ldots,N\}$ where the last inequality is obtained by using
	\begin{align} \label{eq:reminder}
		|y_1|^{\bar{p}-2}&- |y_2|^{\bar{p}-2}=(\bar{p}-2)\displaystyle\int_{0}^1|y_2+\theta (y_1-y_2)|^{\bar{p}-4}(y_2+\theta (y_1-y_2)) (y_1-y_2) d\theta \notag
		\\
		\leq& (\bar{p}-2) \int_0^1|y_2+\theta(y_1-y_2)|^{\bar{p}-3}|y_1-y_2| d\theta 
		\leq K \big(|y_2|^{\bar{p}-3}+|y_1-y_2|^{\bar{p}-3}\big) |y_1-y_2|
		\end{align}
	for any $y_1, y_2 \in \mathbb{R}^d$.
Further, one uses  Assumption  \ref{asump:tame}, {Young's inequality} and Lemma \ref{lem:dif:scm}  to get
	\begin{align*} 
		\Pi_1 &\leq  K E\int_{0}^{t} n^{\frac{1}{3}}\big(1+|x_{\kappa_n(s)}^{i,N,n}|+\mathcal{W}_2(\mu_{\kappa_n(s)}^{x,N,n},\delta_0)\big)\Big(|x_{\kappa_n(s)}^{i,N,n}|^{\bar p-2}  |x_s^{i,N,n}-x_{\kappa_n(s)}^{i,N,n}| 
		\notag
		\\
		& \qquad + |x_{\kappa_n(s)}^{i,N,n}|^{\bar p-3} |x_s^{i,N,n}-x_{\kappa_n(s)}^{i,N,n}|^{2} + |x_s^{i,N,n}-x_{\kappa_n(s)}^{i,N,n}|^{\bar p-2} |x_{\kappa_n(s)}^{i,N,n}| \notag
		 + |x_s^{i,N,n}-x_{\kappa_n(s)}^{i,N,n}|^{\bar p-1}\Big)  ds  \notag
		\\
		&\leq  K E\Big\{\int_{0}^{t} n^{\frac{1}{3}}\Big(1+|x_{\kappa_n(s)}^{i,N,n}|+\frac{1}{N} \sum_{j=1}^{N}|x_{\kappa_n(s)}^{j,N,n}|\Big) \Big(|x_{\kappa_n(s)}^{i,N,n}|^{\bar p-3}n^{-\frac{1}{3}}|x_{\kappa_n(s)}^{i,N,n}|^{2} 
  \\
  &\qquad+ |x_{\kappa_n(s)}^{i,N,n}|^{\bar p-3}n^{\frac{1}{3}}E_{\kappa_{n}(s)}|x_s^{i,N,n}-x_{\kappa_n(s)}^{i,N,n}| ^{2}+ |x_{\kappa_n(s)}^{i,N,n}|^{\bar p-3} E_{\kappa_{n}(s)}|x_s^{i,N,n}-x_{\kappa_n(s)}^{i,N,n}|^{2} \notag
		\\
		&\qquad+ |x_{\kappa_n(s)}^{i,N,n}|E_{\kappa_{n}(s)}|x_s^{i,N,n}-x_{\kappa_n(s)}^{i,N,n}|^{\bar p-2} + E_{\kappa_{n}(s)}|x_s^{i,N,n}-x_{\kappa_n(s)}^{i,N,n}|^{\bar p-1}\Big)   ds \Big\}  \notag
		\\
	&\leq Kn^{\frac{1}{3}} \big(n^{-\frac{1}{3}}+n^{-\frac{2}{3}}+\epsilon_n(\bar p)+n^{-\frac{3}{4}}\big)E\int_{0}^t\big(1+|x_{\kappa_n(s)}^{i,N,n}|^{\bar p}+\frac{1}{N} \sum_{k=1}^{N}|x_{\kappa_n(s)}^{k,N,n}|^{\bar p}\big) ds \notag
	\end{align*}
	for any $t\in[ 0, T]$ and $i\in\{1,\ldots,N\}$ where the following estimate  is also used (due to Lemma \ref{lem:dif:scm})
	\begin{align} \label{eq:onestep:p-2}
		E_{\kappa_{n}(s)}|x_s^{i,N,n}-x_{\kappa_n(s)}^{i,N,n}|^{\bar p-2}\leq K\epsilon_n(\bar p) \big(1+|x_{\kappa_n(s)}^{i,N,n}|^{\bar p-2}+\frac{1}{N} \sum_{k=1}^{N}|x_{\kappa_n(s)}^{k,N,n}|^{\bar p-2}\big)
	\end{align}
for any $\bar p>4$ with
$
	\epsilon_n(\bar p) := 
	\begin{cases}
		n^{-\frac{3}{4}}, & \mbox{ if }\,\,  17/4\leq \bar p,
		\vspace{0.2cm}
		\\
		n^{-\frac{\bar p-2}{3}}, & \mbox{ if }\,\, 4< \bar p<17/4
	\end{cases}.
$
Thus, one obtains,
\begin{align} \label{eq:pi1}
\Pi_1\leq  K\int_{0}^t\big(1+\sup_{i \in \{1,\ldots,N\}}\sup_{ r\in[0,s]}E|x_r^{i,N,n}|^{\bar p}\big)ds
\end{align}
	for any $t\in[ 0, T]$ and $i\in\{1,\ldots,N\}$.
	For $\Pi_2$, one  applies  Equation \eqref{eq:reminder} to get
\begin{align}
	\Pi_2 &:=\frac{\bar{p}(\bar{p}-1)}{2}E\int_{0}^{t}\big\{|x_s^{i,N,n}|^{{\bar p}-2}-|x_{\kappa_n(s)}^{i,N,n}|^{{\bar p}-2}\big\}\big|\Sigtame\big|^2  ds \notag
	\\
	\leq& KE\int_{0}^{t}\big\{|x_{\kappa_n(s)}^{i,N,n}|^{\bar p-3}+|x_s^{i,N,n}-x_{\kappa_n(s)}^{i,N,n}|^{\bar p-3}\big\} \big|x_s^{i,N,n}-x_{\kappa_n(s)}^{i,N,n}\big| \big|\Sigtame\big|^2  ds  \notag
	\\
	\leq &  K E\int_{0}^{t} \big\{|x_{\kappa_n(s)}^{i,N,n}|^{\bar p-3}  |x_s^{i,N,n}-x_{\kappa_n(s)}^{i,N,n}| 
	+  |x_s^{i,N,n}-x_{\kappa_n(s)}^{i,N,n}|^{\bar p-2} \big\}\big|\Sigtame\big|^2   ds  \notag
\end{align}
which on applying  Assumption  \ref{asump:tame}, {Young's inequality}, Lemma \ref{lem:dif:scm} and Equation \eqref{eq:onestep:p-2} yield
\begin{align} \label{eq:pi2}
	\Pi_2 
	\leq &  K E\int_{0}^{t} n^{\frac{1}{3}}\big(1+|x_{\kappa_n(s)}^{i,N,n}|+\mathcal{W}_2(\mu_{\kappa_n(s)}^{x,N,n},\delta_0)\big)^2\Big( |x_{\kappa_n(s)}^{i,N,n}|^{\bar p-4}n^{-\frac{1}{3}}|x_{\kappa_n(s)}^{i,N,n}|^{2}\notag
	\\
	&  + |x_{\kappa_n(s)}^{i,N,n}|^{\bar p-4}n^{\frac{1}{3}}E_{\kappa_{n}(s)}|x_s^{i,N,n}-x_{\kappa_n(s)}^{i,N,n}| ^{2}+E_{\kappa_{n}(s)}|x_s^{i,N,n}-x_{\kappa_n(s)}^{i,N,n}|^{\bar p-2} \Big) ds  \notag
	\\
 \leq &K n^{\frac{1}{3}}(n^{-\frac{1}{3}}+\epsilon_n(\bar p)) E\int_{0}^{t}\big(1+|x_{\kappa_n(s)}^{i,N,n}|^{\bar p}+\frac{1}{N} \sum_{k=1}^{N}|x_{\kappa_n(s)}^{k,N,n}|^{\bar p}\big) ds \notag
	\\
	\leq  & K  \int_{0}^t\big(1+\sup_{i \in \{1,\ldots,N\}}\sup_{ r\in[0,s]}E|x_r^{i,N,n}|^{\bar p}\big)ds
\end{align}
for any $t\in[ 0, T]$ and $i\in\{1,\ldots,N\}$.
To estimate $\Pi_3$, one uses \, 
$
|y_1+y_2|^2-|y_1|^2 =   |y_2|^2 + 2 \displaystyle  \sum_{u=1}^d \sum_{\ell=1}^m y_1^{u \ell} y_2^{u \ell}$ \, for any $y_1, y_2 \in \mathbb{R}^{d \times m}
$ 
as follows
\begin{align*} \label{eq:pi3}
	\Pi_3:=& \frac{\bar{p}(\bar{p}-1)}{2}E\int_{0}^{t}|x_s^{i,N,n}|^{{\bar p}-2}\Big\{\big| \Sigtame+ \sum_{\mathfrak q=1}^2\Sigq\big|^2 \notag
	\\
	&\qquad\qquad\qquad\qquad\qquad\qquad\qquad\qquad
	-|\Sigtame|^2 \Big\}ds \notag
	\\
	 \leq &KE\int_{0}^t |x_{\kappa_{n}(s)}^{i,N,n}+x_s^{i,N,n}-x_{\kappa_{n}(s)}^{i,N,n}|^{{\bar p-2}} \Big\{\big| \sum_{\mathfrak q=1}^2\Sigq \big|^2  \notag
	\\
	&\qquad 
	+ 2\sum_{u=1}^d\sum_{\ell=1}^m \sigtame \sum_{\mathfrak q=1}^2  \sigq \Big\} ds
\end{align*}
which on using  Cauchy--Schwarz  inequality and Assumption  \ref{asump:tame}  yield
\begin{align*} 
	\Pi_3
	\leq &  K E\int_{0}^{t} \big(|x_{\kappa_n(s)}^{i,N,n}|^{\bar p-2}  
	+  |x_s^{i,N,n}-x_{\kappa_n(s)}^{i,N,n}|^{\bar p-2} \big) \Big\{\sum_{\mathfrak q=1}^2 \big|\Sigq \big|^2    \notag
	\\
	&+    \big|\Sigtame\big|  \sum_{\mathfrak q=1}^2  \big|\Sigq\big| \Big\}ds   \notag
	\\
	&\leq K E\Big\{\int_{0}^{t}  |x_{\kappa_n(s)}^{i,N,n}|^{\bar p-2}  
	\Big(\sum_{\mathfrak q=1}^2E_{\kappa_n(s)}\big| \Sigq \big|^2 \notag
	\\
	&  +n^{\frac{1}{6}}\big(1+|x_{\kappa_n(s)}^{i,N,n}|+\mathcal{W}_2(\mu_{\kappa_n(s)}^{x,N,n},\delta_0)\big)
	\sum_{\mathfrak q=1}^2E_{\kappa_n(s)}\big| \Sigq \big|\Big) ds \Big\} \notag
\\
&  + K E\int_{0}^{t}\Big\{|x_{s}^{i,N,n}-x_{\kappa_n(s)}^{i,N,n}|^{\bar p-2}  n^{\frac{3(\bar p-2)}{4\bar p}}
n^{-\frac{3(\bar p-2)}{4\bar p}}\sum_{\mathfrak q=1}^2\big| \Sigq \big|^2 \notag
\\
&  +n^{\frac{1}{6}}\big(1+|x_{\kappa_n(s)}^{i,N,n}|+\mathcal{W}_2(\mu_{\kappa_n(s)}^{x,N,n},\delta_0)\big)
|x_{s}^{i,N,n}-x_{\kappa_{n}(s)}^{i,N,n}|^{\bar p-2}n^{\frac{7(\bar p-2)}{12(\bar p-1)}}  \notag
\\
&\qquad\qquad\qquad \times n^{-\frac{7(\bar p-2)}{12(\bar p-1)}}\sum_{\mathfrak q=1}^2 \big| \Sigq \big| \Big\}ds 
\end{align*}
for any  $t\in[0,T]$ and $i\in\{1,\ldots,N\}$. 
Further, by using Lemma \ref{lem:sig_gam}, H\"older's inequality, Young's inequality and  Lemma \ref{lem:dif:scm}, one obtains
\begin{align} 
	&\Pi_3
	\leq  K n^{-\frac{5}{12}} E\int_{0}^{t}  \big(1+|x_{\kappa_n(s)}^{i,N,n}|^{\bar p}+\frac{1}{N} \sum_{k=1}^{N}|x_{\kappa_n(s)}^{k,N,n}|^{\bar p}\big)   ds  + K E\Big\{\int_{0}^{t}  |x_{\kappa_n(s)}^{i,N,n}|^{\bar p-2}  
	\big(1+|x_{\kappa_n(s)}^{i,N,n}|+\frac{1}{N} \sum_{j=1}^{N}|x_{\kappa_n(s)}^{j,N,n}|\big) \notag
 \\
	&  \qquad\times n^{\frac{1}{6}}\sum_{\mathfrak q=1}^2\Big(E_{\kappa_n(s)}\big| \Sigq \big|^2\Big)^{\frac{1}{2}}  ds \Big\} \notag
\\
&  + KE\Big\{\int_{0}^{t}   E_{\kappa_{n}(s)}\Big(n^{\frac{3}{4}}|x_{s}^{i,N,n}-x_{\kappa_n(s)}^{i,N,n}|^{\bar p}  + n^{-\frac{3(\bar p-2)}{8}}\sum_{\mathfrak q=1}^2\big| \Sigq \big|^{\bar p}
\Big) ds \Big\} \notag
\\
&  + Kn^{\frac{1}{6}} E\Big\{\int_{0}^{t}  \big(1+|x_{\kappa_n(s)}^{i,N,n}|+\frac{1}{N} \sum_{j=1}^{N}|x_{\kappa_n(s)}^{j,N,n}|\big)  \notag
\\
&  \qquad\times
E_{\kappa_{n}(s)}\Big(n^{\frac{7}{12}}|x_{s}^{i,N,n}-x_{\kappa_n(s)}^{i,N,n}|^{\bar p-1}+n^{-\frac{7(\bar p-2)}{12}}\sum_{\mathfrak q=1}^2 \big| \Sigq \big|^{\bar p-1}\Big) ds \Big\}  \notag
\\
&\quad\leq K\big(1+n^{-\frac{1}{24}}
+n^{-\frac{3(\bar p-2)}{8}}n^{\frac{2\bar p-9}{12}}+n^{\frac{1}{6}}n^{-\frac{7(\bar p-2)}{12}}n^{\frac{2\bar p-11}{12}}\big)
E\int_{0}^{t}\big(1+|x_{\kappa_n(s)}^{i,N,n}|^{\bar p}+\frac{1}{N} \sum_{k=1}^{N}|x_{\kappa_n(s)}^{k,N,n}|^{\bar p}\big) ds \notag
	\\
	&\quad \leq   K  \int_{0}^t\big(1+\sup_{i \in \{1,\ldots,N\}}\sup_{ r\in[0,s]}E|x_r^{i,N,n}|^{\bar p}\big)ds
\end{align}
for any  $t\in[0,T]$ and $i\in\{1,\ldots,N\}$.
For $\Pi_4$, one uses
Lemma \eqref{lem:mvt}  to obtain 
\begin{align*}
	\Pi_4 &:=E\int_{0}^{t}\int_Z \Big\{\big|x_s^{i,N,n}+\Gamtame+\sum_{\mathfrak q=1}^2\Gamq\big|^{\bar p}-|x_s^{i,N,n}|^{\bar p} \notag
\\
&
\quad-{\bar p}\,|x_s^{i,N,n}|^{{\bar p}-2}x_s^{i,N,n}\Big(\Gamtame+\sum_{\mathfrak q=1}^2\Gamq\Big)\notag
\\
&\quad-\bar p(\bar p-1)\big| \Gamtame \big|^2 \hspace{-0.04cm}
\int_{0}^1 \hspace{-0.04cm}(1-\theta)  \big |x_{\kappa_{n}(s)}^{i,N,n}\hspace{-0.02cm}+\hspace{-0.02cm}\theta\Gamtame \big|^{\bar p-2} d\theta \Big\}\nu(d\bar z)ds \notag
	\\
	& \leq \bar p(\bar p-1)E\int_{0}^{t}\int_Z\Big\{\big| \Gamtame+\sum_{\mathfrak q=1}^2\Gamq \big|^2
	\\
	&\quad \times  \int_{0}^1 (1-\theta)\Big |x_s^{i,N,n}+\theta\Gamtame+\theta\sum_{\mathfrak q=1}^2\Gamq \Big|^{\bar p-2} d\theta 
	\\
	&\quad- \big| \Gamtame \big|^2
	\int_{0}^1 (1-\theta)  \big |x_{\kappa_{n}(s)}^{i,N,n}+\theta\Gamtame \big|^{\bar p-2} d\theta\Big\} \nu(d\bar z)ds
\end{align*}
 which on using \, $|y_1+y_2|^2 = |y_1|^2 + |y_2|^2 + 2 \displaystyle  \sum_{u=1}^d y_1^{u} y_2^{u}$ \, for any $y_1, y_2 \in \mathbb{R}^d$ yields
\begin{align} \label{eq:jump_term}
	 \Pi_4&\leq\bar p(\bar p-1)E\bigg[\int_{0}^{t}\int_Z \big| \Gamtame \big|^2
	\int_{0}^1 (1-\theta) \Big\{ \big |x_s^{i,N,n}+\theta\Gamtame \notag
	\\
	&      +\theta\sum_{\mathfrak q=1}^2\Gamq \big|^{\bar p-2}    - \big |x_{\kappa_{n}(s)}^{i,N,n}+\theta\Gamtame \big|^{\bar p-2}\Big\} d\theta\nu(d\bar z)ds\bigg]  \notag
	\\
	& \qquad +p(\bar p-1)E\bigg[\int_{0}^{t}\int_Z\Big\{\big| \sum_{\mathfrak q=1}^2\Gamq \big|^2  \notag
	\\
	&\qquad  \qquad \qquad
	+2 \sum_{u=1}^d \gamtame \sum_{\mathfrak q=1}^2  \gamq \Big\} \notag
	\\
	& \qquad \qquad\qquad \quad \times \int_{0}^1 (1-\theta) \Big |x_s^{i,N,n} {\bar p-2}+\theta\Gamtame \notag
 \\
 &\quad\qquad \qquad\qquad\quad +\theta\sum_{\mathfrak q=1}^2\Gamq \Big|^    d\theta\nu(d\bar z)ds\bigg] 
	=:\Pi_{41}+\Pi_{42}
\end{align}
for any $t\in[0,T]$ and  $i\in\{1,\ldots,N\}$.
For  $\Pi_{41}$, one can write as follows by applying Equation \eqref{eq:reminder}
\begin{align*} 
	&\Pi_{41}:=\bar p(\bar p -1)E\int_{0}^{t}\int_Z \big| \Gamtame \big|^2
	\int_{0}^1 (1-\theta) \Big\{ \big |x_s^{i,N,n}+\theta\Gamtame \notag
	\\
	& \quad  +\theta\sum_{\mathfrak q=1}^2\Gamq \big|^{\bar p-2} - \big |x_{\kappa_{n}(s)}^{i,N,n}+\theta\Gamtame \big|^{\bar p-2}\Big\} d\theta\nu(d\bar z)ds \notag
	\\
	&\leq KE\int_{0}^{t}\int_Z \big| \Gamtame \big|^{2}  \int_{0}^{1} \Big\{\big| x_{\kappa_{n}(s)}^{i,N,n}+\theta\Gamtame\big|^{\bar p-3}  \notag
	\\
	&\qquad  +\big|\big(x_s^{i,N,n}-x_{\kappa_{n}(s)}^{i,N,n}\big) +\theta\sum_{\mathfrak q=1}^2\Gamq\big|^{\bar p-3} \Big\}  \notag
	\\
	&\qquad\qquad \times 
	\big|\big(x_s^{i,N,n}-x_{\kappa_{n}(s)}^{i,N,n}\big)+\theta\sum_{\mathfrak q=1}^2\Gamq\big| d\theta\nu(d\bar z)ds 
	\\
	&\leq KE\int_{0}^{t}\int_Z   \Big\{ \big| \Gamtame \big|^2  |x_{\kappa_{n}(s)}^{i,N,n}|^{\bar p-3} |x_s^{i,N,n}-x_{\kappa_{n}(s)}^{i,N,n}|
	\\
	&\quad+\big| \Gamtame \big|^{\bar p-1} |x_s^{i,N,n}-x_{\kappa_{n}(s)}^{i,N,n}|
	+\big| \Gamtame \big|^{2} |x_s^{i,N,n}-x_{\kappa_{n}(s)}^{i,N,n}\big|^{\bar p-2} 
	\\
	&\quad+ \big| \Gamtame \big|^2
	\sum_{\mathfrak q=1}^2\big|\Gamq\big|^{\bar p-3} |x_s^{i,N,n}-x_{\kappa_{n}(s)}^{i,N,n}|
	\\
	&\quad+\big| \Gamtame \big|^2 |x_{\kappa_{n}(s)}^{i,N,n}|^{\bar p-3} \sum_{\mathfrak q=1}^2\big|\Gamq\big| 
	\\
	&\quad+\big| \Gamtame \big|^{\bar p-1} \sum_{\mathfrak q=1}^2\big|\Gamq\big|
	\\
	&\quad+\big| \Gamtame \big|^{2} |x_s^{i,N,n}-x_{\kappa_{n}(s)}^{i,N,n}\big|^{\bar p-3} 
	\sum_{\mathfrak q=1}^2\big|\Gamq\big| 
	\\
	&\quad+ \big| \Gamtame \big|^2\sum_{\mathfrak q=1}^2\big|\Gamq\big|^{\bar p-2}\Big\} \nu(d\bar z)ds \notag
		\end{align*}
which on applying Assumption \ref{asump:tame} and Young's inequality yields 
\begin{align*}
	\Pi_{41}
	&\leq  KE\Big\{\int_{0}^{t} \Big[
	n^{\frac{1}{4}}\big(1+|x_{\kappa_n(s)}^{i,N,n}|+\mathcal{W}_2(\mu_{\kappa_n(s)}^{x,N,n},\delta_0)\big)^2 |x_{\kappa_{n}(s)}^{i,N,n}|^{\bar p-4} \Big(n^{-\frac{1}{3}}|x_{\kappa_n(s)}^{i,N,n}|^{2}
	\\
	&  \qquad+ n^{\frac{1}{3}}E_{\kappa_{n}(s)}|x_s^{i,N,n}-x_{\kappa_n(s)}^{i,N,n}| ^{2}\Big)
 +n^{\frac{1}{4}}\big(1+|x_{\kappa_n(s)}^{i,N,n}|+\mathcal{W}_2(\mu_{\kappa_n(s)}^{x,N,n},\delta_0)\big)^{\bar p-1} |x_s^{i,N,n}-x_{\kappa_{n}(s)}^{i,N,n}|
	\\
	&\qquad+  n^{\frac{1}{4}}\big(1+|x_{\kappa_n(s)}^{i,N,n}|+\mathcal{W}_2(\mu_{\kappa_n(s)}^{x,N,n},\delta_0)\big)^{2} E_{\kappa_{n}(s)}|x_s^{i,N,n}-x_{\kappa_{n}(s)}^{i,N,n}\big|^{\bar p-2} \Big] ds \Big\}
	\\
	&\quad+ KE\Big\{\int_{0}^{t}\int_Z\Big\{ \big| \Gamtame \big|^{2} n^{-\frac{1}{2\bar p}}
	\\
	&\qquad\times  n^{\frac{1}{2\bar p}}E_{\kappa_{n}(s)}\Big(\sum_{\mathfrak q=1}^2\big|\Gamq\big|^{\bar p-2}+\big|x_s^{i,N,n}-x_{\kappa_{n}(s)}^{i,N,n}\big|^{\bar p-2}\Big)  \notag 
	\\
	&
 +|x_{\kappa_{n}(s)}^{i,N,n}|^{\bar p-3}\big| \Gamtame \big|^2 n^{-\frac{1}{6}}
	 n^{\frac{1}{6}}\sum_{\mathfrak q=1}^2E_{\kappa_{n}(s)}\big|\Gamq\big|
	\\
	&+\big| \Gamtame \big|^{\bar p-1} n^{-\frac{\bar p-1}{4\bar p}} n^{\frac{\bar p-1}{4\bar p}}\sum_{\mathfrak q=1}^2E_{\kappa_{n}(s)}\big|\Gamq\big| 
	\\
	&+ \big| \Gamtame \big|^2  n^{-\frac{1}{2\bar p}}n^{\frac{1}{2\bar p}}\sum_{\mathfrak q=1}^2E_{\kappa_{n}(s)}\big|\Gamq\big|^{\bar p-2} \Big\} \nu(d\bar z)ds\Big\} \notag
\end{align*}
for any $t\in[0,T]$ and $i\in\{1,\ldots,N\}$.
Moreover, by using Lemma \ref{lem:dif:scm}, Equation \eqref{eq:onestep:p-2} and Young's inequality, one gets
\begin{align*}
	\Pi_{41}
	\leq & K n^{\frac{1}{4}}(n^{-\frac{1}{3}}+\epsilon_n(\bar p)) E\int_{0}^{t}\big(1+|x_{\kappa_n(s)}^{i,N,n}|^{\bar p}+\frac{1}{N} \sum_{k=1}^{N}|x_{\kappa_n(s)}^{k,N,n}|^{\bar p}\big) ds \notag
 \\
 &+  KE\Big\{\int_{0}^{t} 
	n^{\frac{1}{4}}\big(1+|x_{\kappa_n(s)}^{i,N,n}|+\mathcal{W}_2(\mu_{\kappa_n(s)}^{x,N,n},\delta_0)\big)^{\bar p-2} 
 \\
 &\qquad\times\Big(n^{-\frac{1}{3}}\big(1+|x_{\kappa_n(s)}^{i,N,n}|+\mathcal{W}_2(\mu_{\kappa_n(s)}^{x,N,n},\delta_0)\big)^{2} + n^{\frac{1}{3}}E_{\kappa_{n}(s)}|x_s^{i,N,n}-x_{\kappa_n(s)}^{i,N,n}| ^{2}\Big)ds\Big\} \color{black}
	\\
	&+ KE\int_{0}^{t}\int_Z  n^{-\frac{1}{4}}\big| \Gamtame \big|^{\bar p}  \nu(d\bar z)ds
	\\
	&+ KE\Big\{\int_{0}^{t}\int_Z n^{\frac{1}{2(\bar p-2)}}\Big(\sum_{\mathfrak q=1}^2E_{\kappa_{n}(s)}\big|\Gamq\big|^{\bar p-2}
	\\
	&\qquad\qquad\qquad\qquad\qquad\qquad
	+E_{\kappa_{n}(s)}\big|x_s^{i,N,n}-x_{\kappa_{n}(s)}^{i,N,n}\big|^{\bar p-2}\Big)^{\frac{\bar p}{\bar p-2}}  \nu(d\bar z)ds\Big\} \notag
	\\
	&+ KE\int_{0}^{t}\int_Z  |x_{\kappa_{n}(s)}^{i,N,n}|^{\bar p-3} n^{-\frac{1}{4}}\big| \Gamtame \big|^{3}  \nu(d\bar z)ds
	\\
	&
	+KE\Big\{\int_{0}^{t}\int_Z  |x_{\kappa_{n}(s)}^{i,N,n}|^{\bar p-3} n^{\frac{1}{2}}\sum_{\mathfrak q=1}^2 \big(E_{\kappa_{n}(s)}\big|\Gamq\big|\big)^3  \nu(d\bar z)ds \Big\}
	\\
	& + KE\Big\{\int_{0}^{t}\int_Z  n^{\frac{\bar p-1}{4}}\sum_{\mathfrak q=1}^2\big(E_{\kappa_{n}(s)}\big|\Gamq\big|\big)^{\bar p}  \nu(d\bar z)ds\Big\}
	\\
	&
	+KE\Big\{\int_{0}^{t}\int_Z    n^{\frac{1}{2(\bar p-2)}}\sum_{\mathfrak q=1}^2 \big(E_{\kappa_{n}(s)}\big|\Gamq\big|^{\bar p-2}\big)^{\frac{\bar p}{\bar p-2}}  \nu(d\bar z)ds  \Big\}
\end{align*}
 which on   applying Assumption  \ref{asump:tame}, Lemma \ref{lem:sig_gam}, Equation \eqref{eq:onestep:p-2} and  H\"older's inequality give
\begin{align} \label{eq:pi41}
	\Pi_{41}
\leq &  K  \Big\{1+ {n^{\frac{1}{4}}n^{-\frac{1}{3}}}+n^{\frac{1}{2(\bar p-2)}}(n^{-\frac{\bar p}{4}}+n^{-\frac{\bar p}{(\bar p-2)}+\frac{1}{2}}+\epsilon_n(\bar p)^{\frac{\bar p}{\bar p-2}})\Big\}  E\int_{0}^{t}\big(1+|x_{\kappa_n(s)}^{i,N,n}|^{\bar p}+\frac{1}{N} \sum_{k=1}^{N}|x_{\kappa_n(s)}^{k,N,n}|^{\bar p}\big) ds \notag
		\\
	&
	+KE\Big\{\int_{0}^{t}\int_Z  |x_{\kappa_{n}(s)}^{i,N,n}|^{\bar p-3}n^{\frac{1}{2}}\sum_{\mathfrak q=1}^2 \big(E_{\kappa_{n}(s)}\big|\Gamq\big|^2\big)^{\frac{3}{2}} \nu(d\bar z)ds\Big\}  \notag
	\\
	& + KE\Big\{\int_{0}^{t}\int_Z  n^{\frac{\bar p-1}{4}}\sum_{\mathfrak q=1}^2\big(E_{\kappa_{n}(s)}\big|\Gamq\big|^2\big)^{\frac{\bar p}{2}}  \nu(d\bar z)ds\Big\} \notag
	\\
\leq & K(1+n^{\frac{1}{2}}(n^{-\frac{3}{4}}+n^{-\frac{3}{2}+\frac{1}{2}})+ n^{\frac{\bar p-1}{4}}(n^{-\frac{\bar p}{4}}+n^{-\frac{\bar p}{2}+\frac{1}{2}}) )  E\int_{0}^{t}\big(1+|x_{\kappa_n(s)}^{i,N,n}|^{\bar p}+\frac{1}{N} \sum_{k=1}^{N}|x_{\kappa_n(s)}^{k,N,n}|^{\bar p}\big) ds \notag
\\
\leq & K \int_{0}^t\big(1+\sup_{i \in \{1,\ldots,N\}}\sup_{ r\in[0,s]}E|x_r^{i,N,n}|^{\bar p}\big)ds
\end{align}
for any $t\in[0,T]$ and $i\in\{1,\ldots,N\}$.
To estimate the  the last  term $\Pi_{42}$  of  Equation \eqref{eq:jump_term}, one uses Cauchy--Schwarz  inequality to proceed as follows
\begin{align*} 
	&\Pi_{42}:=\bar p(\bar p -1)E\int_{0}^{t}\int_Z \Big\{\big| \sum_{\mathfrak q=1}^2\Gamq \big|^2  \notag
	\\
	& 
	+2 \sum_{u=1}^d \gamtame \sum_{\mathfrak q=1}^2  \gamq \Big\} \notag
	\\
	 \times& \int_{0}^1 (1-\theta) \Big |x_s^{i,N,n}+\theta\Gamtame+\theta\sum_{\mathfrak q=1}^2\Gamq \Big|^{\bar p-2}    d\theta\nu(d\bar z)ds \notag
	\\
	\leq& K E\int_{0}^{t}\int_Z \Big\{\big| \sum_{\mathfrak q=1}^2\Gamq \big|^2  \notag
	\\
	&
	+\hspace{-0.025cm}\big|\Gamtame\big| \big|\sum_{\mathfrak q=1}^2  \Gamq\big| \Big\}\hspace{-0.025cm} \Big \{|x_{\kappa_{n}(s)}^{i,N,n}|^{\bar p-2} \hspace{-0.025cm} +\hspace{-0.025cm} |x_s^{i,N,n}\hspace{-0.025cm}-\hspace{-0.025cm}x_{\kappa_{n}(s)}^{i,N,n}|^{\bar p-2}  \notag
	\\
	&  +\big|\Gamtame\big|^{\bar p-2} +\big|\sum_{\mathfrak q=1}^2\Gamq\big|^{\bar p-2}  \Big\}   \nu(d\bar z)ds \notag
	\\
	\leq & K E\int_{0}^{t}\int_Z \Big\{ \sum_{\mathfrak q=1}^2  \big|\Gamq \big|^2 |x_{\kappa_{n}(s)}^{i,N,n}|^{\bar p-2} \notag
	\\
	&+\sum_{\mathfrak q=1}^2  \big|\Gamq \big|^2 |x_s^{i,N,n}-x_{\kappa_{n}(s)}^{i,N,n}|^{\bar p-2}  \notag
	\\
	& +\sum_{\mathfrak q=1}^2  \big|\Gamq \big|^2 \big|\Gamtame\big|^{\bar p-2} 
	\\
	&  + \sum_{\mathfrak q=1}^2  \big|\Gamq\big|^{\bar p} \notag
		\\
	& +  \big|\Gamtame\big| \sum_{\mathfrak q=1}^2  \big|\Gamq\big| |x_{\kappa_{n}(s)}^{i,N,n}|^{\bar p-2} \notag
	\\
	& +\big|\Gamtame\big| \sum_{\mathfrak q=1}^2  \big|\Gamq\big| |x_s^{i,N,n}-x_{\kappa_{n}(s)}^{i,N,n}|^{\bar p-2}  \notag
	\\
	& +\big|\Gamtame\big|^{\bar p-1} \sum_{\mathfrak q=1}^2  \big|\Gamq\big|  
	\\
	&  +\big|\Gamtame\big|\sum_{\mathfrak q=1}^2  \big|\Gamq\big|^{\bar p-1}    \notag
 \Big\} \nu(d\bar z)ds
\end{align*}
which on using   Young's inequality yields
\begin{align*} 
	&\Pi_{42}
	\leq  K E\Big\{\int_{0}^{t}\int_Z  |x_{\kappa_{n}(s)}^{i,N,n}|^{\bar p-2} E_{\kappa_{n}(s)}\sum_{\mathfrak q=1}^2\big|\Gamq \big|^2\nu(d\bar z)ds\Big\} \notag
	\\
	& + K E\Big\{\int_{0}^{t}\int_Z\sum_{\mathfrak q=1}^2  E_{\kappa_{n}(s)}\big|\Gamq \big|^{\bar p} \nu(d\bar z)ds\Big\} \notag
	\\
	& +  K E\Big\{\int_{0}^{t}E_{\kappa_{n}(s)}|x_s^{i,N,n}-x_{\kappa_{n}(s)}^{i,N,n}|^{\bar p} ds\Big\}+ K E\Big\{\int_{0}^{t}\int_Z\big|\Gamtame\big|^{\bar p-2} \notag
	\\
	&  \qquad \times  n^{-\frac{\bar p-2}{4\bar p}} n^{\frac{\bar p-2}{4\bar p}} \sum_{\mathfrak q=1}^2E_{\kappa_{n}(s)}\big|\Gamq \big|^2\nu(d\bar z)ds\Big\} 
	\\
	& + K E\Big\{\int_{0}^{t}\int_Z |x_{\kappa_{n}(s)}^{i,N,n}|^{\bar p-2}\big|\Gamtame\big|\notag
	\\
	&  \qquad \times n^{-\frac{1}{8}} n^{\frac{1}{8}} \sum_{\mathfrak q=1}^2 E_{\kappa_{n}(s)}\big|\Gamq\big|\nu(d\bar z)ds\Big\}  \notag
	\\
	& +K  E\Big\{\int_{0}^{t}\int_Z\big|\Gamtame\big|n^{-\frac{1}{4\bar p}} n^{\frac{1}{4\bar p}}   \notag
	\\
	& \qquad\times  \Big(\sum_{\mathfrak q=1}^2E_{\kappa_{n}(s)}\big|\Gamq\big|^{\bar p-1}
	+E_{\kappa_{n}(s)}|x_s^{i,N,n}-x_{\kappa_{n}(s)}^{i,N,n}|^{\bar p-1}\Big)\nu(d\bar z)ds\Big\}  \notag
	\\
	&  +K E\Big\{\int_{0}^{t}\int_Z\big|\Gamtame\big|^{\bar p-1}\notag
	\\
	&  \qquad \times n^{-\frac{\bar p-1}{4\bar p}} n^{\frac{\bar p-1}{4\bar p}} \sum_{\mathfrak q=1}^2 E_{\kappa_{n}(s)}\big|\Gamq\big| \nu(d\bar z)ds\Big\} 
	\\
	&  +K E\Big\{\int_{0}^{t}\int_Z\big|\Gamtame\big|\notag
	\\
	&  \qquad \times n^{-\frac{1}{4\bar p}} n^{\frac{1}{4\bar p}} \sum_{\mathfrak q=1}^2 E_{\kappa_{n}(s)}\big|\Gamq\big|^{\bar p-1}     \nu(d\bar z)ds\Big\}
\end{align*}
for any $t\in[0,T]$ and $i\in\{1,\ldots,N\}$.
Furthermore, by applying Lemmas \ref{lem:sig_gam} and \ref{lem:dif:scm} along with Young's inequality, one obtains
\begin{align*}
\Pi_{42}
\leq &  K  \big(n^{-\frac{1}{2}}+n^{-\frac{\bar p}{4}}+n^{-\frac{3}{4}}\big) E\int_{0}^{t}\big(1+|x_{\kappa_n(s)}^{i,N,n}|^{\bar p}+\frac{1}{N} \sum_{k=1}^{N}|x_{\kappa_n(s)}^{k,N,n}|^{\bar p}\big) ds \notag
\\
& 
+ K E\int_{0}^{t}\int_Z
n^{-\frac{1}{4}}\big|\Gamtame\big|^{\bar p} \nu(d\bar z) ds  
\\
&+ K E\Big\{\int_{0}^{t}\int_Z n^{\frac{\bar p-2}{8}} \sum_{\mathfrak q=1}^2\big(E_{\kappa_{n}(s)}\big|\Gamq \big|^{2} \big)^{\frac{\bar p}{2}}  \nu(d\bar z)ds \Big\}
\\
& 
+ K E\int_{0}^{t}\int_Z
|x_{\kappa_{n}(s)}^{i,N,n}|^{\bar p-2} n^{-\frac{1}{4}}\big|\Gamtame\big|^{2} \nu(d\bar z) ds 
\\
& + K E\Big\{\int_{0}^{t}\int_Z |x_{\kappa_{n}(s)}^{i,N,n}|^{\bar p-2} n^{\frac{1}{4}} \sum_{\mathfrak q=1}^2 \big(E_{\kappa_{n}(s)}\big|\Gamq\big|\big)^2 \nu(d\bar z)ds \Big\} \notag
\\
 &+ K E\Big\{\int_{0}^{t}\int_Z  n^{\frac{1}{4(\bar p-1)}}\Big(\sum_{\mathfrak q=1}^2 E_{\kappa_{n}(s)}\big|\Gamq\big|^{\bar p-1} 
\\
&\qquad\qquad\qquad\qquad\qquad\qquad
+E_{\kappa_{n}(s)}|x_s^{i,N,n}-x_{\kappa_{n}(s)}^{i,N,n}|^{\bar p-1}\Big)^{\frac{\bar p}{\bar p-1}} \nu(d\bar z)ds \Big\} \notag
\\
&\quad+ K E\Big\{\int_{0}^{t}\int_Z n^{\frac{\bar p-1}{4}} \sum_{\mathfrak q=1}^2 \big(E_{\kappa_{n}(s)}\big|\Gamq\big|\big) ^{\bar p} \nu(d\bar z)ds \Big\}
\\
& + K E\Big\{\int_{0}^{t}\int_Z n^{\frac{1}{4(\bar p-1)}}\sum_{\mathfrak q=1}^2 \big(E_{\kappa_{n}(s)}\big|\Gamq\big|^{\bar p-1}\big)^{\frac{\bar p}{\bar p-1}} \nu(d\bar z)ds \Big\}
\end{align*}
which on using  Assumption  \ref{asump:tame}, Lemmas \ref{lem:sig_gam}, \ref{lem:dif:scm} and H\"older's inequality give
\begin{align} \label{eq:pi42}
	&\Pi_{42}
	\leq  K    \Big\{1+ n^{\frac{\bar p-2}{8}}(n^{-\frac{\bar p}{4}}+n^{-\frac{\bar p}{2}+\frac{1}{2}})+ n^{\frac{1}{4(\bar p-1)}}\big(n^{-\frac{\bar p}{4}}+n^{-\frac{\bar p}{(\bar p-1)}+\frac{1}{2}}+n^{-\frac{3\bar p}{4(\bar p-1)}}\big)\Big\}  \notag
	\\
	&\qquad\qquad\times E\int_{0}^{t}\big(1+|x_{\kappa_n(s)}^{i,N,n}|^{\bar p}+\frac{1}{N} \sum_{k=1}^{N}|x_{\kappa_n(s)}^{k,N,n}|^{\bar p}\big) ds  \notag
	\\
	&\qquad\quad+ K E\Big\{\int_{0}^{t}\int_Z |x_{\kappa_{n}(s)}^{i,N,n}|^{\bar p-2}n^{\frac{1}{4}} \sum_{\mathfrak q=1}^2 E_{\kappa_{n}(s)}\big|\Gamq\big|^2  \nu(d\bar z)ds\Big\} \notag
	\\
	&\qquad\quad+ K E\Big\{\int_{0}^{t}\int_Z n^{\frac{\bar p-1}{4}} \sum_{\mathfrak q=1}^2 \big(E_{\kappa_{n}(s)}\big|\Gamq\big|^2\big) ^{\frac{\bar p}{2}} \nu(d\bar z)ds\Big\} \notag
	\\
	&\quad\,\, \leq  K    \Big\{1+n^{\frac{1}{4}}n^{-\frac{1}{2}}+ n^{\frac{\bar p-1}{4}}(n^{-\frac{\bar p}{4}}+n^{-\frac{\bar p}{2}+\frac{1}{2}})\Big\}   E\int_{0}^{t}\big(1+|x_{\kappa_n(s)}^{i,N,n}|^{\bar p}+\frac{1}{N} \sum_{k=1}^{N}|x_{\kappa_n(s)}^{k,N,n}|^{\bar p}\big) ds  \notag
	\\
	&\quad\,\,\leq K \int_{0}^t\big(1+\sup_{i \in \{1,\ldots,N\}}\sup_{ r\in[0,s]}E|x_r^{i,N,n}|^{\bar p}\big)ds
\end{align}
for any $t\in[0,T]$ and $i\in\{1,\ldots,N\}$. 
Further, by using Equations \eqref{eq:pi41} and \eqref{eq:pi42} in Equation \eqref{eq:jump_term}   one gets
\begin{align} \label{eq:pi4}
	& \Pi_4\leq K \int_{0}^t\big(1+\sup_{i \in \{1,\ldots,N\}}\sup_{ r\in[0,s]}E|x_r^{i,N,n}|^{\bar p}\big)ds
\end{align}
for any $t\in[0,T]$ and $i\in\{1,\ldots,N\}$.
Then, by using Equations \eqref{eq:pi1},  \eqref{eq:pi2},  \eqref{eq:pi3},  \eqref{eq:pi4} and  \eqref{eq:pi4} in Equation \eqref{eq:mb_est} one gets for any  $t\in[0,T]$ and  $i\in\{1,\ldots,N\}$
\begin{align*}
	\sup_{i \in \{1,\ldots,N\}}\sup_{ r\in[0,t]} E|x_r^{i,N,n}|^{\bar p}\leq E|x_0^{i}|^{\bar p} +K+ K \int_{0}^t\sup_{i \in \{1,\ldots,N\}}\sup_{ r\in[0,s]}E|x_r^{i,N,n}|^{\bar p} ds
\end{align*}
 which on using Assumption \ref{asum:ic} and Gr\"onwall's inequality  completes the proof.
 \end{proof}

	\section{Rate of Convergence} \label{sec:rate}
	In this section, we prove the main result  stated in Theorem \ref{thm:mr}.
 For this purpose, first one needs to  establish  auxiliary results given below.

	The  following corollaries are consequences of Lemma \ref{lem:scm:mb}.
	\begin{cor} \label{cor:sig:gam}
		If Assumptions \mbox{\normalfont  \ref{asum:ic}}, \mbox{\normalfont  \ref{asum:lip}} and \mbox{\normalfont  \ref{asump:super}} to \mbox{\normalfont  \ref{asum:liningam}}  hold,
		 then for  $ 0\leq p_0\leq {\bar p}/{(\eta+1)}$
		\begin{align*}
E|\Btame|^{p_0}+E|\Sigtame|^{p_0}+\int_Z E|\Gamtame|^{p_0} \nu(d\bar z) \leq& K
		\end{align*}
		 for any   $s\in[0,T]$ and $i\in\{1,\ldots,N\}$   where $K>0$ is independent of $N$ and $n$.
	\end{cor}

	\begin{cor} \label{cor1:sig:gam}
		If Assumptions \mbox{\normalfont  \ref{asum:ic}}, \mbox{\normalfont  \ref{asum:lip}} and \mbox{\normalfont  \ref{asump:super}} to \mbox{\normalfont  \ref{asum:derv:liningam}}  hold, then for   $ 2\leq p_0\leq \min\{{\bar p}/((\eta+1)
  (\eta/8+1)),\bar p/2\}$
		\begin{align*}
			\sum_{\mathfrak q=1}^{2}E\big|\Sigq\big|^{p_0} 
			+ \sum_{\mathfrak q=1}^{2}\int_Z E\big|\Gamq\big|^{p_0}\nu(d\bar z)
			\leq &  
			Kn^{-1}
		\end{align*}
  for any   $s\in[0,T]$ and $i\in\{1,\ldots,N\}$   where $K>0$ is independent of $N$ and $n$.
	\end{cor}

	\begin{cor} \label{cor:one:step:error}
		If Assumptions \mbox{\normalfont  \ref{asum:ic}}, \mbox{\normalfont  \ref{asum:lip}} and \mbox{\normalfont  \ref{asump:super}} to \mbox{\normalfont  \ref{asum:derv:liningam}} hold, then  for $ 2\leq p_0\leq \min\{{\bar p}/((\eta+1)(\eta/8+1))\color{black},\bar p/2\}$
		\begin{align*}
			E|x_s^{i,N,n}-x_{\kappa_{n}(s)}^{i,N,n}|^{p_0}\leq Kn^{-1}
		\end{align*}
	  for any   $s\in[0,T]$ and $i\in\{1,\ldots,N\}$   where $K>0$ is independent of $N$ and $n$.
	\end{cor}

	For notational simplicity, let us introduce 
	 a $d\times m$ matrix $(\widehat{\Lambda})_{tam}^{\displaystyle n}$ whose $u\ell$-th  element is  
	\begin{align} 
		&\Lambdatameus:= \sigtame+\sum_{\mathfrak{q}=1}^2\sigq \notag
	\end{align}
 and a $d\times 1$ vector $(\widehat{\Gamma})_{tam}^{\displaystyle n}$  whose $u$-th  component is
\begin{align} 
	&\Gammatameus:= \gamtame+\sum_{\mathfrak{q}=1}^2\gamq \notag
\end{align}
for any  $s\in[0,T]$, $u\in\{1,\ldots,d\}$, $\ell\in\{1,\ldots,m\}$, $i\in\{1,\ldots,N\}$ and $\bar z\in Z$.

Now, let us rewrite the Milstein-type scheme \eqref{eq:scm} as follows
	\begin{align} \label{eq:scm**}
		x_t^{i, N,n}= x_{0}^{i}&+\int_{0}^{t}  \Btame ds 
		+\int_{0}^{t}  \Lambdatams dw_s^{i}    \notag
		\\
		& +\int_{0}^{t}\int_Z  \Gammatames \tilde n_p^i(ds,d\bar z)  
	\end{align}
	almost surely for any $t\in[ 0, T]$ and $i\in\{1,\ldots,N\}$.

\begin{lem} \label{lem:gam}
Let Assumptions \mbox{\normalfont  \ref{asum:ic}}, \mbox{\normalfont  \ref{asum:lip}} and \mbox{\normalfont  \ref{asump:super}} to \mbox{\normalfont  \ref{asum:convergence2}}  hold with 
   $\bar p\geq \max\{2,\eta/8+1\}(\varepsilon+2)(\eta+1)(\eta+2)/{\varepsilon}\color{black}$
   where $\varepsilon\in(0,1)$.
	 Then,
	\begin{align*}
		&\int_Z E\big|\gamma\big( x_s^{i, N,n},\mu_s^{x, N,n}, z\big)-\Gammaztames\big|^2\nu(d z)\leq Kn^{-1-\frac{2}{\varepsilon+2}}
	\end{align*}
	for any $s\in[0,T]$ and $i\in\{1,\ldots,N\}$	where $K>0$ is a constant independent of $N$ and $n$.
\end{lem} 
\begin{proof}
	First,  recall   the Milstein-type scheme \eqref{eq:scm**} and  use It\^o's formula in Equation  \eqref{eq:ito} (Lemma \ref{lem:ito}) to get
	\begin{align*} 
		&\gamma^{u}\big( x_s^{i, N,n},\mu_s^{x, N,n}, z\big) =\gamma^{u}\big(x_{\kappa_n(s)}^{i,N,n},\mu_{\kappa_n(s)}^{x,N,n}, z\big) +\int_{\kappa_n(s)}^s\partial_x\gamma^{u}\big(x_{r}^{i,N,n},\mu_{r}^{x,N,n},z\big)\Btamer dr \notag
		\\
		&-\int_{\kappa_n(s)}^s\int_Z \partial_x\gamma^{u} \big(x_{r}^{i,N,n},\mu_{r}^{x,N,n},z\big)\Gammatame\nu(d\bar z)dr\notag
		\\
		&+\frac{1}{N}\sum_{k= 1}^{N}\int_{\kappa_n(s)}^s\partial_\mu\gamma^{u}\big( x_{r}^{i,N,n},\mu_{r}^{x,N,n}, x_{r}^{k,N,n},z\big)\Btamekr dr \notag
		\\
		&-\frac{1}{N}\sum_{k=1}^{N}\int_{\kappa_n(s)}^s\int_Z \partial_\mu\gamma^{u} \big(x_{r}^{i,N,n},\mu_{r}^{x,N,n},x_{r}^{k,N,n},z\big)\Gammatamek\nu(d\bar z)dr\notag
		\\
		&+\sum_{\ell_1=1}^{m}\int_{\kappa_n(s)}^s \partial_x\gamma^{u}\big( x_{r}^{i,N,n},\mu_{r}^{x,N,n},z\big)\Lambdatame dw^{i\ell_1}_r\notag
		\\
		&+\frac{1}{N}\sum_{\ell_1=1}^m\sum_{k= 1}^{N}\int_{\kappa_n(s)}^s\partial_\mu\gamma^{u}\big( x_{r}^{i,N,n},\mu_{r}^{x,N,n},x_{r}^{k,N,n},z\big)\Lambdatamek dw^{k\ell_1}_r\notag
		\\
		& + \frac{1}{2}\int_{\kappa_n(s)}^{s} \tr\big[\partial_{x}^2 \gamma^{u}\big(x_r^{i,N,n},\mu_r^{x,N,n},z\big) \notag
		\\
		&\qquad\qquad\qquad\qquad
		\times
  \Lambdatam\Lambdatamtrans\big]  dr\notag
		\\
		&+\frac{1}{N} \int_{\kappa_n(s)}^{s} \tr\big[\partial_{x}\partial_{\mu} \gamma^{u}\big(x_r^{i,N,n},\mu_r^{x,N,n},x_r^{i,N,n},z\big) \notag
		\\
		&\qquad\qquad\qquad\qquad
		\times
		\Lambdatam\Lambdatamtrans\big] dr \notag
		\\
		&+\frac{1}{2N}\sum_{k=1}^N   \int_{\kappa_n(s)}^{s} \tr\big[\partial_{y}\partial_{\mu} \gamma^{u}\big(x_r^{i,N,n},\mu_r^{x,N,n},x_r^{k,N,n},z\big) \notag
		\\
		&\qquad\qquad\qquad\qquad
		\times
		\Lambdatamk\Lambdatamktrans\big]dr \notag
		\\
		&+\frac{1}{2N^2}\sum_{k=1}^N \int_{\kappa_n(s)}^{s} \tr\big[\partial_{\mu}^2 \gamma^{u}\big(x_r^{i,N,n},\mu_r^{x,N,n},x_r^{k,N,n},x_r^{k,N,n},z\big) \notag
		\\
		&\qquad\qquad\qquad\qquad
		\times\Lambdatamk\Lambdatamktrans\big]dr \notag
		\\
		&+\sum_{k=1}^{N}\int_{\kappa_{n}(s)}^{s}\int_Z\Big\{\gamma^{u}\big(x_{r-}^{i,N,n}+1_{\{k=i\}}\Gammatamek,\tilpikappaGammatameve,z\big) \notag
		\\
		&\qquad\qquad\qquad\qquad
		-\gamma^{u}\big( x_{r-}^{i,N,n},\mu_{r-}^
		{x,N,n},z\big)\Big\} n_p^k(dr,d\bar z)
	\end{align*}
	almost surely for any   $s\in[0,T]$, $u\in\{1,\ldots,d\}$, $z\in Z$ and $i\in\{1,\ldots,N\}$	 where for any   $r\in[0,T]$, $\bar z\in Z$ and $k\in\{1,\ldots,N\}$
	\begin{align} \label{eq:empirical}
		\tilpikappaGammatame:=\empiri.
	\end{align} \color{black}
Also,  recall   $\gammaonetameu$ and $\gammatwotameu$   from Equations  \eqref{eq:gamhat1} and  \eqref{eq:gamhat2},  respectively  to obtain
	\begin{align}  \label{eq:lam}
		&\int_ZE\big|\gamma^{u}\big( x_s^{i, N,n},\mu_s^{x, N,n},z\big)-\Gammaztameu\big|^2 \nu(dz) \notag
		\\
		&\leq K\int_ZE\big|\gamma^{u}\big(x_{\kappa_n(s)}^{i,N,n},\mu_{\kappa_n(s)}^{x,N,n},z\big)-\gamtamez\big|^2\nu(dz)\notag
		\\
		&\quad+K\int_ZE\Big|\int_{\kappa_n(s)}^s\Big\{\partial_x\gamma^{u}\big(x_{r}^{i,N,n},\mu_{r}^{x,N,n},z\big)\Btamer    \notag
		\\
		&\qquad\qquad+\frac{1}{N}\sum_{k= 1}^{N}\partial_\mu\gamma^{u}\big(x_{r}^{i,N,n},\mu_{r}^{x,N,n},x_{r}^{k,N,n},z\big)\Btamekr\Big\}  dr\Big|^2 \nu(dz) \notag
		\\
		&\quad+K\int_ZE\Big|\int_{\kappa_n(s)}^s\int_Z \Big\{ \partial_x\gamma^{u}\big(x_{r}^{i,N,n},\mu_{r}^{x,N,n},z\big)\Gammatame\notag
		\\
		&\qquad\qquad+\frac{1}{N}\sum_{k=1}^{N} \partial_\mu\gamma^{u}\big(x_{r}^{i,N,n},\mu_{r}^{x,N,n},x_{r}^{k,N,n},z\big)\Gammatamek\Big\}\nu(d\bar z)dr\Big|^2 \nu(dz) \notag
		\\
		&\quad+K\int_ZE\Big|\sum_{\ell_1=1}^{m}\int_{\kappa_n(s)}^s \Big\{\partial_x\gamma^{u}\big(x_{r}^{i,N,n},\mu_{r}^{x,N,n},z\big)\Sigtameloner \notag
		\\
		&\qquad\qquad\qquad\qquad\qquad -\gamzxsigtamer\Big\} dw^{i\ell_1}_r\Big|^2 \nu(dz) \notag
		\\
		&+K\int_ZE\Big|\frac{1}{N}\sum_{\ell_1=1}^m\sum_{k= 1}^{N}\int_{\kappa_n(s)}^s\Big\{\partial_\mu\gamma^{u}\big(x_{r}^{i,N,n},\mu_{r}^{x,N,n},x_{r}^{k,N,n},z\big)\Sigtamelonekr \notag
		\\
		&\qquad\qquad\qquad\qquad\qquad-\gamzmusigtamer\Big\}dw^{k\ell_1}_r\Big|^2 \nu(dz) \notag
		\\
		&+K\int_ZE\Big|\sum_{\ell_1=1}^m\int_{\kappa_n(s)}^s \partial_x\gamma^{u}\big(x_{r}^{i,N,n},\mu_{r}^{x,N,n},z\big)\sum_{\mathfrak{q}=1}^{2}\sigqloner dw^{i\ell_1}_r\notag
		\\
		&\qquad+\frac{1}{N}\sum_{\ell_1=1}^m\sum_{k= 1}^{N}\int_{\kappa_n(s)}^s \partial_\mu\gamma^{u} \big(x_{r}^{i,N,n},\mu_{r}^{x,N,n},x_{r}^{k,N,n},z\big)\sum_{\mathfrak{q}=1}^{2}\sigqlonekr dw^{k\ell_1}_r\Big|^2 \nu(dz) \notag
		\\
		& + K\int_ZE\Big|\int_{\kappa_n(s)}^{s}\Big\{ 
		\big| \partial_x^2\gamma^{u}\big(x_{r}^{i,N,n},\mu_{r}^{x,N,n},z\big)+\partial_x\partial_\mu\gamma^{u}\big(x_{r}^{i,N,n},\mu_{r}^{x,N,n},x_{r}^{i,N,n},z\big)\big|\notag
		\\
		&\qquad\qquad\qquad\qquad\times\big|\Lambdatam\big|^2 \notag
		\\
		&\qquad+\frac{1}{N}\sum_{k=1}^N    \big|\partial_y\partial_\mu\gamma^{u}\big(x_{r}^{i,N,n},\mu_{r}^{x,N,n},x_{r}^{k,N,n},z\big)+\partial_\mu^2\gamma^{u}\big(x_{r}^{i,N,n},\mu_{r}^{x,N,n},x_{r}^{k,N,n},z\big)\big| \notag
		\\
		&\qquad\qquad\qquad\qquad\times\big|\Lambdatamk\big|^2 \notag
		\Big\}dr\Big|^2 \nu(dz) \notag
		\\
		&+K\int_ZE\Big|\sum_{k=1}^{N}\int_{\kappa_{n}(s)}^{s}\int_Z\Big\{\gamma^{u}\big(x_{r-}^{i,N,n}+1_{\{k=i\}}\Gammatamek,\tilpikappaGammatameve,z\big) \notag
		\\
		&
		\qquad\qquad\qquad\qquad\qquad -\gammatameu\big(x_{\kappa_n(r)}^{i,N,n}+1_{\{k=i\}}\Gamtamekzbarr,\tilmukappatamezbarr,z\big) \notag
		\\
		& \qquad\qquad\qquad\qquad\qquad -\gamma^{u}\big( x_{r-}^{i,N,n},\mu_{r-}^{x,N,n},z\big)+\gammatameu\big( x_{\kappa_n(r)}^{i,N,n},\mu_{\kappa_{n}(r)}^{x,N,n},z\big)\Big\}n_p^k(dr,d\bar z) \Big|^2 \nu(dz)\notag
		\\
&=:\lambda_1+\lambda_2+\lambda_3+\lambda_4+\lambda_5+\lambda_6+\lambda_7+\lambda_8
	\end{align}   
	for any   $s\in[0,T]$, $u\in\{1,\ldots,d\}$ and $i\in\{1,\ldots,N\}$.
	
	For $\lambda_1$,   apply Assumption \ref{asum:convergence}  to obtain
	\begin{align} \label{eq:lam1}
		\lambda_1:= K\int_ZE\big|\gamma^{u}\big(x_{\kappa_n(s)}^{i,N,n},\mu_{\kappa_n(s)}^{x,N,n},z\big)-\gamtamez\big|^2\nu(dz) \leq K n^{-1-\frac{2}{\varepsilon+2}}
	\end{align}
	for any $s\in[0,T]$, $u\in\{1,\ldots,d\}$ and $i\in\{1,\ldots,N\}$.
	To estimate $\lambda_{2}$, one uses H\"older's inequality, Remark \ref{rem:super_linear}, Assumption 
	\ref{asum:measure:derv:bound}, Young's inequality along with Corollary \ref{cor:sig:gam} and
	Lemma \ref{lem:scm:mb}   as follows
	\begin{align} 	\label{eq:lam2}
		\lambda_2:=& K\int_ZE\Big|\int_{\kappa_n(s)}^s\Big\{\partial_x\gamma^{u}\big(x_{r}^{i,N,n},\mu_{r}^{x,N,n},z\big)\Btamer    \notag
		\\
		&\qquad\qquad+\frac{1}{N}\sum_{k= 1}^{N}\partial_\mu\gamma^{u}\big(x_{r}^{i,N,n},\mu_{r}^{x,N,n},x_{r}^{k,N,n},z\big)\Btamekr\Big\}  dr\Big|^2 \nu(dz)  \notag
		\\
		\leq &  Kn^{-1}     E\int_{\kappa_n(s)}^s\Big\{\big(1+|x_r^{i,N,n}|\big)^{\eta}\big|\Btamer\big|^2 +\frac{1}{N} \sum_{k=1}^{N}\big|\Btamekr\big|^2\Big\} dr  \notag
		\\
		\leq &  Kn^{-1}     \int_{\kappa_n(s)}^s\Big\{E\big(1+|x_r^{i,N,n}|\big)^{3\eta}+E\big|\Btamer\big|^3\Big\} dr+Kn^{-2}  
		\leq Kn^{-2}
	\end{align}
	for any $s\in[0,T]$, $u\in\{1,\ldots,d\}$ and $i\in\{1,\ldots,N\}$.
	For $\lambda_3$, by applying H\"older's inequality, Remark \ref{rem:super_linear}, Assumption 
	\ref{asum:measure:derv:bound}, Young's inequality, Corollaries \ref{cor:sig:gam},  \ref{cor1:sig:gam} and
	Lemma  \ref{lem:scm:mb}, one gets
	\begin{align} \label{eq:lam3}
		&\lambda_3:=K\int_ZE\Big|\int_{\kappa_n(s)}^s\int_Z \Big\{ \partial_x\gamma^{u}\big(x_{r}^{i,N,n},\mu_{r}^{x,N,n},z\big)\Gammatame\notag
		\\
		&\qquad\qquad+\frac{1}{N}\sum_{k=1}^{N} \partial_\mu\gamma^{u}\big(x_{r}^{i,N,n},\mu_{r}^{x,N,n},x_{r}^{k,N,n},z\big)\Gammatamek\Big\}\nu(d\bar z)dr\Big|^2 \nu(dz) \notag
		\\
		\leq &  Kn^{-1}     E\int_{\kappa_n(s)}^s\int_Z\Big\{\big(1+|x_r^{i,N,n}|\big)^{\eta}\big|\Gammatame\big|^2 \notag
		\\
		&\qquad\qquad\qquad\qquad+\frac{1}{N} \sum_{k=1}^{N}\big|\Gammatamek\big|^2\Big\} \nu(d\bar z) dr  \notag
		\\
		\leq &  Kn^{-1}   \int_{\kappa_n(s)}^s\Big\{E\big(1+|x_r^{i,N,n}|\big)^{3\eta}+\int_ZE\big|\Gammatame\big|^{3}\nu(d\bar z)\Big\}  dr +Kn^{-2} \notag
  \\
		\leq & Kn^{-2}
	\end{align}
	for any   $s\in[0,T]$, $u\in\{1,\ldots,d\}$ and  $i\in\{1,\ldots,N\}$.	
	To evaluate $\lambda_4$,  one  recalls operator ${\mathfrak{D}}_x^{\sigma^{\ell_1}}\gamma^{u}$ from Section \ref{sec:main_result} to write
	\begin{align*} 
		\lambda_4:=&K\int_ZE\Big|\sum_{\ell_1=1}^{m}\int_{\kappa_n(s)}^s \Big\{\partial_x\gamma^{u}\big(x_{r}^{i,N,n},\mu_{r}^{x,N,n},z\big)\Sigtameloner \notag
		\\
		&\qquad\qquad\qquad\qquad\qquad -\gamzxsigtamer\Big\} dw^{i\ell_1}_r\Big|^2 \nu(dz) \notag
		\\
		\leq &K \sum_{\ell_1=1}^{m}\int_Z\int_{\kappa_n(s)}^s E\Big|\big\{\partial_x\gamma^{u}\big( x_{r}^{i,N,n},\mu_{r}^{x,N,n},z\big)-\partial_x\gamma^{u}\big( x_{\kappa_{n}(r)}^{i,N,n},\mu_{\kappa_{n}(r)}^{x,N,n},z\big)\big\}\Sigtameloner \notag
		\\ 
		&\qquad+\partial_x\gamma^{u}\big( x_{\kappa_{n}(s)}^{i,N,n},\mu_{\kappa_{n}(s)}^{x,N,n},z\big)\big\{\Sigtameloner-\Sigloner\big\}
		\notag
		\\
		&\qquad +\gamzxsigr-\gamzxsigtamer\Big|^2 dr\nu(dz) \notag
	\end{align*}
	 which on    applying   Assumption \ref{asum:dlip}, Remark \ref{rem:super_linear}, H\"older's inequality and Equation \eqref{eq:empirical_dif} along with Assumptions  \ref{asum:convergence}, \ref{asum:derv:convergence}, Lemma \ref{lem:scm:mb}, and Corollaries \ref{cor:sig:gam}, \ref{cor:one:step:error} yield  
	\begin{align} \label{eq:lam4}
		&\color{black}\lambda_4 \leq  K\sum_{\ell_1=1}^{m}\int_{\kappa_n(s)}^sE\Big[\big\{(1+|x_{r}^{i,N,n}|+|x_{\kappa_n(r)}^{i,N,n}|)^{\eta-2}|x_{r}^{i,N,n}-x_{\kappa_n(r)}^{i,N,n}|^2+\mathcal W_2^2(\mu_{r}^{x,N,n},\mu_{\kappa_{n}(r)}^{x,N,n})\big\} \notag
		\\
		&\color{black}\qquad\qquad\qquad\qquad\times\big|\Sigtameloner 
		\big|^2\Big]dr \notag
		\\
		&\color{black}+K\sum_{\ell_1=1}^{m}E\int_{\kappa_n(s)}^s(1+x_{\kappa_n(r)}^{i,N,n})^{\eta}\big|\Sigtameloner-\Sigloner\big|^2dr \notag
		\\
		& \color{black}+K \sum_{\ell_1=1}^{m}\int_Z\int_{\kappa_n(s)}^s E\Big|\gamzxsigr-\gamzxsigtamer\Big|^2 dr \nu(dz)   
  \notag
	\\
	\leq& K\sum_{\ell_1=1}^{m}\int_{\kappa_n(s)}^s\Big[\big\{E(1+|x_{r}^{i,N,n}|+|x_{\kappa_n(r)}^{i,N,n}|)^{\frac{2(\eta-2)(\varepsilon+2)}{\varepsilon}}\big\}^{\frac{\varepsilon}{2(\varepsilon+2)}} \notag
	\\
	&\qquad\times \big\{E|x_{r}^{i,N,n}-x_{\kappa_n(r)}^{i,N,n}|^{\varepsilon+2}\big\}^{\frac{2}{\varepsilon+2}} \big\{E\big|\Sigtameloner 
	\big|^{\frac{4(\varepsilon+2)}{\varepsilon}}\big\}^{\frac{\varepsilon}{2(\varepsilon+2)}} \notag
	\\
	&+ \Big\{\frac{1}{N}\sum_{j=1}^N E|x_{r}^{j,N,n}-x_{\kappa_n(r)}^{j,N,n}|^{\varepsilon+2}\Big\}^{\frac{2}{\varepsilon+2}} \big\{E\big|\Sigtameloner 
	\big|^{\frac{2(\varepsilon+2)}{\varepsilon}}\big\}^{\frac{\varepsilon}{\varepsilon+2}}\Big] dr \notag
	\\
	& +K\sum_{\ell_1=1}^{m}\int_{\kappa_n(s)}^s\Big[ \big\{E(1+x_{\kappa_n(r)}^{i,N,n})^{4\eta}\big\}^{\frac{1}{4}} \Big\{E\big|\Sigtameloner-\Sigloner\big|^4\Big\}^{\frac{1}{4}} \notag
	\\
	&\qquad
	\times \Big\{E\big|\Sigtameloner-\Sigloner\big|^2\Big\}^{\frac{1}{2}}\Big] dr +Kn^{-1-\frac{2}{\varepsilon+2}}   \notag
	\\
	\leq & K\int_{\kappa_n(s)}^s\big\{\sup_{i \in \{1,\ldots,N\}}\sup_{ r\in[0,s]}E|x_{r}^{i,N,n}-x_{\kappa_n(r)}^{i,N,n}|^{\varepsilon+2}\big\}^{\frac{2}{\varepsilon+2}} dr  +Kn^{-1-\frac{2}{\varepsilon+2}}+Kn^{-\frac{3}{2}-\frac{1}{\varepsilon+2}}
		\leq K n^{-1-\frac{2}{\varepsilon+2}}
	\end{align}
	  for any  $s\in[0,T]$, $u\in\{1,\ldots,d\}$ and $i\in\{1,\ldots,N\}$. \color{black}
	For $\lambda_5$, by recalling operator ${\mathfrak{D}}_\mu^{\sigma^{\ell_1}}\gamma^{u}$ from Section \ref{sec:main_result}, one  obtains 
	\begin{align*}
		\lambda_5:=&K\int_ZE\Big|\frac{1}{N}\sum_{\ell_1=1}^m\sum_{k= 1}^{N}\int_{\kappa_n(s)}^s\Big\{\partial_\mu\gamma^{u}\big(x_{r}^{i,N,n},\mu_{r}^{x,N,n},x_{r}^{k,N,n},z\big)\Sigtamelonekr \notag
		\\
		&\qquad\qquad\qquad\qquad\qquad-\gamzmusigtamer\Big\}dw^{k\ell_1}_r\Big|^2 \nu(dz) \notag
		\\
		\leq &\frac{K}{N} \sum_{\ell_1=1}^{m}\sum_{k= 1}^{N}\int_Z\int_{\kappa_n(s)}^s E\Big|\big\{\partial_\mu\gamma^{u}\big( x_{r}^{i,N,n},\mu_{r}^{x,N,n},x_{r}^{k,N,n},z\big)-\partial_\mu\gamma^{u}\big( x_{\kappa_{n}(r)}^{i,N,n},\mu_{\kappa_{n}(r)}^{x,N,n},x_{\kappa_{n}(r)}^{k,N,n},z\big)\big\} \notag
		\\ 
		&\qquad\qquad\qquad\times \Sigtamelonekr
		\\
		&\qquad+\partial_\mu\gamma^{u}\big( x_{\kappa_{n}(r)}^{i,N,n},\mu_{\kappa_{n}(r)}^{x,N,n},x_{\kappa_{n}(r)}^{k,N,n},z\big)\big\{\Sigtamelonekr-\Siglonekr\big\}
		\notag
		\\
		& \qquad+\gamzmusigr-\gamzmusigtamer\Big|^2 dr\nu(dz)   	 \notag
	\end{align*}
	 which on using Assumptions  \ref{asum:dlip} and  \ref{asum:measure:derv:bound}, H\"older's inequality and Equation \eqref{eq:empirical_dif} together with Assumptions  \ref{asum:convergence}, \ref{asum:derv:convergence}, Lemma \ref{lem:scm:mb}, and Corollaries \ref{cor:sig:gam}, \ref{cor:one:step:error} give 
	\begin{align} \label{eq:lam5}
		&\color{black}\lambda_5
		\leq \frac{K}{N} \sum_{\ell_1=1}^{m}\sum_{k= 1}^{N}\int_{\kappa_n(s)}^sE\Big[\Big\{(1+|x_{r}^{i,N,n}|+|x_{\kappa_n(r)}^{i,N,n}|)^{\eta}|x_{r}^{i,N,n}-x_{\kappa_n(r)}^{i,N,n}|^2+|x_{r}^{k,N,n}-x_{\kappa_n(r)}^{k,N,n}|^2 \notag
		\\
		&\color{black}\qquad\qquad\qquad\qquad+\mathcal W_2^2(\mu_{r}^{x,N,n},\mu_{\kappa_{n}(r)}^{x,N,n})\Big\}\big|\Sigtamelonekr 
		\big|^2\Big]dr \notag
		\\
		&\color{black}+\frac{K}{N} \sum_{\ell_1=1}^{m}\sum_{k= 1}^{N}\int_{\kappa_n(s)}^sE\big|\Sigtamelonekr-\Siglonekr\big|^2 dr \notag
		\\
		&\color{black} +\frac{K}{N} \sum_{\ell_1=1}^{m}\sum_{k= 1}^{N}\int_Z\int_{\kappa_n(s)}^s E\Big|\gamzmusigr	\notag
		\\
		&\color{black}\qquad\qquad-\gamzmusigtamer\Big|^2 dr \nu(dz) 
  \notag
		\\
		\leq&  \frac{K}{N} \sum_{\ell_1=1}^{m}\sum_{k= 1}^{N}\int_{\kappa_n(s)}^s\Big[\big\{E(1+|x_{r}^{i,N,n}|+|x_{\kappa_n(r)}^{i,N,n}|)^{\frac{2\eta(\varepsilon+2)}{\varepsilon}}\big\}^{\frac{\varepsilon}{2(\varepsilon+2)}} \notag
		\\
		&\qquad\times \big\{E|x_{r}^{i,N,n}-x_{\kappa_n(r)}^{i,N,n}|^{\varepsilon+2}\big\}^{\frac{2}{\varepsilon+2}} \big\{E\big|\Sigtamelonekr
		\big|^{\frac{4(\varepsilon+2)}{\varepsilon}}\big\}^{\frac{\varepsilon}{2(\varepsilon+2)}} \notag
		\\
		&+\Big\{E|x_{r}^{k,N,n}-x_{\kappa_n(r)}^{k,N,n}|^{\varepsilon+2}+\frac{1}{N}\sum_{j=1}^N E|x_{r}^{j,N,n}-x_{\kappa_n(r)}^{j,N,n}|^{\varepsilon+2}\Big\}^{\frac{2}{\varepsilon+2}}    \notag
		\\
		&\qquad\times \big\{E\big|\Sigtamelonekr
		\big|^{\frac{2(\varepsilon+2)}{\varepsilon}}\big\}^{\frac{\varepsilon}{\varepsilon+2}}\Big] dr 
         +Kn^{-1-\frac{2}{\varepsilon+2}} \notag
		\\
		\leq & Kn^{-1-\frac{2}{\varepsilon+2}} +K\int_{\kappa_n(s)}^s\big\{\sup_{i \in \{1,\ldots,N\}}\sup_{ r\in[0,s]}E|x_{r}^{i,N,n}-x_{\kappa_n(r)}^{i,N,n}|^{\varepsilon+2}\big\}^{\frac{2}{\varepsilon+2}} dr 
		\leq K n^{-1-\frac{2}{\varepsilon+2}}
	\end{align}
	 for any  $s\in[0, T]$, $u\in\{1,\ldots,d\}$ and $i\in\{1,\ldots, N\}$.  \color{black}
	To estimate $\lambda_{6}$, one uses Remark \ref{rem:super_linear}, Assumption \ref{asum:measure:derv:bound} along with   H\"older's inequality,  Corollary \ref{cor1:sig:gam} and Lemma \ref{lem:scm:mb}
 \begin{align} \label{eq:lam6}
		\lambda_6:= &K\int_ZE\Big|\sum_{\ell_1=1}^m\int_{\kappa_n(s)}^s \partial_x\gamma^{u}\big(x_{r}^{i,N,n},\mu_{r}^{x,N,n},z\big)\sum_{\mathfrak{q}=1}^{2}\sigqloner dw^{i\ell_1}_r\notag
		\\
		&+\frac{1}{N}\sum_{\ell_1=1}^m\sum_{k= 1}^{N}\int_{\kappa_n(s)}^s \partial_\mu\gamma^{u} \big(x_{r}^{i,N,n},\mu_{r}^{x,N,n},x_{r}^{k,N,n},z\big)\notag
		\\
		&\qquad\qquad\qquad\times \sum_{\mathfrak{q}=1}^{2}\sigqlonekr dw^{k\ell_1}_r\Big|^2 \nu(dz) \notag
		\\
		\leq &K\hspace{-0.05cm}\sum_{\ell_1=1}^m\hspace{-0.05cm}\int_Z\hspace{-0.05cm}E\int_{\kappa_n(s)}^s \hspace{-0.05cm}\Big|\partial_x\gamma^{u}\big(x_{r}^{i,N,n},\mu_{r}^{x,N,n},z\big)\sum_{\mathfrak{q}=1}^{2}\sigqloner\Big|^2 dr\nu(dz)\notag
		\\
		&+\frac{1}{N}\sum_{\ell_1=1}^m\sum_{k= 1}^{N}\int_ZE\int_{\kappa_n(s)}^s \Big|\partial_\mu\gamma^{u} \big(x_{r}^{i,N,n},\mu_{r}^{x,N,n},x_{r}^{k,N,n},z\big)\notag
		\\
		&\qquad\qquad\qquad\times\sum_{\mathfrak{q}=1}^{2}\sigqlonekr\Big|^2 dr \nu(dz) \notag
		\\
		\leq &  K\sum_{\ell_1=1}^mE\int_{\kappa_n(s)}^s (1+|x_{r}^{i,N,n}|)^{\eta} \sum_{\mathfrak{q}=1}^{2}\big|\sigqloner\big|^2  dr\notag
		\\
		&\qquad\qquad+\frac{K}{N}\sum_{\ell_1=1}^m\sum_{k= 1}^{N}\int_{\kappa_n(s)}^s \sum_{\mathfrak{q}=1}^{2}E\big|\sigqlonekr\big|^2 dr\notag
		\\
		\leq &  K\sum_{\ell_1=1}^m\int_{\kappa_n(s)}^s \big\{E(1+|x_{r}^{i,N,n}|)^{\frac{\eta(\varepsilon+2)}{\varepsilon}}\big\}^{\frac{\varepsilon}{\varepsilon+2}} \notag
		\\
		&\times \Big\{\sum_{\mathfrak{q}=1}^{2}E\big|\sigqloner\big|^{\varepsilon+2}\Big\}^{\frac{2}{\varepsilon+2}}  dr+Kn^{-2}\leq Kn^{-1-\frac{2}{\varepsilon+2}} 
	\end{align}
	for any  $s\in[0,T]$, $u\in\{1,\ldots,d\}$ and $i\in\{1,\ldots,N\}$. 
	For $\lambda_7$, by applying H\"older's inequality, Remark \ref{rem:super_linear}, Assumption \ref{asum:second:measure:derv}, Corollaries \ref{cor:sig:gam}, \ref{cor1:sig:gam} and Lemma  \ref{lem:scm:mb},  one gets
	\begin{align} \label{eq:lam7}
\lambda_7:=&K\int_ZE\Big|\int_{\kappa_n(s)}^{s}\Big\{ 
		\big| \partial_x^2\gamma^{u}\big(x_{r}^{i,N,n},\mu_{r}^{x,N,n},z\big)+\partial_x\partial_\mu\gamma^{u}\big(x_{r}^{i,N,n},\mu_{r}^{x,N,n},x_{r}^{i,N,n},z\big)\big|\notag
		\\
		&\qquad\qquad\qquad\qquad\times\big|\Lambdatamk\big|^2 \notag
		\\
		&\qquad+\frac{1}{N}\sum_{k=1}^N    \big|\partial_y\partial_\mu\gamma^{u}\big(x_{r}^{i,N,n},\mu_{r}^{x,N,n},x_{r}^{k,N,n},z\big)+\partial_\mu^2\gamma^{u}\big(x_{r}^{i,N,n},\mu_{r}^{x,N,n},x_{r}^{k,N,n},z\big)\big| \notag
		\\
		&\qquad\qquad\qquad\qquad\times\big|\Lambdatamk\big|^2 \notag
		\Big\}dr\Big|^2 \nu(dz) \notag
		\\
		\leq &  Kn^{-1}E\int_{\kappa_n(s)}^s (1+|x_{r}^{i,N,n}|)^{\eta} \big|\Lambdatam\big|^4  dr\notag
		\\
		&+  \frac{K}{N}n^{-1}\sum_{k=1}^N \int_{\kappa_n(s)}^s  E\big|\Lambdatamk\big|^4  dr\notag 
  	\\
	\leq &  Kn^{-1}\int_{\kappa_n(s)}^s \big\{E(1+|x_{r}^{i,N,n}|)^{\frac{\eta(\varepsilon+2)}{\epsilon}}\big\}^{\frac{\varepsilon}{\varepsilon+2}}\big\{E\big|\Lambdatam\big|^{\frac{4(\varepsilon+2)}{2}} \big\}^{\frac{2}{\varepsilon+2}}  dr+Kn^{-2} \notag
 \\
		\leq & Kn^{-2}
	\end{align}
	for any  $s\in[0,T]$, $u\in\{1,\ldots,d\}$ and $i\in\{1,\ldots,N\}$. 
	 To evaluate $\lambda_{8}$, one can proceed as follows
	\begin{align} \label{eq:lam8*}
		\color{black}\lambda_{8}:=& \color{black}K\int_ZE\Big|\sum_{k=1}^{N}\int_{\kappa_{n}(s)}^{s}\int_Z\Big\{\gamma^{u}\big(x_{r-}^{i,N,n}+1_{\{k=i\}}\Gammatamek,\tilpikappaGammatameve,z\big) \notag
		\\
		&
		\color{black}\qquad\qquad\qquad\qquad -\gammatameu\big(x_{\kappa_n(r)}^{i,N,n}+1_{\{k=i\}}\Gamtamekzbarr,\tilmukappatamezbarr,z\big) \notag
		\\
		& \color{black}\qquad\qquad\qquad\qquad -\gamma^{u}\big( x_{r-}^{i,N,n},\mu_{r-}^{x,N,n},z\big)+\gammatameu\big( x_{\kappa_n(r)}^{i,N,n},\mu_{\kappa_{n}(r)}^{x,N,n},z\big)\Big\}n_p^k(dr,d\bar z) \Big|^2 \nu(dz)
  \notag
				\\
				=& K\int_Z E\Big|\int_{\kappa_{n}(s)}^{s}\int_Z\Big\{\gamma^{u}\big(x_{r-}^{i,N,n}+\Gammatame,\tilpikappaGammatameive,z\big) \notag
				\\
				&
				 \qquad\qquad-\gammatameu\big(x_{\kappa_n(r)}^{i,N,n}+\Gamtamer,\tilmukappatameizbarr,z\big) \notag
				\\
				& \qquad\qquad -\gamma^{u}\big( x_{r-}^{i,N,n},\tilpikappaGammatameive,z\big)+\gammatameu\big( x_{\kappa_n(r)}^{i,N,n},\tilmukappatameizbarr,z\big)\Big\}n_p^i(dr,d\bar z) \notag
				\\
				& +\int_{\kappa_{n}(s)}^{s}\int_Z\Big\{\gamma^{u}\big( x_{r-}^{i,N,n},\tilpikappaGammatameive,z\big) - \gammatameu\big( x_{\kappa_n(r)}^{i,N,n},\tilmukappatameizbarr,z\big)\Big\} n_p^i(dr,d\bar z) \notag
				\\
				& +\hspace{-0.03cm}\sum_{k\neq i}\hspace{-0.03cm}\int_{\kappa_{n}(s)}^{s}\hspace{-0.07cm}\int_Z\Big\{\gamma^{u}\big(x_{r-}^{i,N,n},\tilpikappaGammatameve,z\big) \hspace{-0.05cm}-\hspace{-0.05cm}\gammatameu\big(x_{\kappa_n(r)}^{i,N,n},\tilmukappatamezbarr,z\big)\Big\} n_p^k(dr,d\bar z)\notag
				\\
				&
				 \qquad\qquad -\sum_{k=1 }^N\int_{\kappa_{n}(s)}^{s}\int_Z\Big\{\gamma^{u}\big( x_{r-}^{i,N,n},\mu_{r-}^{x,N,n},z\big)
				 -\gammatameu\big(x_{\kappa_n(r)}^{i,N,n},\mu_{\kappa_{n}(r)}^{x,N,n},z\big)\Big\}n_p^k(dr,d\bar z) \Big|^2 \nu(dz) \notag
		\\
		\color{black}\leq & \color{black}K\int_ZE\int_{\kappa_{n}(r)}^{s}\int_Z\Big|\gamma^{u}\big(x_{r}^{i,N,n}+\Gammatame,\tilpikappaGammatamei,z\big) \notag
		\\
		&
		\color{black}\qquad\qquad -\gammatameu\big(x_{\kappa_n(r)}^{i,N,n}+\Gamtamer,\tilmukappatameizbarr,z\big) \notag
		\\
		& \color{black}\qquad\qquad -\gamma^{u}\big( x_{r}^{i,N,n},\tilpikappaGammatamei,z\big)+\gammatameu\big( x_{\kappa_n(r)}^{i,N,n},\tilmukappatameizbarr,z\big)\Big|^2 \nu(d\bar z) dr \nu(dz)\notag
		\\
		& \color{black}+KN\sum_{k=1}^N\int_ZE\int_{\kappa_{n}(s)}^{s}\int_Z\Big|\gamma^{u}\big(x_r^{i,N,n},\tilpikappaGammatame,z\big)-\gammatameu\big(x_{\kappa_n(r)}^{i,N,n},\tilmukappatamezbarr,z\big) \notag
		\\
		&
		\color{black}\qquad \qquad -\gamma^{u}\big( x_r^{i,N,n},\mu_r^{x,N,n},z\big)+\gammatameu\big( x_{\kappa_n(r)}^{i,N,n},\mu_{\kappa_{n}(r)}^{x,N,n},z\big) \Big|^2 \nu(d\bar z)  dr \nu(dz)
		:= \lambda_{81} +\lambda_{82}
	\end{align}
	for any  $s\in[0,T]$, $u\in\{1,\ldots,d\}$ and $i\in\{1,\ldots,N\}$. \color{black}
	For $\lambda_{81}$, use Assumption \ref{asum:convergence1} to obtain
	\begin{align*}
		\lambda_{81}:=& K\int_ZE\int_{\kappa_{n}(s)}^{s}\int_Z\Big|\gamma^{u}\big(x_{r}^{i,N,n}+\Gammatame,\tilpikappaGammatamei,z\big) \notag
		\\
		&
		\qquad -\gammatameu\big(x_{\kappa_n(r)}^{i,N,n}+\Gamtamer,\tilmukappatameizbarr,z\big) \notag
		\\
		& \qquad-\gamma^{u}\big( x_{r}^{i,N,n},\tilpikappaGammatamei,z\big)+\gammatameu\big( x_{\kappa_n(r)}^{i,N,n},\tilmukappatameizbarr,z\big)\Big|^2 \nu(d\bar z) dr \nu(dz)\notag
		\\
		=& KE\int_Z\int_{\kappa_{n}(s)}^{s}\int_Z\Big|\gamma^{u}\big(x_r^{i,N,n}+\Gammatame,\tilpikappaGammatamei,z\big) \notag
		\\
		&
		\qquad-\gamma^{u}\big(x_{\kappa_n(r)}^{i,N,n}+\Gamtamer,\tilmukappatameizbarr,z\big)  	\\
		&
		\qquad-\gamma^{u}\big( x_r^{i,N,n},\tilpikappaGammatamei,z\big)+\gamma^{u}\big( x_{\kappa_n(r)}^{i,N,n},\tilmukappatameizbarr,z\big)
		\\
		&
		\qquad +\gamma^{u}\big(x_{\kappa_n(r)}^{i,N,n}+\Gamtamer,\tilmukappatameizbarr,z\big) 
		\\
		&\qquad-\gammatameu\big(x_{\kappa_n(r)}^{i,N,n}+\Gamtamer,\tilmukappatameizbarr,z\big) \notag
		\\
		& \qquad -\gamma^{u}\big( x_{\kappa_n(r)}^{i,N,n},\tilmukappatameizbarr,z\big)+\gammatameu\big( x_{\kappa_n(r)}^{i,N,n},\tilmukappatameizbarr,z\big)\Big|^2 \nu(d\bar z) dr\nu(dz)\notag
		\\
		\leq &  Kn^{-1-\frac{2}{2+\varepsilon}}\hspace{-0.02cm}+\hspace{-0.05cm}KE\int_Z\hspace{-0.06cm}\int_{\kappa_{n}(s)}^{s}\hspace{-0.06cm}\int_Z\Big\{\Big|\gamma^{u}\big(x_r^{i,N,n}\hspace{-0.04cm}+\hspace{-0.04cm}\Gammatame,\tilpikappaGammatamei,z\big) \notag
		\\
		&
		\qquad -\gamma^{u}\big(x_{\kappa_n(r)}^{i,N,n}+\Gamtamer,\tilmukappatameizbarr,z\big)\Big|^2  	\\
		& 
		\qquad +\Big|\gamma^{u}\big( x_r^{i,N,n},\tilpikappaGammatamei,z\big)-\gamma^{u}\big( x_{\kappa_n(r)}^{i,N,n},\tilmukappatameizbarr,z\big)\Big|^2\Big\} \nu(d\bar z) dr \nu(dz) \notag
	\end{align*}
	 which on recalling Equations \eqref{eq:emperical},  \eqref{eq:empirical} and  applying Remark \ref{rem:poly:sig:gam},  H\"older's inequality, Equation \eqref{eq:empirical_dif}, Lemma \ref{lem:scm:mb},  Corollaries \ref{cor:sig:gam}, \ref{cor1:sig:gam}, \ref{cor:one:step:error} yield
	\begin{align} \label{eq:lam81}
		&\color{black}\lambda_{81}
		\leq   Kn^{-1-\frac{2}{2+\varepsilon}}+K\int_{\kappa_{n}(s)}^{s}E\int_Z\Big[\Big\{1+|x_r^{i,N,n}|+|x_{\kappa_{n}(r)}^{i,N,n}|+\big|\Gammatame\big| \notag
		\\
		&\color{black}+\big|\Gamtamer\big|\Big\}^{\eta}\Big\{|x_r^{i,N,n}-x_{\kappa_{n}(r)}^{i,N,n}|+\big|\Gammatame \notag
		\\
		& \color{black}-\Gamtamer\big|\Big\}^2+\mathcal{W}_2^2\big(\tilpikappaGammatamei,\tilmukappatameizbarr\big)\Big] \nu(d\bar z)dr  \notag
				\\
				\leq & Kn^{-1-\frac{2}{2+\varepsilon}}+K\int_{\kappa_n(s)}^s\bigg\{\Big[E\int_ Z\Big\{1+|x_r^{i,N,n}|+|x_{\kappa_{n}(r)}^{i,N,n}|+\big|\Gamtamer\big| \notag
				\\
				&\qquad\quad+\sum_{\mathfrak q=1}^{2}\big|\Gamqr\big|\Big\}^{\frac{\eta(\varepsilon+2)}{\varepsilon}}\nu(d\bar z)\Big]^{\frac{\varepsilon}{\varepsilon+2}}\notag
					\\
				&\times  \Big[E\int_ Z\Big\{|x_{r}^{i,N,n}-x_{\kappa_n(r)}^{i,N,n}|^{\varepsilon+2}+\sum_{\mathfrak{q}=1}^2 \big|\Gamqr\big|^{\varepsilon+2}\Big\} \nu(d\bar z)\Big]^{\frac{2}{\varepsilon+2}} \bigg\}dr\notag
				\\
				&+  K\int_{\kappa_{n}(s)}^{s}E\int_ Z\Big\{\frac{1}{N}\sum_{j=1}^{N} |x_{r}^{j,N,n}-x_{\kappa_n(r)}^{j,N,n}|^{2}+\frac{1}{N}\sum_{\mathfrak{q}=1}^2 \big|\Gamqr\big|^{2}\Big\} \nu(d\bar z)dr   \notag
				\\
				\leq & Kn^{-1-\frac{2}{\varepsilon+2}} \notag
                \\
                &+Kn^{-\frac{2}{\varepsilon+2}}\int_{\kappa_{n}(s)}^s\Big[K+\sum_{\mathfrak q=1}^{2}\int_Z \Big\{E\big|\Gamqr\big|^{\frac{(\eta+2)(\varepsilon+2)}{\varepsilon}}\Big\}^{\frac{\eta}{\eta+2}}\nu(d\bar z)\Big]^{\frac{\varepsilon}{\varepsilon+2}}dr \notag
				\\
				&\qquad\qquad\qquad+Kn^{-2}
				\leq Kn^{-1-\frac{2}{\varepsilon+2}}
	\end{align}
	for any  $s\in[0,T]$, $u\in\{1,\ldots,d\}$ and $i\in\{1,\ldots,N\}$. \color{black}
	To estimate $\lambda_{82}$, one applies Assumption \ref{asum:convergence2} as follows
	\begin{align} \label{eq:lam82*}
		&\lambda_{82}:=KN\sum_{k=1}^N\int_ZE\int_{\kappa_{n}(s)}^{s}\int_Z\Big|\gamma^{u}\big(x_r^{i,N,n},\tilpikappaGammatame,z\big)-\gammatameu\big(x_{\kappa_n(r}^{i,N,n},\tilmukappatamezbarr,z\big) \notag
		\\
		&
		\qquad \qquad -\gamma^{u}\big( x_r^{i,N,n},\mu_r^{x,N,n},z\big)+\gammatameu\big( x_{\kappa_n(r)}^{i,N,n},\mu_{\kappa_{n}(r)}^{x,N,n},z\big) \Big|^2 \nu(d\bar z)  dr \nu(dz) \notag
		\\
		=&KN\sum_{k=1}^N E\int_Z\int_{\kappa_{n}(s)}^{s}\int_Z\Big|\gamma^{u}\big(x_r^{i,N,n},\tilpikappaGammatame,z\big) -\gamma^{u}\big( x_r^{i,N,n},\mu_r^{x,N,n},z\big)\notag
		\\
		&
		\qquad\qquad-\gamma^{u}\big(x_{\kappa_n(r)}^{i,N,n},\tilmukappatamezbarr,z\big) +\gamma^{u}\big( x_{\kappa_n(r)}^{i,N,n},\mu_{\kappa_{n}(r)}^{x,N,n},z\big) \notag
		\\
		&\qquad\qquad +\gamma^{u}\big(x_{\kappa_n(r)}^{i,N,n},\tilmukappatamezbarr,z\big) -\gamma^{u}\big( x_{\kappa_n(r)}^{i,N,n},\mu_{\kappa_{n}(r)}^{x,N,n},z\big) \notag
		\\
		&
		\qquad\qquad\ -\gammatameu\big(x_{\kappa_n(r)}^{i,N,n},\tilmukappatamezbarr,z\big) +\gammatameu\big( x_{\kappa_n(r)}^{i,N,n},\mu_{\kappa_{n}(r)}^{x,N,n},z\big) \Big|^2 \nu(d\bar z) dr\nu(dz) \notag
		\\
		\leq & KN\sum_{k=1}^NE\int_Z\int_{\kappa_{n}(s)}^{s}\int_Z\Big|\gamma^{u}\big(x_r^{i,N,n},\tilpikappaGammatame,z\big) -\gamma^{u}\big( x_r^{i,N,n},\mu_r^{x,N,n},z\big)\notag
		\\
		&
		 -\gamma^{u}\big(x_{\kappa_n(r)}^{i,N,n},\tilmukappatamezbarr,z\big)+\gamma^{u}\big( x_{\kappa_n(r)}^{i,N,n},\mu_{\kappa_{n}(r)}^{x,N,n},z\big)
		\Big|^2 \nu(d\bar z)  dr \nu(dz)	+\sum_{k=1}^{N}KN^{-1}n^{-1-\frac{2}{2+\varepsilon}}
  \end{align}
	for any  $s\in[0,T]$, $u\in\{1,\ldots,d\}$ and $i\in\{1,\ldots,N\}$. 
  To estimate the  Equation \eqref{eq:lam82*}, one needs to estimate  the following term  by recalling Equations \eqref{eq:emperical},  \eqref{eq:empirical} and  using Lemma \ref{lem:MVT} 
	\begin{align*}
		& \color{black}\Big|\gamma^{u}\big(x_r^{i,N,n},\tilpikappaGammatame,z\big) -\gamma^{u}\big( x_r^{i,N,n},\mu_r^{x,N,n},z\big) \notag
  \\
  &\qquad\qquad\qquad\qquad\color{black}-\gamma^{u}\big(x_{\kappa_n(r)}^{i,N,n},\tilmukappatamezbarr,z\big)+\gamma^{u}\big( x_{\kappa_n(r)}^{i,N,n},\mu_{\kappa_{n}(r)}^{x,N,n},z\big)\Big|
  \notag
		\\
		&= \Big|\frac{1}{N}\int_0^1\sum_{\widehat j=1}^N\partial_\mu\gamma^{u}\Big( x_r^{i,N,n},\frac{1}{N}\sum_{j=1}^{N}\delta_{x_r^{j,N,n}+\theta\big\{x_r^{j,N,n}+1_{\{j=k\}}\Gammatamejanother-x_r^{j,N,n}\big\}}, \notag
		\\
		&\qquad\qquad	x_r^{\widehat j,N,n}+\theta\big\{x_r^{\widehat j,N,n}+1_{\{\widehat j=k\}}\Gammatamejhat-x_r^{\widehat j,N,n}\big\},z\Big) \notag
		\\
		&\qquad\qquad\qquad \times \big\{x_r^{\widehat j,N,n}+1_{\{\widehat j=k\}}\Gammatamejhat-x_r^{\widehat j,N,n}\big\} d\theta \notag
		\\
		&\quad\quad-\frac{1}{N}\int_0^1\sum_{\widehat j=1}^N\partial_\mu\gamma^{u}\Big( x_{\kappa_{n}(r)}^{i,N,n},\frac{1}{N}\sum_{j=1}^{N}\delta_{x_{\kappa_{n}(r)}^{j,N,n}+\theta\big\{x_{\kappa_{n}(r)}^{j,N,n}+1_{\{j=k\}}(\widehat{\gamma})_{tam}^n\big(x_{\kappa_{n}(r)}^{j,N,n},\mu_{\kappa_{n}(r)}^{x,N,n},\bar z\big)-x_{\kappa_{n}(r)}^{j,N,n}\big\}}, \notag
		\\
		&\qquad\qquad\quad	x_{\kappa_{n}(r)}^{\widehat j,N,n}+\theta\big\{x_{\kappa_{n}(r)}^{\widehat j,N,n}+1_{\{\widehat j=k\}}(\widehat{\gamma})_{tam}^n\big(x_{\kappa_{n}(r)}^{\widehat j,N,n},\mu_{\kappa_{n}(r)}^{x,N,n},\bar z\big)-x_{\kappa_{n}(r)}^{\widehat j,N,n}\big\},z\Big) \notag
		\\
		&\qquad\qquad\qquad\quad \times \big\{x_{\kappa_{n}(r)}^{\widehat j,N,n}+1_{\{\widehat j=k\}}(\widehat{\gamma})_{tam}^n\big(x_{\kappa_{n}(r)}^{\widehat j,N,n},\mu_{\kappa_{n}(r)}^{x,N,n},\bar z\big)-x_{{\kappa_n(r)}}^{\widehat j,N,n}\big\} d\theta{\Big|}
		\notag
		\\
		&\color{black}={\Big|}\frac{1}{N}\int_0^1\partial_\mu\gamma^{u}\Big( x_r^{i,N,n},\frac{1}{N}\sum_{j=1}^{N}\delta_{x_r^{j,N,n}+\theta1_{\{j=k\}}\Gammatamejanother}, \notag
		\\
		&\color{black}\qquad\qquad	x_r^{k,N,n}+\theta\Gammatamek,z\Big) 
		\Gammatamek d\theta \notag
		\\
		&\color{black}\quad-\frac{1}{N}\int_0^1\partial_\mu\gamma^{u}\Big( x_{\kappa_{n}(r)}^{i,N,n},\frac{1}{N}\sum_{j=1}^{N}\delta_{x_{\kappa_{n}(r)}^{j,N,n}+\theta 1_{\{j=k\}}(\widehat{\gamma})_{tam}^n\big(x_{\kappa_{n}(r)}^{j,N,n},\mu_{\kappa_{n}(r)}^{x,N,n},\bar z\big)}, 	\notag
		\\
		&\color{black}\qquad\qquad x_{\kappa_{n}(r)}^{k,N,n}+\theta(\widehat{\gamma})_{tam}^n\big(x_{\kappa_{n}(r)}^{k,N,n},\mu_{\kappa_{n}(r)}^{x,N,n},\bar z\big),z\Big) 
		(\widehat{\gamma})_{tam}^n\big(x_{\kappa_{n}(r)}^{k,N,n},\mu_{\kappa_{n}(r)}^{x,N,n},\bar z\big) d\theta{\Big|} \notag
		\\
		&\color{black}\leq\frac{1}{N}\big|(\widehat{\gamma})_{tam}^n\big(x_{\kappa_{n}(r)}^{k,N,n},\mu_{\kappa_{n}(r)}^{x,N,n},\bar z\big)\big|\int_0^1\Big|\partial_\mu\gamma^{u}\Big( x_r^{i,N,n},\frac{1}{N}\sum_{j=1}^{N}\delta_{x_r^{j,N,n}+\theta1_{\{j=k\}}\Gammatamejanother}, \notag
		\\
		&\color{black}\qquad\qquad	x_r^{k,N,n}+\theta\Gammatamek,z\Big) 
		\notag
		\\
		&\color{black}\quad-\partial_\mu\gamma^{u}\Big( x_{\kappa_{n}(r)}^{i,N,n},\frac{1}{N}\sum_{j=1}^{N}\delta_{x_{\kappa_{n}(r)}^{j,N,n}+\theta 1_{\{j=k\}}(\widehat{\gamma})_{tam}^n\big(x_{\kappa_{n}(r)}^{j,N,n},\mu_{\kappa_{n}(r)}^{x,N,n},\bar z\big)}, 	x_{\kappa_{n}(r)}^{k,N,n}+\theta(\widehat{\gamma})_{tam}^n\big(x_{\kappa_{n}(r)}^{k,N,n},\mu_{\kappa_{n}(r)}^{x,N,n},\bar z\big),z\Big) \Big| d\theta \notag
		\\
		&\color{black}\quad+\frac{1}{N}\int_0^1\Big|\partial_\mu\gamma^{u}\Big( x_r^{i,N,n},\frac{1}{N}\sum_{j=1}^{N}\delta_{x_r^{j,N,n}+\theta1_{\{j=k\}}\Gammatamejanother}, \notag
		\\
		&\color{black}\qquad\qquad	x_r^{k,N,n}+\theta\Gammatamek,z\Big)\Big| 
		\sum_{\mathfrak{q}=1}^2\big|\Gamqkr\big| d\theta
	\end{align*}
\color{black}	which on applying Assumptions \ref{asum:dlip} and \ref{asum:measure:derv:bound} yield
	\begin{align*}
		& \int_Z\int_Z\Big|\gamma^{u}\big(x_r^{i,N,n},\tilpikappaGammatame,z\big) -\gamma^{u}\big( x_r^{i,N,n},\mu_r^{x,N,n},z\big)\notag
		\\
		&
		 -\gamma^{u}\big(x_{\kappa_n(r)}^{i,N,n},\tilmukappatamezbarr,z\big)+\gamma^{u}\big( x_{\kappa_n(r)}^{i,N,n},\mu_{\kappa_{n}(r)}^{x,N,n},z\big)
		\Big|^2 \nu(d\bar z)  \nu(dz)
		\notag
		\\
		&\leq  \frac{1}{N^2} \int_Z\int_Z\int_0^1{|\bar C_z|^2}\big|(\widehat{\gamma})_{tam}^n\big(x_{\kappa_{n}(r)}^{k,N,n},\mu_{\kappa_{n}(r)}^{x,N,n},\bar z\big)\big|^2 \Big\{\big(1+|x_r^{i,N,n}|+|x_{\kappa_{n}(r)}^{i,N,n}|\big)^{\frac{\eta}{2}} 
		\big|x_r^{i,N,n}-x_{\kappa_{n}(r)}^{i,N,n}\big|
		\\
		&\quad+\big|x_r^{k,N,n}-x_{\kappa_{n}(r)}^{k,N,n}+\theta\Gammatamek-\theta(\widehat{\gamma})_{tam}^n\big(x_{\kappa_{n}(r)}^{k,N,n},\mu_{\kappa_{n}(r)}^{x,N,n},\bar z\big)\big|
		\\
		&\quad+\mathcal{W}_2\Big(\frac{1}{N}\sum_{j=1}^{N}\delta_{x_r^{j,N,n}+\theta1_{\{j=k\}}\Gammatamejanother},\notag
		\\
		&\qquad\qquad\qquad\qquad\qquad\qquad\frac{1}{N}\sum_{j=1}^{N}\delta_{x_{\kappa_{n}(r)}^{j,N,n}+\theta 1_{\{j=k\}}(\widehat{\gamma})_{tam}^n\big(x_{\kappa_{n}(r)}^{j,N,n},\mu_{\kappa_{n}(r)}^{x,N,n},\bar z\big)}\Big)\Big\}^2 d\theta \nu(d\bar z)  \nu(dz)
		\\
		&\quad +\frac{1}{N^2}\sum_{\mathfrak{q}=1}^2\int_Z\int_Z{|\bar C_z|^2}\big|\Gamqkr\big|^2\nu(d\bar z)\nu(dz) 
	\end{align*}
	for any  $s\in[0,T]$, $u\in\{1,\ldots,d\}$,  $z\in Z$ and $i,k\in\{1,\ldots,N\}$. 
 Furthermore,  by substituting the above equation into Equation \eqref{eq:lam82*} and using H\"older's inequality, Equation \eqref{eq:empirical_dif} along with Lemma \ref{lem:scm:mb}, Corollaries \ref{cor:sig:gam}, \ref{cor1:sig:gam}, \ref{cor:one:step:error}, one obtains
	\begin{align} \label{eq:lam82}
		&\color{black}\lambda_{82}
		\leq \frac{K}{N}\sum_{k=1}^N \int_{\kappa_{n}(s)}^sE\int_Z \int_0^1 \big|(\widehat{\gamma})_{tam}^n\big(x_{\kappa_{n}(r)}^{k,N,n},\mu_{\kappa_{n}(r)}^{x,N,n},\bar z\big)\big|^2 \Big\{\big(1+|x_r^{i,N,n}|+|x_{\kappa_{n}(r)}^{i,N,n}|\big)^{\eta} 
		\big|x_r^{i,N,n}-x_{\kappa_{n}(r)}^{i,N,n}\big|^2 \notag
		\\
		&\color{black}+\big|x_r^{k,N,n}-x_{\kappa_{n}(r)}^{k,N,n}+\theta\Gammatamek-\theta(\widehat{\gamma})_{tam}^n\big(x_{\kappa_{n}(r)}^{k,N,n},\mu_{\kappa_{n}(r)}^{x,N,n},\bar z\big)\big|^2 \notag
		\\
		&\color{black}+\mathcal{W}_2^2\Big(\frac{1}{N}\sum_{j=1}^{N}\delta_{x_r^{j,N,n}+\theta1_{\{j=k\}}\Gammatamejanother}, \notag
  \\
    &\color{black}\qquad\qquad\qquad\qquad\qquad\qquad\frac{1}{N}\sum_{j=1}^{N}\delta_{x_{\kappa_{n}(r)}^{j,N,n}+\theta 1_{\{j=k\}}(\widehat{\gamma})_{tam}^n\big(x_{\kappa_{n}(r)}^{j,N,n},\mu_{\kappa_{n}(r)}^{x,N,n},\bar z\big)}\Big)\Big\}d\theta\nu(d\bar z) dr \notag
		\\
		& \color{black}+\frac{K}{N}\sum_{k=1}^N E\int_{\kappa_{n}(s)}^s\int_Z \sum_{\mathfrak{q}=1}^2\big|\Gamqkr\big|^2\nu(d\bar z) dr +Kn^{-1-\frac{2}{2+\varepsilon}} 
  \notag
	\\
	&\quad\leq  \frac{K}{N}\sum_{k=1}^N\int_{\kappa_n(s)}^s\Big\{E\int_ Z\big|(\widehat{\gamma})_{tam}^n\big(x_{\kappa_{n}(r)}^{k,N,n},\mu_{\kappa_{n}(r)}^{x,N,n},\bar z\big)\big|^{\frac{4(\varepsilon+2)}{\varepsilon}}\nu(d\bar z)\Big\}^{\frac{\varepsilon}{2(\varepsilon+2)}} \notag
	\\
	&\qquad\times  \big\{E(1+|x_{r}^{i,N,n}|+|x_{\kappa_n(r)}^{i,N,n}|)^{\frac{2\eta(\varepsilon+2)}{\varepsilon}}\big\}^{\frac{\varepsilon}{2(\varepsilon+2)}} \big\{E|x_{r}^{i,N,n}-x_{\kappa_n(r)}^{i,N,n}|^{\varepsilon+2}\big\}^{\frac{2}{\varepsilon+2}}  dr \notag
	\\
&+ \frac{K}{N}\sum_{k=1}^{N}\int_{\kappa_n(s)}^s\Big\{E\int_ Z\big|(\widehat{\gamma})_{tam}^n\big(x_{\kappa_{n}(r)}^{k,N,n},\mu_{\kappa_{n}(r)}^{x,N,n},\bar z\big)\big|^{\frac{2(\varepsilon+2)}{\varepsilon}}\nu(d\bar z)\Big\}^{\frac{\varepsilon}{\varepsilon+2}}\notag
	\\
	&\qquad\times  \Big[E\int_ Z \Big\{|x_{r}^{k,N,n}-x_{\kappa_n(r)}^{k,N,n}|^{\varepsilon+2}+\sum_{\mathfrak{q}=1}^2\big|\Gamqkr\big|^{\varepsilon+2} \notag
	\\
	&\qquad\qquad+\frac{1}{N}\sum_{j=1}^{N}|x_{r}^{j,N,n}-x_{\kappa_n(r)}^{j,N,n}|^{\varepsilon+2}\Big\}\nu(d\bar z)\Big]^{\frac{2}{\varepsilon+2}} dr +Kn^{-2}+Kn^{-1-\frac{2}{2+\varepsilon}}   \notag
	\\
	&\quad\leq Kn^{-1-\frac{2}{\varepsilon+2}}+K\int_{\kappa_n(s)}^s\big\{\sup_{i \in \{1,\ldots,N\}}\sup_{ r\in[0,s]}E|x_{r}^{i,N,n}-x_{\kappa_n(r)}^{i,N,n}|^{\varepsilon+2}\big\}^{\frac{2}{\varepsilon+2}} dr  
	\leq K n^{-1-\frac{2}{\varepsilon+2}}
	\end{align}
 for any  $s\in[0,T]$, $u\in\{1,\ldots,d\}$ and $i\in\{1,\ldots, N\}$.	\color{black}
	Then, by   substituting Equations \eqref{eq:lam81} and \eqref{eq:lam82} into Equation \eqref{eq:lam8*}, one  gets
	\begin{align} \label{eq:lam8}
		\lambda_8\leq K n^{-1-\frac{2}{\varepsilon+2}}
	\end{align}
	for any  $s\in[0,T]$, $u\in\{1,\ldots,d\}$ and $i\in\{1,\ldots,N\}$.
	Finally, one can substitute Equations \eqref{eq:lam1} to \eqref{eq:lam7} and \eqref{eq:lam8}  into Equation \eqref{eq:lam} to complete the proof.
\end{proof} 
 One can prove the following lemma by using arguments similar to the Lemma \ref{lem:gam} shown above.
\color{black}
	\begin{lem} \label{lem:sig}
		Let Assumptions \mbox{\normalfont  \ref{asum:ic}}, \mbox{\normalfont  \ref{asum:lip}} and \mbox{\normalfont  \ref{asump:super}} to \mbox{\normalfont  \ref{asum:convergence2}}  hold 
			with  
   $\bar p\geq \max\{2,\eta/8+1\}(\varepsilon+2)(\eta+1)(\eta+2)/{\varepsilon}\color{black}$ where  $\varepsilon\in(0,1)$.
		 Then,
		\begin{align*}
			&E\big|\sigma\big( x_s^{i, N,n},\mu_s^{x, N,n}\big)-\Lambdatams\big|^2\leq Kn^{-1-\frac{2}{\varepsilon+2}}
		\end{align*}
		for any $s\in[0,T]$ and $i\in\{1,\ldots,N\}$	where $K>0$ is a constant independent of $N$ and $n$.
	\end{lem} 

Recalling  the interacting particle system \eqref{eq:int} and  the Milstein-type scheme \eqref{eq:scm**}, let us  use   the following $d\times1$ vector notation whose $u$-th element is
\begin{align}\label{eq:error}
	e_t^{iu,N,n}:=x_t^{iu,N}-&x_t^{iu,N,n}=\int_{0}^{t}\Big\{b^u( x_s^{iu, N},\mu_s^{x, N})-\bbtame\Big\}ds \notag
	\\
	&+\int_{0}^{t}\Big\{\sigma^{(u)}(x_s^{i,N},\mu_s^{x,N})-\Lambdatamsu\Big\}dw^{i}_s \notag
	\\
	&+\int_{0}^{t}\int_Z \Big\{\gamma^u(x_s^{i,N},\mu_s^{x,N},\bar z)-\Gammatameus\Big\}\tilde{n}_p^i(ds,d\bar z)
\end{align}
for any $t\in[0,T]$, $i\in\{1,\ldots,N\}$  and $u\in\{1,\ldots,d\}$.
	\begin{lem} \label{lem:b}
			Let Assumptions \mbox{\normalfont  \ref{asum:ic}} to \mbox{\normalfont  \ref{asum:lin1}} and \mbox{\normalfont  \ref{asum:lin**}} 
   to \mbox{\normalfont  \ref{asum:convergence2}}  hold
		with  
  $\bar p\geq \max\{2,\eta/8+1\}(\varepsilon+2)(\eta+1)(\eta+2)/{\varepsilon}\color{black}$ 
  where $\varepsilon\in(0,1)$.
			Then,  	
		\begin{align*}
			E\int_{0}^{t}&e_s^{i,N,n}\Big\{ b\big( x_s^{i, N,n},\mu_s^{x, N,n}\big)-\Btame\Big\}ds
   \\
   &\leq K\int_{0}^{t} \sup_{i \in \{1,\ldots,N\}}\sup_{r\in[0,s]}E|e_r^{i,N,n}|^2ds+Kn^{-1-\frac{2}{\varepsilon+2}}
		\end{align*}
	for any $t\in[0,T]$ and $i\in\{1,\ldots,N\}$ where  $K>0$ is a  constant independent of $N$ and $n$.
	\end{lem}  
	\begin{proof}
		First,  recall   the Milstein-type scheme \eqref{eq:scm**} and use It\^o's formula in Equation  \eqref{eq:ito} (Lemma \ref{lem:ito}) to obtain
		\begin{align*} 
			&b^{u}\big( x_s^{i, N,n},\mu_s^{x, N,n}\big) =b^{u}\big(x_{\kappa_n(s)}^{i,N,n},\mu_{\kappa_n(s)}^{x,N,n}\big)
			+\int_{\kappa_n(s)}^s\partial_x b^{u}\big(x_{r}^{i,N,n},\mu_{r}^{x,N,n}\big)\Btamer dr  \notag
			\\
			&-\int_{\kappa_n(s)}^s\int_Z \partial_xb^{u} \big(x_{r}^{i,N,n},\mu_{r}^{x,N,n}\big)\Gammatame\nu(d\bar z)dr \notag
			\\
			&+\frac{1}{N}\sum_{k= 1}^{N}\int_{\kappa_n(s)}^s\partial_\mu b^{u}\big( x_{r}^{i,N,n},\mu_{r}^{x,N,n}, x_{r}^{k,N,n}\big)\Btamekr dr  \notag
			\\
			&-\frac{1}{N}\sum_{k=1}^{N}\int_{\kappa_n(s)}^s\int_Z \partial_\mu b^{u} \big(x_{r}^{i,N,n},\mu_{r}^{x,N,n},x_{r}^{k,N,n}\big)\Gammatamek\nu(d\bar z)dr \notag
			\\
			&+\sum_{\ell_1=1}^{m}\int_{\kappa_n(s)}^s \partial_x b^{u}\big( x_{r}^{i,N,n},\mu_{r}^{x,N,n}\big)\Lambdatame dw^{i\ell_1}_r \notag
			\\
			&+\frac{1}{N}\sum_{\ell_1=1}^m\sum_{k= 1}^{N}\int_{\kappa_n(s)}^s\partial_\mu b^{u}\big( x_{r}^{i,N,n},\mu_{r}^{x,N,n},x_{r}^{k,N,n}\big)\Lambdatamek dw^{k\ell_1}_r \notag
			\\
			& + \frac{1}{2}\int_{\kappa_n(s)}^{s} \tr\big[\partial_{x}^2 b^{u}\big(x_r^{i,N,n},\mu_r^{x,N,n}\big) 
			\Lambdatam\Lambdatamtrans\big]  dr \notag
			\\
			&+\frac{1}{N}\int_{\kappa_n(s)}^{s} \tr\big[\partial_{x}\partial_{\mu} b^{u}\big(x_r^{i,N,n},\mu_r^{x,N,n},x_r^{i,N,n}\big) \notag
			\\
			&\qquad\qquad\qquad\qquad
			\times
			\Lambdatam\Lambdatamtrans\big] dr  \notag
			\\
			&+\frac{1}{2N}\sum_{k=1}^N   \int_{\kappa_n(s)}^{s} \tr\big[\partial_{y}\partial_{\mu} b^{u}\big(x_r^{i,N,n},\mu_r^{x,N,n},x_r^{k,N,n}\big) \notag
			\\
			&\qquad\qquad\qquad\qquad
			\times
			\Lambdatamk\Lambdatamktrans\big]dr  \notag
			\\
			&+\frac{1}{2N^2}\sum_{k=1}^N  \int_{\kappa_n(s)}^{s} \tr\big[\partial_{\mu}^2 b^{u}\big(x_r^{i,N,n},\mu_r^{x,N,n},x_r^{k,N,n},x_r^{k,N,n}\big) \notag
			\\
			&\qquad\qquad\qquad\qquad
			\times\Lambdatamk\Lambdatamktrans\big]dr  \notag
			\\
			&+\sum_{k=1}^{N}\int_{\kappa_{n}(s)}^{s}\int_Z\Big\{b^{u}\big(x_{r-}^{i,N,n}+1_{\{k=i\}}\Gammatamek,\tilpikappaGammatameve\big) \notag
			\\
			&\qquad\qquad\qquad\qquad
			-b^{u}\big( x_{r-}^{i,N,n},\mu_{r-}^
			{x,N,n}\big)\Big\} n_p^k(dr,d\bar z) 
		\end{align*}
		almost surely for any   $s\in[0,T]$, $u\in\{1,\ldots,d\}$ and $i\in\{1,\ldots,N\}$	where  $\tilpikappaGammatame$ is defined in Equation \eqref{eq:empirical}.
	
	Now,  for simplicity of presentation, let us introduce the following notations of the sixth, seventh and the last terms of the above equation as below
	\begin{align}
&\Phi^{iu,N,n}_1\big(\kappa_n(s),s\big):=\sum_{\ell_1=1}^{m}\int_{\kappa_n(s)}^s \partial_x b^{u}\big( x_{r}^{i,N,n},\mu_{r}^{x,N,n}\big)\Lambdatame dw^{i\ell_1}_r \notag
		\\
		&\qquad\qquad+\frac{1}{N}\sum_{\ell_1=1}^{m}\sum_{k= 1}^{N} \int_{\kappa_n(s)}^s\partial_\mu b^{u}\big( x_{r}^{i,N,n},\mu_{r}^{x,N,n},x_{r}^{k,N,n}\big)\Lambdatamek dw^{k\ell_1}_r, \label{eq:phi1}
		\\
&\Phi^{iu,N,n}_2\big(\kappa_n(s),s\big)  :=\sum_{k=1}^{N} \int_{\kappa_{n}(s)}^{s}\int_Z\Big\{b^{u}\big(x_{r-}^{i,N,n}+1_{\{k=i\}}\Gammatamek,\tilpikappaGammatameve\big) \notag
		\\
&\qquad\qquad\qquad\qquad\qquad\qquad\qquad
		-b^{u}\big( x_{r-}^{i,N,n},\mu_{r-}^
		{x,N,n}\big)\Big\} \tilde n_p^k(dr,d\bar z) \label{eq:phi2}
	\end{align}
	for any $s\in[0,T]$, $i\in\{1,\ldots,N\}$ and $u\in\{1,\ldots,d\}$ and notice that 
	$$
	E_{\kappa_{n}(s)}\Phi^{iu,N,n}_1\big(\kappa_n(s),s\big)=E_{\kappa_{n}(s)}\Phi^{iu,N,n}_2\big(\kappa_n(s),s\big)=0.
	$$
	Furthermore, for the second moment of   $\Phi_1^{iu,N,n}$,
	one  recalls Equation \eqref{eq:phi1} and  applies  Remark \ref{rem:super_linear}, Assumption \ref{asum:measure:derv:bound} along with Lemma \ref{lem:scm:mb}, Corollaries \ref{cor:sig:gam}, \ref{cor1:sig:gam} as follows
	\begin{align} \label{eq:phi1*}
		&E\big|\Phi_1^{iu,N,n}\big(\kappa_n(s),s\big)\big|^2
		\leq  K\sum_{\ell_1=1}^{m}E\int_{\kappa_n(s)}^{s}\Big\{ |\partial_x b^{u}\big( x_{r}^{i,N,n},\mu_{r}^{x,N,n}\big)|^2 \big|\Lambdatame\big|^2   \notag
		\\
		&\qquad\qquad+\frac{1}{N}\sum_{k=1}^N    |\partial_\mu b^{u}\big( x_{r}^{i,N,n},\mu_{r}^{x,N,n},x_{r}^{k,N,n}\big)|^2\big|\Lambdatamek\big|^2 
		\Big\}dr \notag
		\\
		&\leq  K\sum_{\ell_1=1}^{m}E\int_{\kappa_n(s)}^{s}\Big\{ \big(1+|x_r^{i,N,n}|\big)^{2\eta} \big|\Lambdatame\big|^2   \notag
		\\
		&\qquad\qquad\qquad\qquad\qquad+\frac{1}{N}\sum_{k=1}^N    \big|\Lambdatamek\big|^2 
		\Big\}dr \notag
		\\
		&\leq  K\hspace{-0.05cm}\sum_{\ell_1=1}^{m}\hspace{-0.05cm}\int_{\kappa_n(s)}^s \hspace{-0.15cm}\big\{E(1+|x_{r}^{i,N,n}|)^{4\eta}\big\}^{\frac{1}{2}} \big\{E\big|\Lambdatame\big|^4 \big\}^{\frac{1}{2}}  dr +Kn^{-1}
		\leq Kn^{-1}
	\end{align}
	for any  $s\in[0,T]$, $u\in\{1,\ldots,d\}$ and $i\in\{1,\ldots,N\}$. 
	Moreover, to estimate the second moment of  $\Phi_2^{iu,N,n}$, one needs to derive the following auxiliary results. For this,  recall Equation \eqref{eq:empirical} and  apply  Remark \ref{rem:super_linear},  Lemma \ref{lem:scm:mb}, Corollaries \ref{cor:sig:gam}, \ref{cor1:sig:gam} and Equation \eqref{eq:empirical_dif} as follows
	\begin{align} \label{eq:b:estem}
		&E\int_Z\Big|b^{u}\big(x_{r}^{i,N,n}+\Gammatame,\tilpikappaGammatamei\big) \notag
		\\
		&\qquad\qquad\qquad\qquad
		-b^{u}\big( x_{r}^{i,N,n},\tilpikappaGammatamei\big)\Big|^2 \nu(d\bar z)   \notag
		\\
		\leq &\int_ZE\Big(\Big\{1+|x_r^{i,N,n}|+\big|\Gammatame\big| \Big\}^{2(\eta+1)}  +\mathcal{W}_2^2\big(\tilpikappaGammatamei,\delta_0\big) \Big) \nu(d\bar z) \notag
		\\
		\leq & K+ K\int_ Z E\Big( \frac{1}{N}\sum_{j=1}^{N}|x_{r}^{j,N,n}|^{2}+\frac{1}{N} \big|\Gammatame\big|^{2}\Big) \nu(d\bar z) 
		\leq  K
	\end{align}
	for any  $t\in[0,T]$, $u\in\{1,\ldots,d\}$ and $i\in\{1,\ldots,N\}$. 
	Also,  by recalling Equation \eqref{eq:empirical} and using Lemma \ref{lem:MVT}, Assumption \ref{asum:measure:derv:bound} together with Corollaries \ref{cor:sig:gam}, \ref{cor1:sig:gam}, one obtains
	\begin{align}  \label{eq:b:estem*}
		&E \int_Z\big|b^{u}\big(x_{r}^{i,N,n},\tilpikappaGammatame\big) 
		-b^{u}\big( x_{r}^{i,N,n},\mu_{r}^
		{x,N,n}\big)\big|^2 \nu(d\bar z) \notag
		\\
		&=E\int_Z{\int_0^1}\Big|\frac{1}{N}\sum_{\widehat j=1}^N\partial_\mu b^{u}\Big( x_r^{i,N,n},\frac{1}{N}\sum_{j=1}^{N}\delta_{x_r^{j,N,n}+\theta\big\{x_r^{j,N,n}+1_{\{j=k\}}\Gammatamejanother-x_r^{j,N,n}\big\}}, \notag
		\\
		&\qquad\qquad	x_r^{\widehat j,N,n}+\theta\big\{x_r^{\widehat j,N,n}+1_{\{\widehat j=k\}}\Gammatamejhat-x_r^{\widehat j,N,n}\big\}\Big)  \notag
		\\
		&\qquad\qquad\qquad \times \big\{x_r^{\widehat j,N,n}+1_{\{\widehat j=k\}}\Gammatamejhat-x_r^{\widehat j,N,n}\big\}\Big|^2 d\theta\nu(d\bar z) \notag
		\\
		&\leq \frac{K}{N^2} 
		\int_ZE\big|\Gammatamek\big|^2 \nu(d\bar z)\leq  \frac{K}{N^2} 
	\end{align}
	for any  $t\in[0,T]$, $u\in\{1,\ldots,d\}$ and $i\in\{1,\ldots,N\}$. 
Thus, by recalling Equation \eqref{eq:phi2} and  using  Equations  \eqref{eq:b:estem} and  \eqref{eq:b:estem*}, one obtains
	\begin{align} \label{eq:phi2*}
		&E\big|\Phi_2^{iu,N,n}\big(\kappa_n(s),s\big)\big|^2 \notag
		\\
		=& 
		KE\Big| \int_{\kappa_{n}(s)}^{s}\int_Z\Big\{b^{u}\big(x_{r-}^{i,N,n}+\Gammatame,\tilpikappaGammatameive\big) \notag
		\\
		&\qquad\qquad\qquad\qquad
		-b^{u}\big( x_{r-}^{i,N,n},\tilpikappaGammatameive\big)\Big\} \tilde n_p^i(dr,d\bar z)
		\notag
		\\
		&+ 
		\sum_{k=1}^{N} \int_{\kappa_{n}(s)}^{s}\int_Z\Big\{b^{u}\big(x_{r-}^{i,N,n},\tilpikappaGammatameve\big)
		-b^{u}\big( x_{r-}^{i,N,n},\mu_{r-}^
		{x,N,n}\big) \Big\} \tilde n_p^k(dr,d\bar z)\Big|^2
		\notag
		\\
		\leq 
		&KE\int_{\kappa_{n}(s)}^{s}\int_Z\Big|b^{u}\big(x_{r}^{i,N,n}+\Gammatame,\tilpikappaGammatamei\big) \notag
		\\
		&\qquad\qquad\qquad\qquad
		-b^{u}\big( x_{r}^{i,N,n},\tilpikappaGammatamei\big)\Big|^2 \nu(d\bar z)dr   \notag
		\\
		&+KN\sum_{k=1}^{N}E \int_{\kappa_{n}(s)}^{s}\int_Z\big|b^{u}\big(x_{r}^{i,N,n},\tilpikappaGammatame\big) 
		-b^{u}\big( x_{r}^{i,N,n},\mu_{r}^
		{x,N,n}\big)\big|^2 \nu(d\bar z)dr \leq Kn^{-1}
	\end{align}
	for any $s\in[0,T]$, $u\in\{1,\ldots,d\}$ and $i\in\{1,\ldots,N\}$.

	Now,  apply Remark \ref{rem:super_linear}, Assumptions \ref{asum:measure:derv:bound}, \ref{asum:second:measure:derv} and the  Equations \eqref{eq:phi1}, \eqref{eq:phi2} to obtain
		\begin{align} \label{eq:xi}
			&E\int_{0}^{t} e_s^{iu,N,n}\Big\{b^{u}\big( x_s^{i, N,n},\mu_s^{x, N,n}\big) -\bbtame \Big\} ds \notag
			\\
			&\leq E\int_{0}^{t} |e_s^{iu,N,n}|\big|b^{u}\big(x_{\kappa_n(s)}^{i,N,n},\mu_{\kappa_n(s)}^{x,N,n}\big) -\bbtame \big| ds+ K E\int_{0}^{t} |e_s^{iu,N,n}| \notag
			\\
			&\qquad \times
			\int_{\kappa_n(s)}^s\Big\{\big(1+|x_r^{i,N,n}|\big)^{\eta}\big|\Btamer\big|   +\frac{1}{N}\sum_{k= 1}^{N}\big|\Btamekr\big|\Big\}  dr ds\notag
			\\
			&\quad+KE\int_{0}^{t} |e_s^{iu,N,n}|\int_{\kappa_n(s)}^s\int_Z \Big\{ \big(1+|x_r^{i,N,n}|\big)^{\eta}\big|\Gammatame\big| \notag
			\\ 
			&\qquad\qquad\qquad\qquad +\frac{1}{N}\sum_{k=1}^{N} \big|\Gammatamek\big|\Big\}\nu(d\bar z)drds \notag
			\\
			& + KE\int_{0}^{t} |e_s^{iu,N,n}|\int_{\kappa_n(s)}^{s}\Big\{ \big(1+|x_r^{i,N,n}|\big)^{\eta} \big|\Lambdatam\big|^2  \notag
			\\ 
			&\qquad\qquad\qquad\qquad +\frac{1}{N}\sum_{k=1}^N    \big|\Lambdatamk\big|^2 \notag
			\Big\}dr ds\notag
			\\
			&+E\int_{0}^{t} |e_s^{iu,N,n}|\sum_{k=1}^{N}\int_{\kappa_{n}(s)}^{s}\int_Z\Big|b^{u}\big(x_{r}^{i,N,n}+1_{\{k=i\}}\Gammatamek,\tilpikappaGammatameve\big) \notag
			\\
			&\qquad\qquad\qquad\qquad
			-b^{u}\big( x_{r}^{i,N,n},\mu_{r}^
			{x,N,n}\big)\Big| \nu(d\bar z)dr ds +E\int_{0}^{t} e_s^{iu,N,n} \sum_{\mathfrak{q}=1}^2 \Phi^{iu,N,n}_{\mathfrak{q}}\big(\kappa_n(s),s\big) ds \notag
			\\
			&=: \xi_1+\xi_2+\xi_3+\xi_4+\xi_5+\xi_6
		\end{align}
		for any $t\in[0,T]$, $u\in\{1,\ldots,d\}$ and $i\in\{1,\ldots,N\}$.
		
			For estimating $\xi_1$,   apply Young's inequality and Assumption \ref{asum:convergence}  to obtain
		\begin{align} \label{eq:xi1}
			\xi_1:=&E\int_{0}^{t} |e_s^{iu,N,n}|\big|b^{u}\big(x_{\kappa_n(s)}^{i,N,n},\mu_{\kappa_n(s)}^{x,N,n}\big) -\bbtame \big| ds \notag
			\\
			\leq & K \int_{0}^{t}E|e_s^{iu,N,n}|^2ds + K\int_{0}^{t} E\big|b^{u}\big(x_{\kappa_n(s)}^{i,N,n},\mu_{\kappa_n(s)}^{x,N,n}\big) -\bbtame \big|^2 ds \notag
			\\
			\leq &  K \int_{0}^{t}E|e_s^{iu,N,n}|^2ds + K n^{-1-\frac{2}{\varepsilon+2}}
		\end{align}
	for any $t\in[0,T]$, $u\in\{1,\ldots,d\}$ and $i\in\{1,\ldots,N\}$.
	To estimate $\xi_{2}$, one uses Young's inequality, H\"older's inequality along with Corollary \ref{cor:sig:gam} and Lemma \ref{lem:scm:mb} 
	  as follows
	\begin{align} \label{eq:xi2}
	\xi_2:=&K E\int_{0}^{t} \hspace{-0.15cm}|e_s^{iu,N,n}|
	\int_{\kappa_n(s)}^s \hspace{-0.15cm}\Big\{\big(1+|x_r^{i,N,n}|\big)^{\eta}\big|\Btamer\big|   \hspace{-0.03cm}+\hspace{-0.03cm}\frac{1}{N}\sum_{k= 1}^{N}\big|\Btamekr\big|\Big\}  dr ds \notag
	\\
	\leq & K \int_{0}^{t}E|e_s^{iu,N,n}|^2ds +K E \int_{0}^{t}\Big|\int_{\kappa_n(s)}^s\Big\{\big(1+|x_r^{i,N,n}|\big)^{\eta}\big|\Btamer\big| \notag
	\\
	&\qquad\qquad\qquad\qquad   +\frac{1}{N}\sum_{k= 1}^{N}\big|\Btamekr\big|\Big\}  dr\Big|^2ds \notag
	\\
	\leq &  K \int_{0}^{t}E|e_s^{iu,N,n}|^2ds +Kn^{-1} E \int_{0}^{t}\int_{\kappa_n(s)}^s\Big\{\big(1+|x_r^{i,N,n}|\big)^{2\eta}\big|\Btamer\big|^2 \notag
	\\
	& \qquad\qquad\qquad\qquad  +\frac{1}{N} \sum_{k=1}^{N}\big|\Btamekr\big|^2\Big\} dr ds  \notag
	\\
	\leq &  K \int_{0}^{t}\hspace{-0.1cm}E|e_s^{iu,N,n}|^2ds\hspace{-0.03cm}+\hspace{-0.03cm}Kn^{-1}      \hspace{-0.1cm}\int_{0}^{t}\hspace{-0.15cm}\int_{\kappa_n(s)}^s\hspace{-0.15cm}\Big\{E\big(1+|x_r^{i,N,n}|\big)^{4\eta}\hspace{-0.03cm}+\hspace{-0.03cm}E\big|\Btamer\big|^4\Big\} drds+Kn^{-2}  \notag
	\\
	\leq &K \int_{0}^{t}E|e_s^{iu,N,n}|^2ds+ Kn^{-2}
	\end{align}
for any $t\in[0,T]$,  $u\in\{1,\ldots,d\}$ and $i\in\{1,\ldots,N\}$.
For $\xi_3$, by applying Young's inequality, H\"older's inequality, Corollaries \ref{cor:sig:gam}, \ref{cor1:sig:gam} and Lemma \ref{lem:scm:mb}, one gets
\begin{align} \label{eq:xi3}
	\xi_3:= &K E\int_{0}^{t} |e_s^{iu,N,n}|\int_{\kappa_n(s)}^s\int_Z \Big\{ \big(1+|x_r^{i,N,n}|\big)^{\eta}\big|\Gammatame\big| \notag
	\\ 
	&\qquad\qquad\qquad\qquad +\frac{1}{N}\sum_{k=1}^{N} \big|\Gammatamek\big|\Big\}\nu(d\bar z)drds \notag
	\\
	\leq & K \int_{0}^{t}E|e_s^{iu,N,n}|^2ds+KE\int_{0}^{t}\Big|\int_{\kappa_n(s)}^s\int_Z \Big\{ \big(1+|x_r^{i,N,n}|\big)^{\eta}\big|\Gammatame\big| \notag
	\\
	&\qquad\qquad\qquad \qquad +\frac{1}{N}\sum_{k=1}^{N} \big|\Gammatamek\big|\Big\}\nu(d\bar z)dr\Big|^2 ds\notag
	\\
	\leq &  K \int_{0}^{t}E|e_s^{iu,N,n}|^2ds+Kn^{-1}     E\int_{0}^{t}\int_{\kappa_n(s)}^s\int_Z\Big\{\big(1+|x_r^{i,N,n}|\big)^{2\eta}\big|\Gammatame\big|^2 \notag
	\\
	&\qquad\qquad\qquad\qquad +\frac{1}{N} \sum_{k=1}^{N}\big|\Gammatamek\big|^2\Big\} \nu(d\bar z) dr ds  \notag
	\\
	\leq & K \int_{0}^{t}E|e_s^{iu,N,n}|^2ds+ Kn^{-1}      \int_{0}^{t}\int_{\kappa_n(s)}^s\Big\{E\big(1+|x_r^{i,N,n}|\big)^{4\eta} \notag
	\\
	&\qquad\qquad\qquad\qquad+\int_ZE\big|\Gammatame\big|^4 \nu(d\bar z)\Big\}  dr ds  +Kn^{-2} \notag
	\\
	\leq & K \int_{0}^{t}E|e_s^{iu,N,n}|^2ds+ Kn^{-2}
\end{align}
for any   $t\in[0,T]$,  $u\in\{1,\ldots,d\}$  and $i\in\{1,\ldots,N\}$.	
To estimate $\xi_4$, one uses Young's inequality, H\"older's inequality,  Corollaries \ref{cor:sig:gam}, \ref{cor1:sig:gam} and
Lemma  \ref{lem:scm:mb}, one obtains
 \color{black}
\begin{align} \label{eq:xi4}
	\xi_4:=& KE\int_{0}^{t} |e_s^{iu,N,n}|\int_{\kappa_n(s)}^{s}\Big\{ \big(1+|x_r^{i,N,n}|\big)^{\eta} \big|\Lambdatam\big|^2  \notag
	\\ 
	&\qquad\qquad\qquad\qquad +\frac{1}{N}\sum_{k=1}^N    \big|\Lambdatamk\big|^2 \notag
	\Big\}dr ds\notag
	\\
	\leq& KE\int_{0}^{t} |e_s^{iu,N,n}|^2ds+KE\int_{0}^{t}\Big|\int_{\kappa_n(s)}^{s}\Big\{ \big(1+|x_r^{i,N,n}|\big)^{\eta} \big|\Lambdatam\big|^2  \notag
	\\ 
	&\qquad\qquad\qquad\qquad +\frac{1}{N}\sum_{k=1}^N    \big|\Lambdatamk\big|^2 \notag
	\Big\}dr\Big|^2 ds\notag
		\\
		\leq &  K \int_{0}^{t}E|e_s^{iu,N,n}|^2ds+Kn^{-1}     E\int_{0}^{t}\int_{\kappa_n(s)}^s\Big\{\big(1+|x_r^{i,N,n}|\big)^{2\eta}\big|\Lambdatam\big|^4 \notag
		\\
		&\qquad\qquad\qquad\qquad +\frac{1}{N} \sum_{k=1}^{N}\big|\Lambdatamk\big|^4\Big\}  dr ds  \notag
		\\
		\leq & K \int_{0}^{t}E|e_s^{iu,N,n}|^2ds+ Kn^{-1}      \int_{0}^{t}\int_{\kappa_n(s)}^s\Big\{E\big(1+|x_r^{i,N,n}|\big)^{\frac{2\eta(\varepsilon+2)}{\varepsilon}}  \notag
		\\
		&\qquad\qquad\qquad\qquad+\int_ZE\big|\Lambdatam\big|^{\frac{4(\varepsilon+2)}{2}} \nu(d\bar z)\Big\}  dr ds  +Kn^{-2} \notag
		\\
		\leq & K \int_{0}^{t}E|e_s^{iu,N,n}|^2ds+ Kn^{-2}
\end{align}
for any   $t\in[0,T]$,  $u\in\{1,\ldots,d\}$  and $i\in\{1,\ldots,N\}$.
For $\xi_5$, one needs to establish the following
\begin{align} 
\sum_{k=1}^{N}&\Big\{b^{u}\big(x_{r}^{i,N,n}+1_{\{k=i\}}\Gammatamek,\tilpikappaGammatame\big) 
-b^{u}( x_{r}^{i,N,n},\mu_{r}^
{x,N,n})\Big\}  \notag
\\
=&b^{u}\big(x_{r}^{i,N,n}+\Gammatame,\tilpikappaGammatamei\big) 
-b^{u}\big( x_{r}^{i,N,n},\tilpikappaGammatamei\big)   \notag
\\
&+b^{u}\big( x_{r}^{i,N,n},\tilpikappaGammatamei\big) +\sum_{k\neq i}b^{u}(x_{r}^{i,N,n},\tilpikappaGammatame) 
-\sum_{k=1}^{N}b^{u}\big( x_{r}^{i,N,n},\mu_{r}^
{x,N,n}\big)  \notag
\\
=&b^{u}\big(x_{r}^{i,N,n}+\Gammatame,\tilpikappaGammatamei\big) 
-b^{u}\big( x_{r}^{i,N,n},\tilpikappaGammatamei\big)   \notag
\\
&+\sum_{k=1}^{N}\Big\{b^{u}\big(x_{r}^{i,N,n},\tilpikappaGammatame\big) 
-b^{u}\big( x_{r}^{i,N,n},\mu_{r}^
{x,N,n}\big) \Big \}  \notag
\end{align}
which on using in $\xi_5$ along with Young's inequality, H\"older's inequality and Equations \eqref{eq:b:estem}, \eqref{eq:b:estem*} yield
\begin{align}  \label{eq:xi5}
	\xi_5:=&E\int_{0}^{t} |e_s^{iu,N,n}|\sum_{k=1}^{N}\int_{\kappa_{n}(s)}^{s}\int_Z\Big|b^{u}\big(x_{r}^{i,N,n}+1_{\{k=i\}}\Gammatamek,\tilpikappaGammatame\big) \notag
	\\
	&\qquad\qquad\qquad\qquad
	-b^{u}\big( x_{r}^{i,N,n},\mu_{r}^
	{x,N,n}\big)\Big| \nu(d\bar z)dr ds  \notag
	\\
	\leq&E\int_{0}^{t} |e_s^{iu,N,n}|\int_{\kappa_{n}(s)}^{s}\int_Z\Big|b^{u}\big(x_{r}^{i,N,n}+\Gammatame,\tilpikappaGammatamei\big) \notag
	\\
	&\qquad\qquad\qquad\qquad
	-b^{u}\big( x_{r}^{i,N,n},\tilpikappaGammatamei\big)\Big| \nu(d\bar z)dr ds  \notag
	\\
	&+E\int_{0}^{t} |e_s^{iu,N,n}|\sum_{k=1}^{N}\int_{\kappa_{n}(s)}^{s}\int_Z\Big|b^{u}\big(x_{r}^{i,N,n},\tilpikappaGammatame\big) 
	-b^{u}\big( x_{r}^{i,N,n},\mu_{r}^
	{x,N,n}\big)\Big| \nu(d\bar z)dr ds   \notag
	\\
	\leq &KE\int_{0}^{t} |e_s^{iu,N,n}|^2ds
    +Kn^{-1}E\int_{0}^{t}\int_{\kappa_{n}(s)}^{s}\int_Z\Big| b^{u}\big( x_{r}^{i,N,n},\tilpikappaGammatamei\big)\notag
	\\
	&\qquad\qquad
	-b^{u}\big(x_{r}^{i,N,n}+\Gammatame,\tilpikappaGammatamei\big)\Big|^2 \nu(d\bar z)dr ds  \notag
	\\
	&+Kn^{-1}N\sum_{k=1}^{N}E\int_{0}^{t} \int_{\kappa_{n}(s)}^{s}\int_Z\big|b^{u}\big(x_{r}^{i,N,n},\tilpikappaGammatame\big) 
	-b^{u}\big( x_{r}^{i,N,n},\mu_{r}^
	{x,N,n}\big)\big|^2 \nu(d\bar z)dr ds   \notag
	\\
	\leq& KE\int_{0}^{t} |e_s^{iu,N,n}|^2ds+Kn^{-2}
\end{align} 
for any  $t\in[0,T]$, $u\in\{1,\ldots,d\}$ and $i\in\{1,\ldots,N\}$. 
For $\xi_6$,  recall the interacting particle system  \eqref{eq:int} and the Milstein-type scheme \eqref{eq:scm**} to write
\begin{align} \label{eq:xi6*}
	&\xi_6:=E\int_{0}^{t} e_s^{iu,N,n} \sum_{\mathfrak{q}=1}^2 \Phi^{iu,N,n}_{\mathfrak{q}}\big(\kappa_n(s),s\big) ds
	=E\int_{0}^{t}e_{\kappa_n(s)}^{iu,N,n}\sum_{\mathfrak{q}=1}^2 \Phi^{iu,N,n}_{\mathfrak{q}}\big(\kappa_n(s),s\big)ds \notag
	\\
	&
	+E\int_0^t\int_{\kappa_{n}(s)}^s\big\{\sigma^{(u)}(x_r^{i,N},\mu_r^{x,N})-\sigma^{(u)}(x_r^{i,N,n},\mu_r^{x,N,n})\big\}dw^i_r\sum_{\mathfrak{q}=1}^2 \Phi^{iu,N,n}_{\mathfrak{q}}\big(\kappa_n(s),s\big) ds \notag
	\\
	&
	+E\int_0^t\int_{\kappa_{n}(s)}^s\Big\{\sigma^{(u)}(x_r^{i,N,n},\mu_r^{x,N,n})-\Lambdatamru\Big\}dw^i_r\sum_{\mathfrak{q}=1}^2 \Phi^{iu,N,n}_{\mathfrak{q}}\big(\kappa_n(s),s\big) ds \notag
	\\
	&+E\int_0^t\int_{\kappa_{n}(s)}^s\int_Z\big\{\gamma^u(x_r^{i,N},\mu_r^{x,N},\bar z)-\gamma^u(x_r^{i,N,n},\mu_r^{x,N,n},\bar z)\big\}\tilde n_p^i(dr,d\bar z)\sum_{\mathfrak{q}=1}^2 \Phi^{iu,N,n}_{\mathfrak{q}}\big(\kappa_n(s),s\big)ds \notag
	\\
	&+E\int_0^t\int_{\kappa_{n}(s)}^s\int_Z\Big\{\gamma^u(x_r^{i,N,n},\mu_r^{x,N,n},\bar z)-\Gammatameru\Big\}\tilde n_p^i(dr,d\bar z) \notag
	\\
	&\qquad\qquad\qquad\qquad\qquad\qquad \times\sum_{\mathfrak{q}=1}^2\Phi^{iu,N,n}_{\mathfrak{q}}\big(\kappa_n(s),s\big)ds \notag
	\\
	&
	+E\int_0^t\int_{\kappa_{n}(s)}^s\big\{b^{u}(x_r^{i,N},\mu_r^{x,N})-(\widehat {b}^u)_{tam}^{\displaystyle n}\big(x_{\kappa_{n}(r)}^{i,N,n},\mu_{\kappa_{n}(r)}^{x,N,n}\big)\big\}dr\sum_{\mathfrak{q}=1}^2 \Phi^{iu,N,n}_{\mathfrak{q}}\big(\kappa_n(s),s\big) ds \notag
	\\
&\quad=:\xi_{61}+\xi_{62}+\xi_{63}+\xi_{64}+\xi_{65}+\xi_{66}
\end{align}
for any $t\in[0,T]$, $u\in\{1,\ldots,d\}$ and $i\in\{1,\ldots,N\}$.
For  $\xi_{61}$, by recalling Equations \eqref{eq:phi1} and \eqref{eq:phi2}, one obtains
\begin{align} \label{eq:xi61}
	\xi_{61}:= E\int_{0}^{t}e_{\kappa_n(s)}^{iu,N,n}\sum_{\mathfrak{q}=1}^2 \Phi^{iu,N,n}_{\mathfrak{q}}\big(\kappa_n(s),s\big)ds
	=E\Big\{\int_{0}^{t}e_{\kappa_n(s)}^{iu,N,n}\sum_{\mathfrak{q}=1}^2 E_{\kappa_{n}(s)}\Phi^{iu,N,n}_{\mathfrak{q}}\big(\kappa_n(s),s\big)ds\Big\}=0
\end{align}
for any $t\in[0,T]$, $u\in\{1,\ldots,d\}$ and $i\in\{1,\ldots,N\}$.
To estimate $\xi_{62}$, recall Equations \eqref{eq:phi1} and \eqref{eq:phi2}  to obtain
\begin{align*}
	\xi_{62}:=
	&E\int_0^t\int_{\kappa_{n}(s)}^s\big\{\sigma^{(u)}\big(x_r^{i,N},\mu_r^{x,N}\big)-\sigma^{(u)}\big(x_r^{i,N,n},\mu_r^{x,N,n}\big)\big\}dw^i_r\sum_{\mathfrak{q}=1}^2 \Phi^{iu,N,n}_{\mathfrak{q}}\big(\kappa_n(s),s\big) ds \notag
	\\
	=&E\int_0^t\bigg[\int_{\kappa_{n}(s)}^s\big\{\sigma^{(u)}\big(x_r^{i,N},\mu_r^{x,N}\big)-\sigma^{(u)}\big(x_r^{i,N,n},\mu_r^{x,N,n}\big)\big\}dw^i_r 
	\\
	&\times\Big\{  \int_{\kappa_n(s)}^s \partial_x b^{u}\big( x_{r}^{i,N,n},\mu_{r}^{x,N,n}\big)\Lambdatam dw^{i}_r \notag
	\\
	&\qquad+\frac{1}{N}\sum_{k= 1}^{N} \int_{\kappa_n(s)}^s\partial_\mu b^{u}\big( x_{r}^{i,N,n},\mu_{r}^{x,N,n},x_{r}^{k,N,n}\big)\Lambdatamk dw^{k}_r \Big\} \bigg]ds \notag
		\\
	\leq &E\int_0^t\int_{\kappa_{n}(s)}^s\Big[\big|\sigma^{(u)}\big(x_r^{i,N},\mu_r^{x,N}\big)-\sigma^{(u)}\big(x_r^{i,N,n},\mu_r^{x,N,n}\big)\big|
	 \notag
	\\
	&\times\Big\{\big|\partial_x b^{u}\big( x_{r}^{i,N,n},\mu_{r}^{x,N,n}\big)	\big|+	\big|\partial_\mu b^{u}\big( x_{r}^{i,N,n},\mu_{r}^{x,N,n},x_{r}^{i,N,n}\big)\big|\Big\} \big|\Lambdatam  \big| \Big]drds \notag
\end{align*}
for any $t\in[0,T]$, $u\in\{1,\ldots,d\}$ and $i\in\{1,\ldots,N\}$ where the term involving $\Phi_{2}^{iu,N,n}$ is zero as $w^i$ and $n_p^k$ are independent for $k\in\{1,\ldots,N\}$ .
Further, one uses Remarks \ref{rem:poly:sig:gam}, \ref{rem:super_linear}, Assumption \ref{asum:measure:derv:bound}, Young's inequality, H\"older's inequality, Equation \eqref{eq:empirical_dif} together with Proposition \ref{prop:mb:mvsde}, Lemma \ref{lem:scm:mb} and Corollaries \ref{cor:sig:gam}, \ref{cor1:sig:gam} to get
\begin{align} \label{eq:xi62}
	\xi_{62}
	\leq & KE\int_{0}^{t}\int_{\kappa_{n}(s)}^s\Big[\big\{\big(1+|x_r^{i,N}|+|x_r^{i,N,n}|\big)^{\frac{\eta}{2}}|x_r^{i,N}-x_r^{i,N,n}|+\mathcal W_2(\mu_{r}^{x,N},\mu_{r}^{x,N,n})\big\} \notag
	\\
	&\times \big(1+|x_r^{i,N,n}|\big)^{\eta} \big|\Lambdatam\big|   \Big]drds \notag
	\\
	\leq & KE\int_{0}^{t}\int_{\kappa_{n}(s)}^s\Big[n^{\frac{1}{2}}\big\{|x_r^{i,N}-x_r^{i,N,n}|+\mathcal W_2(\mu_{r}^{x,N},\mu_{r}^{x,N,n})\big\} \notag
	\\
	&\times n^{-\frac{1}{2}}\big(1+|x_r^{i,N}|+|x_r^{i,N,n}|\big)^{\frac{3\eta}{2}}\big|\Lambdatam\big|   \Big]drds \notag
	\\
	\leq & KnE\int_{0}^{t}\int_{\kappa_{n}(s)}^s\big\{|x_r^{i,N}-x_r^{i,N,n}|^2+\mathcal W_2^2(\mu_{r}^{x,N},\mu_{r}^{x,N,n})\big\}drds \notag
	\\
	&+ Kn^{-1}E\int_{0}^{t}\int_{\kappa_{n}(s)}^s\big(1+|x_r^{i,N}|+|x_r^{i,N,n}|\big)^{3\eta} \big|\Lambdatam\big|^2drds \notag
	\\
	\leq& K\int_{0}^{t}\sup_{i \in \{1,\ldots,N\}}\sup_{r\in[0,s]}E|e_r^{i,N,n}|^2ds + Kn^{-1}\int_{0}^{t}\int_{\kappa_{n}(s)}^s \big\{E\big(1+|x_r^{i,N}|+|x_r^{i,N,n}|\big)^{6\eta}\big\}^{\frac{1}{2}} \notag
	\\
	&\times\big\{E\big|\Lambdatam\big|^4\big\}^{\frac{1}{2}} drds
	\leq K\int_{0}^{t}\sup_{i \in \{1,\ldots,N\}}\sup_{r\in[0,s]}E|e_r^{i,N,n}|^2ds + Kn^{-2}
\end{align}
for any $t\in[0,T]$, $u\in\{1,\ldots,d\}$ and $i\in\{1,\ldots,N\}$.
For estimating  $\xi_{63}$, one uses H\"older's inequality,  Lemma \ref{lem:sig} and Equations \eqref{eq:phi1*}, \eqref{eq:phi2*} to get
\begin{align} \label{eq:xi63}
	\xi_{63}:= & E\int_0^t\int_{\kappa_{n}(s)}^s\Big\{\sigma^{(u)}\big(x_r^{i,N,n},\mu_r^{x,N,n}\big)-\Lambdatamru\Big\}dw^i_r\sum_{\mathfrak{q}=1}^2 \Phi^{iu,N,n}_{\mathfrak{q}}\big(\kappa_n(s),s\big) ds \notag
	\\
	\leq & K\int_0^t\Big[E\Big|\int_{\kappa_{n}(s)}^s\Big\{\sigma^{(u)}\big(x_r^{i,N,n},\mu_r^{x,N,n}\big)-\Lambdatamru\Big\}dw^i_r\Big|^2\Big]^{\frac{1}{2}} \notag
	\\
	&\qquad\times  \Big[\sum_{\mathfrak{q}=1}^2E\big|\Phi^{iu,N,n}_{\mathfrak{q}}\big(\kappa_n(s),s\big)\big|^2\Big]^{\frac{1}{2}} ds 
	\notag
	\\
	\leq & K\int_0^t\Big[E\int_{\kappa_{n}(s)}^s\big|\sigma^{(u)}\big(x_r^{i,N,n},\mu_r^{x,N,n}\big)-\Lambdatamru\big|^2dr\Big]^{\frac{1}{2}} \notag
	\\
	&\qquad\times  \Big[\sum_{\mathfrak{q}=1}^2E\big|\Phi^{iu,N,n}_{\mathfrak{q}}\big(\kappa_n(s),s\big)\big|^2\Big]^{\frac{1}{2}} ds \leq Kn^{-\frac{3}{2}-\frac{1}{\varepsilon+2}}\leq Kn^{-1-\frac{2}{\varepsilon+2}}
\end{align}
for any $t\in[0,T]$, $u\in\{1,\ldots,d\}$ and $i\in\{1,\ldots,N\}$.
For $\xi_{64}$, recall Equations \eqref{eq:phi1} and \eqref{eq:phi2} 
  to write
\begin{align*}
	\xi_{64}:=
	&E\int_0^t\int_{\kappa_{n}(s)}^s\int_Z\Big\{\gamma^u\big(x_r^{i,N},\mu_r^{x,N},\bar z\big)-\gamma^u\big(x_r^{i,N,n},\mu_r^{x,N,n},\bar z\big)\Big\}\tilde n_p^i(dr,d\bar z)\sum_{\mathfrak{q}=1}^2 \Phi^{iu,N,n}_{\mathfrak{q}}\big(\kappa_n(s),s\big)ds \notag
	\\
	=&E\int_0^t\bigg[\int_{\kappa_{n}(s)}^s\int_Z\Big\{\gamma^u\big(x_r^{i,N},\mu_r^{x,N},\bar z\big)-\gamma^u\big(x_r^{i,N,n},\mu_r^{x,N,n},\bar z\big)\Big\}\tilde n_p^i(dr,d\bar z)
	\\
	&\times\sum_{k=1}^{N} \int_{\kappa_{n}(s)}^{s}\int_Z\Big\{b^{u}\big(x_{r-}^{i,N,n}+1_{\{k=i\}}\Gammatamek,\tilpikappaGammatameve\big) \notag
	\\
	&\qquad\qquad\qquad\qquad
	-b^{u}\big( x_{r-}^{i,N,n},\mu_{r-}^
	{x,N,n}\big)\Big\} \tilde n_p^k(dr,d\bar z) \bigg]ds \notag
\end{align*}
where the terms involving $\Phi_{1}^{iu,N,n}$ are zero as $n_p^i$ and $w^k$ are independent for $k\in\{1,\ldots,N\}$, which on using H\"older's inequality yields
\begin{align*}
	\xi_{64}&\leq E\int_0^t\int_{\kappa_{n}(s)}^s\int_Z\Big[\big|\gamma^u\big(x_r^{i,N},\mu_r^{x,N},\bar z\big)-\gamma^u\big(x_r^{i,N,n},\mu_r^{x,N,n},\bar z\big)\big|
	\\
	&\,\, \,\, \times \big|b^{u}\big(x_{r}^{i,N,n}+\Gammatame,\tilpikappaGammatamei\big) 
	-b^{u}\big( x_{r}^{i,N,n},\mu_{r}^
	{x,N,n}\big)\big|  \Big] \nu(d\bar z) drds \notag
		\\
		&\leq E\int_0^t\int_{\kappa_{n}(s)}^s\bigg[\Big\{\int_Z\big|\gamma^u\big(x_r^{i,N},\mu_r^{x,N},\bar z\big) \hspace{-0.03cm}-\hspace{-0.03cm}\gamma^u\big(x_r^{i,N,n},\mu_r^{x,N,n},\bar z\big)\big|^2\nu(d\bar z)\Big\}^{\frac{1}{2}}
	\\
	&\times \Big\{\int_Z\big|b^{u}\big(x_{r}^{i,N,n}+\Gammatame,\tilpikappaGammatamei\big) 
	\\
	&\qquad-b^{u}\big( x_{r}^{i,N,n},\mu_{r}^
	{x,N,n}\big)\big|^2\nu(d\bar z) \Big\}^{\frac{1}{2}} \bigg]  drds \notag
\end{align*}
for any $t\in[0,T]$, $u\in\{1,\ldots,d\}$ and $i\in\{1,\ldots,N\}$.
Further, one uses Remarks \ref{rem:poly:sig:gam}, \ref{rem:super_linear} and Equation \eqref{eq:empirical_dif}  to get
 \begin{align*} 
	\xi_{64}
	\leq & KE\int_{0}^{t}\int_{\kappa_{n}(s)}^s\Big[\Big\{\big(1+|x_r^{i,N}|+|x_r^{i,N,n}|\big)^{\frac{\eta}{2}}|x_r^{i,N}-x_r^{i,N,n}|+\mathcal W_2(\mu_{r}^{x,N},\mu_{r}^{x,N,n})\Big\} \notag
	\\
	&\times\Big\{(1+|x_r^{i,N,n}|)^{\eta+1} +\Big(\int_Z\big|\Gammatame\big|^{2(\eta+1)} \nu(d\bar z)\Big)^{\frac{1}{2}} \notag
	\\
	&\qquad\qquad\qquad+\Big(\int_Z\mathcal{W}_2^2\big(\tilpikappaGammatamei,\delta_0\big) +\mathcal W^2_2(\mu_{r}^{x,N,n},\delta_0) \nu(d\bar z)\Big)^{\frac{1}{2}} \Big\} \Big]drds \notag
	\\
	\leq & KE\int_{0}^{t}\int_{\kappa_{n}(s)}^s\bigg[\Big\{|x_r^{i,N}-x_r^{i,N,n}|+\Big(\frac{1}{N}\sum_{j=1}^{N}|x_r^{j,N}-x_r^{j,N,n}|^2\Big)^{\frac{1}{2}}\Big\}n^{\frac{1}{2}}n^{-\frac{1}{2}}\big(1+|x_r^{i,N}|+|x_r^{i,N,n}|\big)^{\frac{\eta}{2}} \notag
	\\
	&\times\bigg\{(1+|x_r^{i,N,n}|)^{\eta+1} +\Big(\int_Z\big|\Gammatame\big|^{2(\eta+1)} \nu(d\bar z)\Big)^{\frac{1}{2}}  \bigg\} \bigg]drds \notag
	\\
	& +KE\int_{0}^{t}\int_{\kappa_{n}(s)}^s\bigg[\Big\{|x_r^{i,N}-x_r^{i,N,n}|+\Big(\frac{1}{N}\sum_{j=1}^{N}|x_r^{j,N}-x_r^{j,N,n}|^2\Big)^{\frac{1}{2}}\Big\}n^{\frac{1}{2}}n^{-\frac{1}{2}}\big(1+|x_r^{i,N}|+|x_r^{i,N,n}|\big)^{\frac{\eta}{2}} \notag
	\\
	&\times\Big(\int_Z\Big\{ \frac{1}{N}\sum_{j=1}^{N}|x_{r}^{j,N,n}|^{2}+\frac{1}{N}\big|\Gammatame\big|^{2}\Big\} \nu(d\bar z)\Big)^{\frac{1}{2}} \bigg]drds \notag
\end{align*}
which on using Young's inequality yields
	\begin{align*}
	\xi_{64}\leq & KnE\int_{0}^{t}\int_{\kappa_{n}(s)}^s\Big\{|x_r^{i,N}-x_r^{i,N,n}|^2+\frac{1}{N}\sum_{j=1}^{N}|x_r^{j,N}-x_r^{j,N,n}|^2\Big\}drds \notag
	\\
	&+ Kn^{-1}E\int_{0}^{t}\int_{\kappa_{n}(s)}^s\big(1+|x_r^{i,N}|+|x_r^{i,N,n}|\big)^{\eta}
	\\
	&\qquad\qquad\qquad \times  \Big\{(1+|x_r^{i,N,n}|)^{2(\eta+1)} +\int_Z\big|\Gammatame\big|^{2(\eta+1)}\nu(d\bar z)  \Big\} dr ds \notag
	\\
	&+ Kn^{-1}E\int_0^t\int_{\kappa_{n}(s)}^{s}(1+|x_r^{i,N}|+|x_r^{i,N,n}|\big)^{\eta}
	\\
	&\qquad\qquad\qquad\times\Big\{ \frac{1}{N}\sum_{j=1}^{N}|x_{r}^{j,N,n}|^{2}+\int_ Z\big|\Gammatame\big|^{2}\nu(d\bar z)\Big\} dr  ds  \notag
\end{align*}
for any $t\in[0,T]$, $u\in\{1,\ldots,d\}$ and $i\in\{1,\ldots,N\}$.
Moreover, by using H\"older's inequality, Proposition \ref{prop:mb:mvsde}, Lemma \ref{lem:scm:mb} and Corollaries \ref{cor:sig:gam}, \ref{cor1:sig:gam}, one gets
\begin{align} \label{eq:xi64}
	\xi_{64}
	\leq& K\int_{0}^{t}\sup_{i \in \{1,\ldots,N\}}\sup_{r\in[0,s]}E|e_r^{i,N,n}|^2ds + Kn^{-1}\int_{0}^{t}\int_{\kappa_{n}(s)}^s \Big[E\big(1+|x_r^{i,N}|+|x_r^{i,N,n}|\big)^{3\eta+2}\notag
	\\
	&+\big\{E\big(1+|x_r^{i,N}|+|x_r^{i,N,n}|\big)^{\frac{\eta(\varepsilon+2)}{\varepsilon}}\big\}^{\frac{\varepsilon}{\varepsilon+2}} \notag
 \\
 &\quad\times\Big\{E\int_Z\big|\Gammatame\big|^{\frac{2(\eta+1)(\varepsilon+2)}{2}}\nu(d\bar z)\Big\}^{\frac{2}{\varepsilon+2}} \Big] drds \notag
	\\
	&+ Kn^{-1}\int_0^t\int_{\kappa_{n}(s)}^{s}\int_ Z \big\{E(1+|x_r^{i,N}|+|x_r^{i,N,n}|\big)^{{\eta(\varepsilon+2)/\varepsilon}}\big\}^{\frac{\varepsilon}{\varepsilon+2}}\notag
	\\
	&\quad\times\Big\{ \frac{1}{N}\sum_{j=1}^{N}E|x_{r}^{j,N,n}|^{{2(\varepsilon+2)/2}}+E\int_Z\big|\Gammatame\big|^{{2(\varepsilon+2)/2}}\nu(d\bar z)\Big\}^{\frac{2}{\varepsilon+2}} dr  ds  \notag
	\\
	\leq& K\int_{0}^{t}\sup_{i \in \{1,\ldots,N\}}\sup_{r\in[0,s]}E|e_r^{i,N,n}|^2ds + Kn^{-2}
\end{align} 
for any $t\in[0,T]$, $u\in\{1,\ldots,d\}$ and $i\in\{1,\ldots,N\}$.
To estimate $\xi_{65}$, by applying H\"older's inequality, Lemma \ref{lem:gam} and Equations \eqref{eq:phi1*}, \eqref{eq:phi2*} to obtain
\begin{align} \label{eq:xi65}
	\xi_{65}:=& E\int_0^t\int_{\kappa_{n}(s)}^s\int_Z\Big\{\gamma^u\big(x_r^{i,N,n},\mu_r^{x,N,n},\bar z\big)-\Gammatameru\Big\}\tilde n_p^i(dr,d\bar z) \notag
	\\	&\qquad\qquad\times\sum_{\mathfrak{q}=1}^2 \Phi^{iu,N,n}_{\mathfrak{q}}\big(\kappa_n(s),s\big)ds \notag
	\\
	\leq& \int_0^t\Big[E\Big|\int_{\kappa_{n}(s)}^s\int_Z\Big\{\gamma^u\big(x_r^{i,N,n},\mu_r^{x,N,n},\bar z\big)-\Gammatameru\Big\}\tilde n_p^i(dr,d\bar z)\Big|^2\Big]^{\frac{1}{2}} \notag
    \\
    &\qquad\qquad\times  \Big[\sum_{\mathfrak{q}=1}^2E\big|\Phi^{iu,N,n}_{\mathfrak{q}}\big(\kappa_n(s),s\big)\big|^2\Big]^{\frac{1}{2}} ds 
    \notag
    \\
    \leq & K\int_0^t\Big[E\int_{\kappa_{n}(s)}^s\int_Z\big|\gamma^u\big(x_r^{i,N,n},\mu_r^{x,N,n},\bar z\big)-\Gammatameru\big|^2\nu(d\bar z)dr\Big]^{\frac{1}{2}} \notag
    \\
    &\qquad\qquad\times  \Big[\sum_{\mathfrak{q}=1}^2E\big|\Phi^{iu,N,n}_{\mathfrak{q}}\big(\kappa_n(s),s\big)\big|^2\Big]^{\frac{1}{2}} ds \leq Kn^{-\frac{3}{2}-\frac{1}{\varepsilon+2}}\leq Kn^{-1-\frac{2}{\varepsilon+2}}
\end{align}
for any $t\in[0,T]$, $u\in\{1,\ldots,d\}$ and $i\in\{1,\ldots,N\}$.
To estimate $\xi_{66}$, first derive the following  by applying Assumptions  \ref{asum:lip*}, \ref{asum:convergence}, H\"older's inequality, Proposition \ref{prop:mb:mvsde}, Lemma \ref{lem:scm:mb} and Corollary \ref{cor:one:step:error}
\begin{align*} 
	&E\big|b(x_r^{i,N},\mu_r^{x,N})-(\widehat {b})_{tam}^{\displaystyle n}\big(x_{\kappa_{n}(r)}^{i,N,n},\mu_{\kappa_{n}(r)}^{x,N,n}\big)\big|^2 \leq KE\big|b(x_r^{i,N},\mu_r^{x,N})-b(x_r^{i,N,n},\mu_r^{x,N,n})\big|^2
	\\
	&
	+KE\big|b(x_r^{i,N,n},\mu_r^{x,N,n})-b(x_{\kappa_{n}(r)}^{i,N,n},\mu_{\kappa_{n}(r)}^{x,N,n})\big|^2+KE\big|b(x_{\kappa_{n}(r)}^{i,N,n},\mu_{\kappa_{n}(r)}^{x,N,n})-(\widehat {b})_{tam}^{\displaystyle n}\big(x_{\kappa_{n}(r)}^{i,N,n},\mu_{\kappa_{n}(r)}^{x,N,n}\big)\big|^2\notag
	\\
	\leq  & KE\Big(\big(1+|x_r^{i,N}|+|x_r^{i,N,n}|\big)^{2\eta}|x_r^{i,N}-x_r^{i,N,n}|^{2 }\Big)+KE\mathcal W_2^2(\mu_{r}^{x,N},\mu_{r}^{x,N,n}) \notag
	\\
	&+KE\Big((1+|x_r^{i,N,n}|+|x_{\kappa_n(r)}^{i,N,n}|)^{2\eta} |x_r^{i,N,n}-x_{\kappa_n(r)}^{i,N,n}|^{2 }\Big)+KE\mathcal W_2^2(\mu_{r}^{x,N,n},\mu_{\kappa_n(r)}^{x,N,n})+Kn^{-1-\frac{2}{\varepsilon+2}}\notag
	\\
	\leq  & K\Big\{E(1+|x_r^{i,N}|+|x_{r}^{i,N,n}|)^{4\eta+2}\Big\}^{\frac{1}{2}}\Big\{E|x_r^{i,N}-x_r^{i,N,n}|^2\Big\}^{\frac{1}{2}}+\frac{K}{N}\sum_{j=1}^{N} E|x_r^{j,N}-x_{r}^{j,N,n}|^2 \notag
	\\
	&+K\Big\{E\big(1+|x_r^{i,N,n}|+|x_{\kappa_n(r)}^{i,N,n}|\big)^{\frac{2\eta(\varepsilon+2)}{\varepsilon}}\Big\}^{\frac{\varepsilon}{\varepsilon+2}}\big\{E|x_r^{i,N,n}-x_{\kappa_n(r)}^{i,N,n}|^{\varepsilon+2 }\big\}^{\frac{2}{\varepsilon+2}} \notag
	\\
	&+\frac{K}{N}\sum_{j=1}^{N}E|x_r^{j,N,n}-x_{\kappa_n(r)}^{j,N,n}|^{2}+Kn^{-1-\frac{2}{\varepsilon+2}} \notag
	\\
	  \leq&  K \big(\sup_{i \in \{1,\ldots,N\}}\sup_{ s\in[0,r]}E|e_s^{i,N,n}|^2\big)^{\frac{1}{2}}+K  \sup_{i \in \{1,\ldots,N\}}\sup_{ s\in[0,r]}E|e_s^{i,N,n}|^2+Kn^{-\frac{2}{\varepsilon+2}}
\end{align*}
which on using in $\xi_{66}$, Young's inequality and Equations \eqref{eq:phi1*},\eqref{eq:phi2*} yields
\begin{align} \label{eq:xi66}
	\xi_{66}:=&E\int_0^t\int_{\kappa_{n}(s)}^s\Big\{b^{u}(x_r^{i,N},\mu_r^{x,N})-(\widehat {b}^u)_{tam}^{\displaystyle n}\big(x_{\kappa_{n}(r)}^{i,N,n},\mu_{\kappa_{n}(r)}^{x,N,n}\big)\Big\}dr\sum_{\mathfrak{q}=1}^2 \Phi^{iu,N,n}_{\mathfrak{q}}\big(\kappa_n(s),s\big) ds \notag
	\\
	\leq & Kn\int_{0}^{t}E\Big|\int_{\kappa_{n}(s)}^s\Big\{b(x_r^{i,N},\mu_r^{x,N})-(\widehat {b}^u)_{tam}^{\displaystyle n}\big(x_{\kappa_{n}(r)}^{i,N,n},\mu_{\kappa_{n}(r)}^{x,N,n}\big)\Big\}dr\Big|^{2}ds  \notag
 \\
&\qquad\qquad+Kn^{-1}\int_{0}^{t}\sum_{\mathfrak{q}=1}^{2}E\big|\Phi^{iu,N,n}_{\mathfrak{q}}\big(\kappa_n(s),s\big)\big|^2ds \notag
	\\
	\leq & K\int_{0}^{t}E\int_{\kappa_{n}(s)}^s\Big|b(x_r^{i,N},\mu_r^{x,N})-(\widehat {b}^u)_{tam}^{\displaystyle n}\big(x_{\kappa_{n}(r)}^{i,N,n},\mu_{\kappa_{n}(r)}^{x,N,n}\big)\Big|^2dr ds +Kn^{-2}\notag
	\\
	\leq& K\int_{0}^{t}n^{-1}\big(\sup_{i \in \{1,\ldots,N\}}\sup_{ r\in[0,s]}E|e_r^{i,N,n}|^2\big)^{\frac{1}{2}}ds+ K\int_{0}^{t}\sup_{i \in \{1,\ldots,N\}}\sup_{r\in[0,s]}E|e_r^{i,N,n}|^2ds+Kn^{-1-\frac{2}{\varepsilon+2}}  \notag
	\\
	\leq& K\int_{0}^{t}\sup_{i \in \{1,\ldots,N\}}\sup_{r\in[0,s]}E|e_r^{i,N,n}|^2ds+Kn^{-1-\frac{2}{\varepsilon+2}} 
\end{align}
	for any $t\in[0,T]$, $u\in\{1,\ldots,d\}$ and $i\in\{1,\ldots,N\}$.
		Then, by substituting  \eqref{eq:xi61}, \eqref{eq:xi62}, \eqref{eq:xi63}, \eqref{eq:xi64}, \eqref{eq:xi65} and \eqref{eq:xi66} in Equation \eqref{eq:xi6*}, one gets
		\begin{align} \label{eq:xi6}
			\xi_{6}\leq K\int_{0}^{t} \sup_{i \in \{1,\ldots,N\}} \sup_{r\in[0,s]}E|e_r^{i,N,n}|^2ds  +Kn^{-1-\frac{2}{\varepsilon+2}}
		\end{align}
		for any $t\in[0,T]$, $u\in\{1,\ldots,d\}$ and $i\in\{1,\ldots,N\}$.  
		By substituting  Equations \eqref{eq:xi1} to \eqref{eq:xi4}, \eqref{eq:xi5} and \eqref{eq:xi6} in Equation \eqref{eq:xi}, one can obtain the required result. 
		\end{proof}		

 \begin{proof}[Proof of Theorem \ref{thm:mr}]
	Recall  Equation \eqref{eq:error} and use  It\^o's  formula \cite[Theorem 94]{Situ2006} to write
\begin{align*}
	|x_t^{i,N}&-x_t^{i,N,n}|^2=2\int_{0}^{t}(x_s^{i,N}-x_s^{i,N,n})\Big\{b( x_s^{i, N},\mu_s^{x, N})-\Btame\Big\}ds
	\\
	&+2\int_{0}^{t}(x_s^{i,N}-x_s^{i,N,n})\Big\{\sigma(x_s^{i,N},\mu_s^{x,N})-\Lambdatams\Big\}dw^{i}_s
	\\
	&+\int_{0}^{t}\big|\sigma(x_s^{i,N},\mu_s^{x,N})-\Lambdatams\big|^2ds
	\\
	&+2\int_{0}^{t}\int_Z (x_s^{i,N}-x_s^{i,N,n})\Big\{\gamma(x_s^{i,N},\mu_s^{x,N},\bar z)-\Gammatames\Big\}\tilde{n}_p^i(ds,d\bar z)
	\\
	&+\int_{0}^{t}\int_Z \Big[\big|x_s^{i,N}-x_s^{i,N,n}+\gamma(x_s^{i,N},\mu_s^{x,N},\bar z)-\Gammatames\big|^2-|x_s^{i,N}-x_s^{i,N,n}|^2
	\\
	&-2(x_s^{i,N}-x_s^{i,N,n})\Big\{\gamma(x_s^{i,N},\mu_s^{x,N},\bar z)-\Gammatames\Big\}\Big]{n}_p^i(ds,d\bar z)
\end{align*}
which on  taking expectations on both sides and using 
$
|y_1+y_2|^2-|y_1|^2-2y_1y_2=|y_2|^2, y_1,y_2\in \mathbb R^d
$
give
	\begin{align*}
		&E|x_t^{i,N}-x_t^{i,N,n}|^2
		= 2E\int_{0}^{t}(x_s^{i,N}-x_s^{i,N,n})\Big\{b( x_s^{i, N},\mu_s^{x, N})-\Btame\Big\}ds
		\\
		&
		+E\int_{0}^{t}\big|\sigma(x_s^{i,N},\mu_s^{x,N})-\Lambdatams\big|^2ds
		\\
		&+E\int_{0}^{t}\int_Z\big|\gamma(x_s^{i,N},\mu_s^{x,N},\bar z)-\Gammatames\big|^2\nu(d\bar z)ds
	\end{align*}
	for any $t\in[0,T]$ and $i\in\{1,\ldots,N\}$.
	Further, one uses \, 
	$
	|y_1+y_2|^2 \leq  \alpha|y_1|^2 +\frac{\alpha}{\alpha-1}|y_2|^2, \, y_1, y_2 \in \mathbb{R}^{d_1 \times m_1}$
 for some $\alpha>1$ (obtained due to  Young's inequality and Cauchy--Schwarz inequality) as follows
	\begin{align*}
		&E|x_t^{i,N}-x_t^{i,N,n}|^2
		\leq E\int_{0}^{t}\Big[2(x_s^{i,N}-x_s^{i,N,n})\Big\{b( x_s^{i, N},\mu_s^{x, N})-b( x_s^{i, N,n},\mu_s^{x, N,n})\Big\}
		\\
		&+\alpha\big|\sigma(x_s^{i,N},\mu_s^{x,N})-\sigma(x_s^{i,N,n},\mu_s^{x,N,n})\big|^2 +\alpha\int_Z \big|\gamma(x_s^{i,N},\mu_s^{x,N},\bar z)-\gamma(x_s^{i,N,n},\mu_s^{x,N,n},\bar z)\big|^2\nu(d\bar z)\Big]ds
		\\
		&+2E\int_{0}^{t}(x_s^{i,N}-x_s^{i,N,n})\Big\{b( x_s^{i, N,n},\mu_s^{x, N,n})-\Btame\Big\}ds
		\\
		&+\frac{\alpha}{\alpha-1}E\int_{0}^{t}\big|\sigma(x_s^{i,N,n},\mu_s^{x,N,n})-\Lambdatams\big|^2ds
		\\
		&+\frac{\alpha}{\alpha-1}E\int_{0}^{t}\int_Z \big|\gamma(x_s^{i,N,n},\mu_s^{x,N,n},\bar z)-\Gammatames\big|^2\nu(d\bar z)ds
	\end{align*}
which on applying  Assumption \ref{asum:lip} and Lemmas \ref{lem:gam} to \ref{lem:b} give
\begin{align*}
&\sup_{i \in \{1,\ldots,N\}}\sup_{ r\in[0,t]}E|x_r^{i,N}-x_r^{i,N,n}|^2
= K \int_{0}^{t}\sup_{i \in \{1,\ldots,N\}}\sup_{ r\in[0,s]}E|x_r^{i,N}-x_r^{i,N,n}|^2ds + K n^{-1-\frac{2}{\varepsilon+2}}
\end{align*}
for any $t\in[0,T]$ and $i\in\{1,\ldots,N\}$. 
Then, one uses Gr\"onwall's inequality to complete the proof.
\end{proof}

\section*{Acknowledgements}
This research was funded in part by the Austrian Science Fund (FWF) [10.55776/DOC78]. For open access purposes, the authors have applied a CC BY public copyright license to any author-accepted manuscript version arising from this submission.
The first named author gratefully acknowledges the partial financial support from the Basal project FB210005.
The second named author gratefully acknowledges the financial support received from Science and Engineering Research Board under its Core Research Grant No. SER-2308-MTD.  
	%------------------------------------------------

\appendix
\section{Proof of Some Auxiliary Results} \label{sec:app}
In this section, we prove some useful lemmas used in our article.
\begin{proof}[Proof of Lemma \ref{lem:ito}]
We recall the intensity $\nu(Z)<\infty$. 
As $\{n^1_p(ds, dz), \ldots, n^N_p(ds, dz)\}$ are $N$ \textit{i.i.d.} copies of the Poisson random measure $n_p(ds, dz)$, the particles $\{x^{1, N}, \ldots, x^{N, N}\}$ do not jump together and each of them jumps finitely many times in a finite interval. 
Let us denote the jump-times of the $i$-th particle $x^{i, N}$ by   $0=\widetilde {\tau}_0^{\,i}<\widetilde \tau_1^{\,i}<\cdots<\widetilde\tau_{\widetilde m_i}^{\,i}<\cdots$ for $i\in\{1,\ldots,N\}$.
Moreover, assume that $0=\widetilde\tau_0<\widetilde\tau_1<\cdots<\widetilde\tau_{\widetilde m}<\cdots$ represent the combined jump-times of all the $N$ particles where $\widetilde m=\widetilde m_1+\cdots+\widetilde m_N$. 
For proceeding further, let us rewrite Equation \eqref{eq:int} as 
\begin{align*} 
	x_t^{iu, N}=x_0^{iu}+\int_0^t b^u(x_s^{i, N}, \mu_s^{x, N})ds&-\int_0^t \int_Z \gamma^u(x^{i,N}_s,\mu_s^{x,N},z)\nu(dz)ds + \sum_{\ell=1}^{m}\int_0^t\sigma^{u\ell}(x_s^{i, N},\mu_s^{x,N})dw_s^{ i\ell} \notag
	\\
	&+ \int_0^t \int_Z \gamma^u(x^{i,N}_s,\mu_s^{x,N},z){n}_p^i(ds,dz)
\end{align*}
almost surely for any $t \in[ 0, T]$,   $u \in \{1,\ldots, d\}$ and $i \in \{1,\ldots, N\}$. 
Further, consider the following empirical projection of $F$
\begin{align*}
	F^N:\mathbb R^d\times \mathbb R^{d\times N} \ni (x,\mathcal{X})\mapsto F\Big(x,\frac{1}{N}\sum_{j=1}^{N}\delta_{x^j}\Big)
\end{align*}
where   $\mathcal X\hspace{-0.07cm}:=\hspace{-0.07cm}(x^{1},\ldots,x^{N})\hspace{-0.06cm}\in \hspace{-0.06cm}\mathbb R^{d\times N}$.
Also,   $\partial_x F^N:=(\partial_{x_u} F^N)_{u \in \{1,\ldots, d\}}$, and $\partial^2_xF^N:=(\partial_{x_v}\partial_{x_{u}}F^N)_{u,v\in\{1,\ldots, d\}}$ are used to denote the gradient and  Hessian matrix of $F^N$ with respect to the first component of $F^N$. 
Similarly, $\partial_{x^k} F^N:=(\partial_{x_u^k} F^N)_{u \in \{1,\ldots, d\}}$ and $\partial^2_{x^k} F^N:=(\partial_{x_v^k} \partial_{x_u^k} F^N)_{u,v \in \{1,\ldots, d\}}$ stand for the gradient and  Hessian matrix of $F^N$ with respect to the $k$-th component $x^k$ of $\mathcal{X}$ for  $k \in \{1,\ldots, N\}$. 
Now, we write
\begin{align}
	F^N(x_t^{i,N},\mathcal X_t^{N})-&	F^N(x_0^{i,N},\mathcal X_0^{N})=\sum_{m=1}^{\infty}\big(	F^N(x_{\widetilde\tau_m\wedge t-}^{i,N},\mathcal X_{\widetilde\tau_m\wedge t-}^{N})-	F^N(x_{\widetilde\tau_{m-1}\wedge t}^{i,N},\mathcal X_{\widetilde\tau_{m-1}\wedge t}^{N})\big) \notag
	\\
	&+\sum_{m=1}^{\infty}\big(	F^N(x_{\widetilde\tau_m\wedge t}^{i,N},\mathcal X_{\widetilde\tau_m\wedge t}^{N})-	F^N(x_{\widetilde\tau_m\wedge t-}^{i,N},\mathcal X_{\widetilde\tau_m\wedge t-}^{N})\big) \label{eq:cj*}
\end{align}
almost surely  where $\mathcal{X}_t^N:=(x_t^{1, N}, \ldots, x_t^{N, N})$ and $\widetilde\tau_m\wedge t-= (\widetilde\tau_m\wedge t)-$ for any $t\geq 0$. 
For the first term on the right side of Equation \eqref{eq:cj*}, notice that there is no jump in the interval $[\widetilde\tau_{m-1}\wedge t, \widetilde\tau_m\wedge t)$. 
Thus, one uses the It\^o's formula  (see for example, Theorem 94 in \cite{Situ2006}) to obtain the following
\begin{align*} 
	\sum_{m=1}^{\infty}\big(&F^N(x_{\widetilde\tau_m\wedge t-}^{i,N},  \mathcal X_{\widetilde\tau_m\wedge t-}^{N}) -	F^N(x_{\widetilde\tau_{m-1}\wedge t}^{i,N},\mathcal X_{\widetilde\tau_{m-1}\wedge t}^{N})\big)  \notag
	\\
	 =&\sum_{m=1}^{\infty}\Big\{\sum_{u=1}^d \int_{\widetilde\tau_{m-1}\wedge t}^{\widetilde\tau_{m}\wedge t}\partial_{x_{u}} 	F^N(x_s^{i,N},\mathcal X_s^{N})b^u(x_s^{i, N}, \mu_s^{x, N})ds
	\\
	&-\sum_{u=1}^d\int_{\widetilde\tau_{m-1}\wedge t}^{\widetilde\tau_{m}\wedge t}\int_Z\partial_{x_{u}}	F^N(x_s^{i,N},\mathcal X_s^{N})\gamma^u(x_s^{i, N}, \mu_s^{x, N},z) \nu(dz) ds
	\notag
	\\
	&+\sum_{k=1}^N\sum_{u=1}^d \int_{\widetilde\tau_{m-1}\wedge t}^{\widetilde\tau_{m}\wedge t}\partial_{x^{k}_u} 	F^N(x_s^{i,N},\mathcal X_s^{N})b^u(x_s^{k, N}, \mu_s^{x, N})ds \notag
	\\
	& -\sum_{k=1}^N\sum_{u=1}^d\int_{\widetilde\tau_{m-1}\wedge t}^{\widetilde\tau_{m}\wedge t}\int_Z\partial_{x^{k}_u}	F^N(x_s^{i,N},\mathcal X_s^{N})\gamma^u(x_s^{k, N}, \mu_s^{x, N},z)\nu(dz) ds	\notag
	\\
	&+\sum_{u=1}^d\sum_{\ell=1}^m\int_{\widetilde\tau_{m-1}\wedge t}^{\widetilde\tau_{m}\wedge t}\partial_{x_{u}}	F^N(x_s^{i,N},\mathcal X_s^{N})\sigma^{u\ell}(x_s^{i, N}, \mu_s^{x, N})dw_s^{i\ell} \notag
	\\
	& 
	+\sum_{k=1}^N\sum_{u=1}^d\sum_{\ell=1}^m\int_{\widetilde\tau_{m-1}\wedge t}^{\widetilde\tau_{m}\wedge t}\partial_{x^{k}_u}	F^N(x_s^{i,N},\mathcal X_s^{N})\sigma^{u\ell}(x_s^{k, N}, \mu_s^{x, N})dw_s^{k\ell} \notag
	\\
	&+\frac{1}{2}\sum_{u,v=1}^d\int_{\widetilde\tau_{m-1}\wedge t}^{\widetilde\tau_{m}\wedge t}\partial_{x_{u}}\partial_{x_{v}} 	F^N(x_s^{i,N},\mathcal X_s^{N})\sigma^{(u)}(x_s^{i, N}, \mu_s^{x, N})\sigma^{(v)}(x_s^{i, N}, \mu_s^{x, N})ds \notag
	\\
	&+\sum_{u,v=1}^d\int_{\widetilde\tau_{m-1}\wedge t}^{\widetilde\tau_{m}\wedge t}\partial_{x_{u}}\partial_{x^{i}_v} 	F^N(x_s^{i,N},\mathcal X_s^{N})\sigma^{(u)}(x_s^{i, N}, \mu_s^{x, N})\sigma^{(v)}(x_s^{i, N}, \mu_s^{x, N})ds \notag
	\\
	&+\frac{1}{2}\sum_{k=1}^N\sum_{u,v=1}^d\int_{\widetilde\tau_{m-1}\wedge t}^{\widetilde\tau_{m}\wedge t}\partial_{x^{k}_u}\partial_{x^{k}_v} 	F^N(x_s^{i,N},\mathcal X_s^{N})\sigma^{(u)}(x_s^{k, N}, \mu_s^{x, N})\sigma^{(v)}(x_s^{k, N}, \mu_s^{x, N})ds\Big\} 
\end{align*}
almost surely for any $t\in[ 0, T]$ and $i\in\{1,\ldots,N\}$ which can be rewritten in the following form
\begin{align*}
	&\sum_{m=1}^{\infty}\big(	F^N(x_{\widetilde\tau_m\wedge t-}^{i,N},  \mathcal X_{\widetilde\tau_m\wedge t-}^{N}) -	F^N(x_{\widetilde\tau_{m-1}\wedge t}^{i,N},\mathcal X_{\widetilde\tau_{m-1}\wedge t}^{N})\big) =\int_{0}^{t}\partial_xF^N(x_s^{i,N},\mathcal X_s^N)b(x_s^{i,N},\mu_s^{x,N})ds \notag
	\\
	&-\int_{0}^{t}\int_Z\partial_x F^N(x_s^{i,N},\mathcal X_s^N)\gamma(x_s^{i,N},\mu_s^{x,N},z)\nu(dz) ds+\sum_{k=1}^N\int_{0}^{t}\partial_{x^k}F^N(x_s^{i,N},\mathcal X_s^N)b(x_s^{k,N},\mu_s^{x,N})ds   \notag
	\\
	& -\sum_{k=1}^N\int_{0}^{t}\int_Z\partial_{x^k}F^N(x_s^{i,N},\mathcal X_s^N)\gamma(x_s^{k,N},\mu_s^{x,N},z)\nu(dz) ds +\sum_{\ell=1}^m\int_{0}^{t}\partial_xF^N(x_s^{i,N},\mathcal X_s^N)\sigma^{\ell}(x_s^{i,N},\mu_s^{x,N})dw_s^{i\ell} \notag
	\\
	&
	+\sum_{\ell=1}^m\sum_{k=1}^N\int_{0}^{t}\partial_{x^k}F^N(x_s^{i,N},\mathcal X_s^N)\sigma^{\ell}(x_s^{k,N},\mu_s^{x,N})dw_s^{k\ell}\notag
	\\
	&+\frac{1}{2}\int_{0}^{t}\tr[\partial_x^2 F^N(x_s^{i,N},\mathcal X_s^N)\sigma(x_s^{i,N},\mu_s^{x,N})\sigma^*(x_s^{i,N},\mu_s^{x,N})]ds \notag
	\\
	&+\int_{0}^{t}\tr[\partial_{x}\partial_{x^i} F^N(x_s^{i,N},\mathcal X_s^N)\sigma(x_s^{i,N},\mu_s^{x,N})\sigma^*(x_s^{i,N},\mu_s^{x,N})]ds \notag
	\\
	&+\frac{1}{2}\sum_{k=1}^N\int_{0}^{t}\tr[\partial_{x^k}^2 F^N(x_s^{i,N},\mathcal X_s^N)\sigma(x_s^{k,N},\mu_s^{x,N})\sigma^*(x_s^{k,N},\mu_s^{x,N})]ds 
\end{align*}
almost surely for any $t\in[0,T]$ and $i\in\{1,\ldots,N\}$.
For expressing terms on the right side of the above equation in Lions derivatives,  we make use of   Proposition 5.35 and Proposition 5.91 from \cite{Carmona2018-I} and write
\begin{align*}
	&\partial_{x^k}F^N(x,x^1, \ldots, x^N)
	=\frac{1}{N}\partial_{\mu}F\Big(x,\frac{1}{N}\sum_{j=1}^{N}\delta_{x^j},x^{k}\Big),
	\\
	&\partial_{x^k}^2F^N(x,x^1, \ldots, x^N)
	=\frac{1}{N}\partial_{y}\partial_{\mu}F\Big(x,\frac{1}{N}\sum_{j=1}^{N}\delta_{x^j},x^{k}\Big)
	+\frac{1}{N^2}\partial_{\mu}^2F\Big(x,\frac{1}{N}\sum_{j=1}^{N}\delta_{x^j},x^{k},x^{k}\Big)
\end{align*}
for any $k\in\{1,\ldots,N\}$, $x\in\mathbb{R}^{d }$ and $(x^1, \ldots, x^N)\in \mathbb{R}^{d \times N}$.
Then, one obtains
\begin{align}  \label{eq:con}
\sum_{m=1}^{\infty}\big(	F^N(x_{\widetilde\tau_m\wedge t-}^{i,N},  &\mathcal X_{\widetilde\tau_m\wedge t-}^{N}) -	F^N(x_{\widetilde\tau_{m-1}\wedge t}^{i,N},\mathcal X_{\widetilde\tau_{m-1}\wedge t}^{N})\big)
=\int_{0}^{t}\partial_xF(x_s^{i,N},\mu_s^{x,N})b(x_s^{i,N},\mu_s^{x,N})ds\notag
	\\
	&\quad-\int_{0}^{t}\int_Z\partial_x F(x_s^{i,N},\mu_s^{x,N})\gamma(x_s^{i,N},\mu_s^{x,N},z)\nu(dz)ds \notag
	\\
	&+\frac{1}{N}\sum_{k=1}^N\int_{0}^{t}\partial_{\mu}F(x_s^{i,N},\mu_s^{x,N},x_s^{k,N})b(x_s^{k,N},\mu_s^{x,N})ds \notag
	\\
	&\quad-\frac{1}{N}\sum_{k=1}^N\int_{0}^{t}\int_Z\partial_{\mu}F(x_s^{i,N},\mu_s^{x,N},x_s^{k,N})\gamma(x_s^{k,N},\mu_s^{x,N},z)\nu(dz)ds \notag
	\\
	&+\sum_{\ell=1}^m\int_{0}^{t}\partial_xF(x_s^{i,N},\mu_s^{x,N})\sigma^\ell(x_s^{i,N},\mu_s^{x,N})dw_s^{i\ell} \notag
	\\
	&+\frac{1}{N}\sum_{\ell=1}^m\sum_{k=1}^N\int_{0}^{t}\partial_{\mu}F(x_s^{i,N},\mu_s^{x,N},x_s^{k,N})\sigma^\ell(x_s^{k,N},\mu_s^{x,N})dw_s^{k\ell}\notag
	\\
	&+\frac{1}{2}\int_{0}^{t}\tr[\partial_x^2 F(x_s^{i,N},\mu_s^{x,N})\sigma(x_s^{i,N},\mu_s^{x,N})\sigma^*(x_s^{i,N},\mu_s^{x,N})]ds \notag
	\\
	&+\frac{1}{N}\int_{0}^{t}\tr[\partial_{x}\partial_{\mu} F(x_s^{i,N},\mu_s^{x,N},x_s^{i,N})\sigma(x_s^{i,N},\mu_s^{x,N})\sigma^*(x_s^{i,N},\mu_s^{x,N})]ds \notag
	\\
	&+\frac{1}{2N}\sum_{k=1}^N\int_{0}^{t}\tr[\partial_{y}\partial_{\mu} F(x_s^{i,N},\mu_s^{x,N},x_s^{k,N})\sigma(x_s^{k,N},\mu_s^{x,N})\sigma^*(x_s^{k,N},\mu_s^{x,N})]ds \notag
	\\
	&+\frac{1}{2N^2}\sum_{k=1}^N\int_{0}^{t}\tr[\partial_{\mu}^2 F(x_s^{i,N},\mu_s^{x,N},x_s^{k,N},x_s^{k,N})\sigma(x_s^{k,N},\mu_s^{x,N})\sigma^*(x_s^{k,N},\mu_s^{x,N})]ds 
\end{align}
almost surely	for any $t\in[0,T]$ and $i\in\{1,\ldots,N\}$.    
Now, the second term on the right side of Equation \eqref{eq:cj*} can be written as
\begin{align*} 
	&\sum_{m=1}^{\infty}\big(	F^N(x_{\widetilde\tau_m\wedge t}^{i,N},\mathcal X_{\widetilde\tau_m\wedge t}^{N})-	F^N(x_{\widetilde\tau_m\wedge t-}^{i,N},\mathcal X_{\widetilde\tau_m\wedge t-}^{N})\big) \notag
	\\
	&=\sum_{m=1}^{\infty} \Big\{
	F^N \Big(x_{\widetilde\tau_m\wedge t}^{i,N},x_{\widetilde\tau_m\wedge t}^{1,N},x_{\widetilde\tau_m\wedge t}^{2,N}, \ldots,x_{\widetilde\tau_m\wedge t}^{N,N}\Big)- 	F^N \Big(x_{\widetilde\tau_m\wedge t}^{i,N},x_{\widetilde\tau_m\wedge t-}^{1,N},x_{\widetilde\tau_m\wedge t}^{2,N}, \ldots,x_{\widetilde\tau_m\wedge t}^{N,N}\Big)  \notag
	\\
	&+ 	F^N \Big(x_{\widetilde\tau_m\wedge t}^{i,N},x_{\widetilde\tau_m\wedge t-}^{1,N},x_{\widetilde\tau_m\wedge t}^{2,N}, \ldots,x_{\widetilde\tau_m\wedge t}^{N,N}\Big)- 	F^N \Big(x_{\widetilde\tau_m\wedge t}^{i,N},x_{\widetilde\tau_m\wedge t-}^{1,N},x_{\widetilde\tau_m\wedge t-}^{2,N}, x_{\widetilde\tau_m\wedge t}^{3,N}, \ldots,x_{\widetilde\tau_m\wedge t}^{N,N}\Big)  +\cdots  \notag
	\\
	&+ 	F^N \Big(x_{\widetilde\tau_m\wedge t}^{i,N},x_{\widetilde\tau_m\wedge t-}^{1,N},\ldots, x_{\widetilde\tau_m\wedge t-}^{i-2,N}, x_{\widetilde\tau_m\wedge t}^{i-1,N}, \ldots,x_{\widetilde\tau_m\wedge t}^{N,N}\Big) \notag
	\\
	&
 \qquad\qquad- 	F^N \Big(x_{\widetilde\tau_m\wedge t}^{i,N},x_{\widetilde\tau_m\wedge t-}^{1,N},\ldots,x_{\widetilde\tau_m\wedge t-}^{i-1,N},x_{\widetilde\tau_m\wedge t}^{i,N}, \ldots,x_{\widetilde\tau_m\wedge t}^{N,N}\Big)  \notag
	\\
	& + 	F^N \Big(x_{\widetilde\tau_m\wedge t}^{i,N},x_{\widetilde\tau_m\wedge t-}^{1,N},\ldots, x_{\widetilde\tau_m\wedge t-}^{i-1,N}, x_{\widetilde\tau_m\wedge t}^{i,N}, \ldots,x_{\widetilde\tau_m\wedge t}^{N,N}\Big)  \notag
	\\
	&\qquad\qquad -	F^N \Big(x_{\widetilde\tau_m\wedge t-}^{i,N},x_{\widetilde\tau_m\wedge t-}^{1,N},\ldots,x_{\widetilde\tau_m\wedge t-}^{i,N},x_{\widetilde\tau_m\wedge t}^{i+1,N}, \ldots,x_{\widetilde\tau_m\wedge t}^{N,N}\Big)  \notag
	\\
	& + 	F^N \Big(x_{\widetilde\tau_m\wedge t-}^{i,N},x_{\widetilde\tau_m\wedge t-}^{1,N},\ldots, x_{\widetilde\tau_m\wedge t-}^{i,N}, x_{\widetilde\tau_m\wedge t}^{i+1,N}, \ldots,x_{\widetilde\tau_m\wedge t}^{N,N}\Big)  \notag
	\\
	&
 \qquad\qquad -	F^N \Big(x_{\widetilde\tau_m\wedge t-}^{i,N},x_{\widetilde\tau_m\wedge t-}^{1,N},\ldots,x_{\widetilde\tau_m\wedge t-}^{i+1,N},x_{\widetilde\tau_m\wedge t}^{i+2,N}, \ldots,x_{\widetilde\tau_m\wedge t}^{N,N}\Big)  \notag
	\\
	& + \cdots + 	F^N \Big(x_{\widetilde\tau_m\wedge t-}^{i,N},x_{\widetilde\tau_m\wedge t-}^{1,N},x_{\widetilde\tau_m\wedge t-}^{2,N}, \ldots, x_{\widetilde\tau_m\wedge t-}^{N,N},x_{\widetilde\tau_m\wedge t}^{N,N}\Big)- 	F^N \Big(x_{\widetilde\tau_m\wedge t-}^{i,N},x_{\widetilde\tau_m\wedge t-}^{1,N},x_{\widetilde\tau_m\wedge t-}^{2,N}, \ldots,x_{\widetilde\tau_m\wedge t-}^{N,N}\Big)  \Big\} \notag
\end{align*}
which on using the fact that particles do not jump together yields
\begin{align*}
	&\sum_{m=1}^{\infty}\big(	F^N(x_{\widetilde\tau_m\wedge t}^{i,N},\mathcal X_{\widetilde\tau_m\wedge t}^{N})-	F^N(x_{\widetilde\tau_m\wedge t-}^{i,N},\mathcal X_{\widetilde\tau_m\wedge t-}^{N})\big) \notag
	\\
	&=
	\int_{0}^{t+}\int_Z\Big\{	F^N\big(x_{s}^{i,N},x_{s-}^{1,N}+\gamma(x_s^{1,N},\mu_s^{x,N},z),x_{s}^{2,N}, \ldots,x_{s}^{N,N}\big)  \notag
	\\
	&
 \qquad\qquad\qquad\qquad-	F^N\big(x_{s}^{i,N},x_{s-}^{1,N},x_{s}^{2,N},\ldots,x_{s}^{N,N}\big)\Big\}  n_p^1(ds,dz) \notag
	\\
	&+\int_{0}^{t+}\int_Z\Big\{	F^N\big(x_{s}^{i,N},x_{s-}^{1,N},x_{s-}^{2,N}+\gamma(x_s^{2,N},\mu_s^{x,N},z), \ldots,x_{s}^{N,N}\big)  \notag
	\\
	&
 \qquad\qquad\qquad\qquad-	F^N\big(x_{s}^{i,N},x_{s-}^{1,N},x_{s-}^{2,N},\ldots,x_{s}^{N,N}\big)\Big\}  n_p^2(ds,dz) \notag
	\\
	&+\cdots+\int_{0}^{t+}\int_Z\Big\{	F^N\big(x_{s}^{i,N},x_{s-}^{1,N},\ldots,x_{s-}^{i-2,N},x_{s-}^{i-1,N}+\gamma(x_s^{i-1,N},\mu_s^{x,N},z),x_{s}^{i,N}, \ldots,x_{s}^{N,N}\big) \notag
	\\
	&\qquad\qquad \qquad \qquad-	F^N\big(x_{s}^{i,N},x_{s-}^{1,N},\ldots,x_{s-}^{i-1,N},x_{s}^{i,N},\ldots,x_{s}^{N,N}\big) \Big\} n_p^{i-1}(ds,dz) +\cdots \notag
	\\
	&+\int_{0}^{t+}\int_Z\Big\{	F^N(x_{s-}^{i,N}+\gamma(x_s^{i,N},\mu_s^{x,N},z),x_{s-}^{1,N},\ldots,x_{s-}^{i-1,N},x_{s-}^{i,N}+\gamma(x_s^{i,N},\mu_s^{x,N},z),x_{s}^{i+1,N}, \ldots,x_{s}^{N,N}) \notag
	\\
	&\qquad\qquad \qquad \qquad-	F^N(x_{s-}^{i,N},x_{s-}^{1,N},\ldots,x_{s-}^{i,N},x_{s}^{i+1,N},\ldots,x_{s}^{N,N}) \Big\} n_p^i(ds,dz) +\cdots \notag
	\\
	&+\int_{0}^{t+}\int_Z\Big\{	F^N(x_{s-}^{i,N},x_{s-}^{1,N},\ldots,x_{s-}^{i,N},x_{s-}^{i+1,N}+\gamma(x_s^{i+1,N},\mu_s^{x,N},z),x_{s}^{i+2,N}, \ldots,x_{s}^{N,N}) \notag
	\\
	&\qquad\qquad \qquad \qquad-	F^N(x_{s-}^{i,N},x_{s-}^{1,N},\ldots,x_{s-}^{i+1,N},x_{s}^{i+2,N},\ldots,x_{s}^{N,N}) \Big\} n_p^{i+1}(ds,dz) +\cdots \notag
	\\
	&+ \int_{0}^{t+}\int_Z\Big\{ 	F^N(x_{s-}^{i,N},x_{s-}^{1,N},\ldots,x_{s-}^{N-1,N},x_{s-}^{N,N}+\gamma(x_s^{N,N},\mu_s^{x,N},z))  \notag
	\\
	&
 \qquad\qquad\qquad\qquad -	F^N(x_{s-}^{i,N},x_{s-}^{1,N},\ldots,x_{s-}^{N,N})\Big\}n_p^N(ds,dz) \notag
\end{align*}
almost surely for any $t\in[0,T]$ and $i\in\{1,\ldots,N\}$.
Further, the above can be expressed as follows
\begin{align} \label{eq:jump*}
	\sum_{m=1}^{\infty}&\big(	F^N(x_{\widetilde\tau_m\wedge t}^{i,N},\mathcal X_{\widetilde\tau_m\wedge t}^{N})-	F^N(x_{\widetilde\tau_m\wedge t-}^{i,N},\mathcal X_{\widetilde\tau_m\wedge t-}^{N})\big) \notag
	\\
	=&
	\sum_{k=1}^{N}\hspace{-0.05cm}\int_{0}^{t+} \hspace{-0.1cm}\int_Z\hspace{-0.05cm}\Big\{F^N\big(x_{s-}^{i,N}\hspace{-0.025cm}+\hspace{-0.025cm}1_{\{k=i\}}\gamma(x_s^{k,N},\mu_s^{x,N},z),x_{s-}^{1,N},\ldots,x_{s-}^{k,N}\hspace{-0.025cm}+\hspace{-0.025cm}\gamma(x_s^{k,N},\mu_s^{x,N},z),x_{s}^{k+1,N}, \ldots,x_{s}^{N,N}\big) \notag
	\\
	&\qquad\qquad-F^N\big(x_{s-}^{i,N},x_{s-}^{1,N},\ldots,x_{s-}^{k,N},x_{s}^{k+1,N},\ldots,x_{s}^{N,N}\big)\Big\}  n_p^k(ds,dz)\notag\\
	=&
	\sum_{k=1}^{N}\int_{0}^{t}\int_Z\Big\{F\Big(x_{s-}^{i,N}+1_{\{k=i\}}\gamma(x_s^{k,N},\mu_s^{x,N},z),\frac{1}{N}\sum_{j=1}^{N}\delta_{x_{s-}^{j,N}+1_{\{j=k\}}\gamma(x_s^{j,N},\mu_{s}^{x,N},z)}\Big) \notag
	\\
	&\qquad \qquad \qquad \qquad
	-F\Big(x_{s-}^{i,N},\frac{1}{N}\sum_{j=1}^{N}\delta_{x_{s-}^{j,N}}\Big)\Big\}n_p^k(ds,dz)
\end{align}
almost surely for any $t\in[0,T]$ and $i\in\{1,\ldots,N\}$.
By substituting Equations \eqref{eq:con} and \eqref{eq:jump*} in Equation \eqref{eq:cj*}, we get 
\begin{align} \label{eq:ito}
	&F(x_t^{i,N},\,\mu_t^{x,N})=F(x_0^{i,N},\mu_0^{x,N}) \notag
	\\
	&+\int_{0}^{t}\partial_xF(x_s^{i,N},\mu_s^{x,N})b(x_s^{i,N},\mu_s^{x,N})ds		-\int_{0}^{t}\int_Z\partial_x F(x_s^{i,N},\mu_s^{x,N})\gamma(x_s^{i,N},\mu_s^{x,N},z)\nu(dz)ds \notag
	\\
	& +\frac{1}{N}\sum_{k=1}^N\int_{0}^{t}\partial_{\mu}F(x_s^{i,N},\mu_s^{x,N},x_s^{k,N})b(x_s^{k,N},\mu_s^{x,N})ds\notag
	\\
	&\qquad -\frac{1}{N}\sum_{k=1}^{N}\int_{0}^{t}\int_Z
	\partial_{\mu}F(x_s^{i,N},\mu_s^{x,N},x_s^{k,N})\gamma(x_s^{k,N},\mu_s^{x,N},z) \nu(dz)ds \notag
	\\
	&+\sum_{\ell=1}^m\int_{0}^{t}\partial_xF(x_s^{i,N},\mu_s^{x,N})\sigma^\ell(x_s^{i,N},\mu_s^{x,N})dw_s^{i\ell} \notag
	+\frac{1}{N}\sum_{\ell=1}^m\sum_{k=1}^N\int_{0}^{t}\partial_{\mu}F(x_s^{i,N},\mu_s^{x,N},x_s^{k,N})\sigma^\ell(x_s^{k,N},\mu_s^{x,N})dw_s^{k\ell} \notag
	\\
	&+\frac{1}{2}\int_{0}^{t}\tr[\partial_x^2 F(x_s^{i,N},\mu_s^{x,N})\sigma(x_s^{i,N},\mu_s^{x,N})\sigma^*(x_s^{i,N},\mu_s^{x,N})]ds \notag
	\\
	&+\frac{1}{N}\int_{0}^{t}\tr[\partial_{x}\partial_{\mu} F(x_s^{i,N},\mu_s^{x,N},x_s^{i,N})\sigma(x_s^{i,N},\mu_s^{x,N})\sigma^*(x_s^{i,N},\mu_s^{x,N})]ds \notag
	\\
	&+\frac{1}{2N}\sum_{k=1}^N\int_{0}^{t}\tr[\partial_{y}\partial_{\mu} F(x_s^{i,N},\mu_s^{x,N},x_s^{k,N})\sigma(x_s^{k,N},\mu_s^{x,N})\sigma^*(x_s^{k,N},\mu_s^{x,N})]ds \notag
	\\
	&+\frac{1}{2N^2}\sum_{k=1}^N\int_{0}^{t}\tr[\partial_{\mu}^2 F(x_s^{i,N},\mu_s^{x,N},x_s^{k,N},x_s^{k,N})\sigma(x_s^{k,N},\mu_s^{x,N})\sigma^*(x_s^{k,N},\mu_s^{x,N})]ds \notag
	\\
	&+ \sum_{k=1}^{N}\int_{0}^{t}\int_Z\Big\{F\big(x_{s-}^{i,N}+1_{\{k=i\}}\gamma(x_s^{k,N},\mu_{s}^{x,N},z), \frac{1}{N}\sum_{j=1}^{N}\delta_{x_{s-}^{j,N}+1_{\{j=k\}}\gamma(x_s^{j,N},\mu_{s}^{x,N},z)}) \notag
	\\
	&
 \qquad\qquad\qquad\qquad
	-F(x_{s-}^{i,N},\mu_{s-}^{x,N}\big)\Big\}n_p^k(ds,dz) 
\end{align}
almost surely for any $t\in[0,T]$ and $i\in\{1,\ldots,N\}$. 
\end{proof}

\begin{proof}[Proof of Lemma \ref{lem:mvt}] 

Let 
 $
 f(x)=|x|^p
 $  
 for any $x\in\mathbb R^d$.
 Then, for any $u,v\in\{1,\ldots,d\}$ and $x\in\mathbb R^d$, 
\begin{align*}
	&\partial_{x^u}f(x)
	=p|x|^{p-2}x^u \mbox{ and } 
\partial_{x^v}\partial_{x^u}f(x)
	=p(p-2)|x|^{p-4} x^u x^v +p|x|^{p-2} 1_{\{u=v\}}. 
\end{align*}
By Taylor's theorem for a $f\in \mathcal C^2(\mathbb R^d)$
\begin{align*}
f(x)-f(y)-\sum_{u=1}^{d} (x^u-y^u)\partial_{x^u}f(y) =\sum_{u,v=1}^{d} (x^u-y^u) (x^v-y^v) \int_0^1(1-\theta) \partial_{x^v} \partial_{x^u} f\big(y+\theta(x-y)\big) d\theta
\end{align*}
which on taking  $f(x)=|x|^p$ yields 
\begin{align*}
	& |x|^p-|y|^p-p\sum_{u=1}^{d} (x^u-y^u) |y|^{p-2}y^u 
	\\
	&= p(p-2)\sum_{u,v=1}^{d} (x^u-y^u) (x^v-y^v)  \int_0^1(1-\theta) \big|y+\theta(x-y)\big|^{p-4} \big(y^u+\theta(x^u-y^u)\big)\big(y^v+\theta(x^v-y^v)\big) d\theta
		\\
	&\quad+  p\sum_{u,v=1}^{d}  (x^u-y^u) (x^v-y^v)   \int_0^1(1-\theta) \big|y+\theta(x-y)\big|^{p-2} 1_{\{u=v\}} d\theta 
		\\
	&= p(p-2)\int_0^1(1-\theta) \big|y+\theta(x-y)\big|^{p-4} \Big(\sum_{u=1}^d (x^u-y^u)\big(y^u+\theta(x^u-y^u)\Big)^2  
        d\theta
	\\
	&\quad+ p\sum_{u=1}^{d} (x^u-y^u)^2  \int_0^1(1-\theta) \big|y+\theta(x-y)\big|^{p-2} d\theta 
		\\
	&\leq p(p-2)  \int_0^1(1-\theta) \big|y+\theta(x-y)\big|^{p-4} |x-y|^2\big|y+\theta(x-y)\big|^2 d\theta
	\\
	&\quad+ p |x-y|^2  \int_0^1(1-\theta) \big|y+\theta(x-y)\big|^{p-2} d\theta 
	=p(p-1) |x-y|^2  \int_0^1(1-\theta) \big|y+\theta(x-y)\big|^{p-2} d\theta
\end{align*}
for any $x,y\in\mathbb R^d$.
\end{proof}
\end{document}

%% file: Newcommand_23-oct-2023.tex
%FOR DRIFT COEFFICIENT, b     
%%%%%%%%%%%%%%%%%%%%%%%%%%%%%%%%%%%%%%%%%%%%%%%%%%%%%

\newcommand{\bb}{ b^{u}\big(x_{\kappa_n(s)}^{i,N,n},\mu^{x,N,n}_{\kappa_n(s)}\big)}

\newcommand{\bbtame}{(\widehat b^{u})_{tam}^{\displaystyle{n}}\big(x_{\kappa_n(s)}^{i,N,n},\mu^{x,N,n}_{\kappa_n(s)}\big)}

\newcommand{\bbtamexu}{(\widehat b^{u})_{tam}^{\displaystyle{n}}(x,\mu)}

\newcommand{\bbtamex}{(\widehat b)_{tam}^{\displaystyle{n}}(x,\mu)}

\newcommand{\B}{ b\big(x_{\kappa_n(s)}^{i,N,n},\mu^{x,N,n}_{\kappa_n(s)}\big)}

\newcommand{\Btame}{(\widehat b)_{tam}^{\displaystyle{n}}\big(x_{\kappa_n(s)}^{i,N,n},\mu^{x,N,n}_{\kappa_n(s)}\big)}

\newcommand{\Btamer}{(\widehat b)_{tam}^{\displaystyle{n}}\big(x_{\kappa_n(r)}^{i,N,n},\mu^{x,N,n}_{\kappa_n(r)}\big)}

\newcommand{\Btamek}{(\widehat b)_{tam}^{\displaystyle{n}}\big(x_{\kappa_n(s)}^{k,N,n},\mu^{x,N,n}_{\kappa_n(s)}\big)}

\newcommand{\Btamekr}{(\widehat b)_{tam}^{\displaystyle{n}}\big(x_{\kappa_n(r)}^{k,N,n},\mu^{x,N,n}_{\kappa_n(r)}\big)}
%%%%%%%%%%%%%%%%%%%%%%%%%%%%%%%%%%%%%%%%%%%%%%%%%%%%%%%

%FOR DIFFUSION COEFFICIENT, \sigma     
%%%%%%%%%%%%%%%%%%%%%%%%%%%%%%%%%%%%%%%%%%%%%%%%%%%%%%%
\newcommand{\sig}{\sigma^{u\ell}\big(x_{\kappa_n(s)}^{i,N,n},\mu^{x,N,n}_{\kappa_n(s)}\big)}

\newcommand{\sigmatameu}{(\widehat{\sigma}^{u\ell})_{tam}^{\displaystyle n}}

\newcommand{\gammatameu}{(\widehat{\gamma}^{u})_{tam}^{\displaystyle n}}

\newcommand{\sigmaonetameu}{(\widehat{\sigma}_1^{u\ell})_{tam}^{\displaystyle n}}

\newcommand{\sigmatwotameu}{(\widehat{\sigma}_2^{u\ell})_{tam}^{\displaystyle n}}

\newcommand{\sigmaonetame}{(\widehat{\sigma}_1)_{tam}^{\displaystyle n}}

\newcommand{\sigmatwotame}{(\widehat{\sigma}_2)_{tam}^{\displaystyle n}}

\newcommand{\sigtame}{(\widehat\sigma^{u\ell})_{tam}^{\displaystyle{n}}\big(x_{\kappa_n(s)}^{i,N,n},\mu^{x,N,n}_{\kappa_n(s)}\big)}

\newcommand{\sigtamer}{(\widehat\sigma^{u\ell})_{tam}^{\displaystyle{n}}\big(x_{\kappa_n(r)}^{i,N,n},\mu^{x,N,n}_{\kappa_n(r)}\big)}

\newcommand{\sigtamexu}{(\widehat\sigma^{u\ell})_{tam}^{\displaystyle{n}}(x,\mu)}

\newcommand{\sigtamex}{(\widehat\sigma)_{tam}^{\displaystyle{n}}(x,\mu)}

\newcommand{\sigxsig}{\mathfrak{D}_x^{\sigma^{\ell_1}}\sigma^{u\ell}\big(x_{\kappa_n(s)}^{i,N,n}, \mu_{\kappa_n(s)}^{x,N,n}\big)}

\newcommand{\sigxsigr}{\mathfrak{D}_x^{\sigma^{\ell_1}}\sigma^{u\ell}\big(x_{\kappa_n(r)}^{i,N,n}, \mu_{\kappa_n(r)}^{x,N,n}\big)}

\newcommand{\sigxsigtame}{(\widehat{\mathfrak{D}}_x^{\sigma^{\ell_1}}\sigma^{u\ell})_{tam}^{\displaystyle{n}}\big(x_{\kappa_n(s)}^{i,N,n}, \mu_{\kappa_n(s)}^{x,N,n}\big)}

\newcommand{\sigxsigtamer}{(\widehat{\mathfrak{D}}_x^{\sigma^{\ell_1}}\sigma^{u\ell})_{tam}^{\displaystyle{n}}\big(x_{\kappa_n(r)}^{i,N,n}, \mu_{\kappa_n(r)}^{x,N,n}\big)}

\newcommand{\sigxsigtamex}{(\widehat{\mathfrak{D}}_x^{\sigma^{\ell_1}}\sigma^{u\ell})_{tam}^{\displaystyle{n}}(x, \mu)}

\newcommand{\sigmusig}{\mathfrak{D}_\mu^{\sigma^{\ell_1}}\sigma^{u\ell}\big(x_{\kappa_n(s)}^{i,N,n},\mu_{\kappa_n(s)}^{x,N,n},x_{\kappa_n(s)}^{k,N,n}\big)}

\newcommand{\sigmusigr}{\mathfrak{D}_\mu^{\sigma^{\ell_1}}\sigma^{u\ell}\big(x_{\kappa_n(r)}^{i,N,n},\mu_{\kappa_n(r)}^{x,N,n},x_{\kappa_n(r)}^{k,N,n}\big)}

\newcommand{\sigmusigtame}{(\widehat{\mathfrak{D}}_\mu^{\sigma^{\ell_1}}\sigma^{u\ell})_{tam}^{\displaystyle{n}}\big(x_{\kappa_n(s)}^{i,N,n},\mu_{\kappa_n(s)}^{x,N,n},x_{\kappa_n(s)}^{k,N,n}\big)}

\newcommand{\sigmusigtamer}{(\widehat{\mathfrak{D}}_\mu^{\sigma^{\ell_1}}\sigma^{u\ell})_{tam}^{\displaystyle{n}}\big(x_{\kappa_n(r)}^{i,N,n},\mu_{\kappa_n(r)}^{x,N,n},x_{\kappa_n(r)}^{k,N,n}\big)}

\newcommand{\sigmusigtamex}{(\widehat{\mathfrak{D}}_\mu^{\sigma^{\ell_1}}\sigma^{u\ell})_{tam}^{\displaystyle{n}}(x,\mu,y)}

\newcommand{\sigxgam}{\mathfrak{D}_x^{\gamma(z)}\sigma^{u\ell}\big(x_{\kappa_n(s)}^{i,N,n}, \mu_{\kappa_n(s)}^{x,N,n}\big)}

\newcommand{\sigxgamtame}{(\widehat{\mathfrak{D}}_x^{\gamma(z)}\sigma^{u\ell})_{tam}^{\displaystyle{n}}\big(x_{\kappa_n(s)}^{i,N,n}, \mu_{\kappa_n(s)}^{x,N,n}\big)}

\newcommand{\sigxgamtamex}{(\widehat{\mathfrak{D}}_x^{\gamma(z)}\sigma^{u\ell})_{tam}^{\displaystyle{n}}(x, \mu)}

\newcommand{\sigmugam}{\mathfrak{D}_\mu^{\gamma(z)}\sigma^{u\ell}\big(x_{\kappa_n(s)}^{i,N,n},\mu_{\kappa_n(s)}^{x,N,n},x_{\kappa_n(s)}^{k,N,n}\big)}

\newcommand{\sigmugamtame}{(\widehat{\mathfrak{D}}_\mu^{\gamma(z)}\sigma^{u\ell})_{tam}^{\displaystyle{n}}\big(x_{\kappa_n(s)}^{i,N,n},\mu_{\kappa_n(s)}^{x,N,n},x_{\kappa_n(s)}^{k,N,n}\big)}

\newcommand{\sigmugamtamex}{(\widehat{\mathfrak{D}}_\mu^{\gamma(z)}\sigma^{u\ell})_{tam}^{\displaystyle{n}}(x,\mu,y)}

\newcommand{\sigxp}{\mathfrak J_x^{\gamma(z)}\sigma^{u\ell} \big(x_{\kappa_n(s)}^{i,N,n}, \mu_{\kappa_n(s)}^{x,N,n}\big)}

\newcommand{\sigxptame}{(\widehat{\mathfrak J}_x^{\gamma(z)}\sigma^{u\ell})_{tam}^{\displaystyle{n}} (x_{\kappa_n(s)}^{i,N,n}, \mu_{\kappa_n(s)}^{x,N,n})}

\newcommand{\sigmup}{\mathfrak J_\mu^{\gamma(z)} \sigma^{u\ell}(x_{\kappa_n(s)}^{i,N,n}, \mu_{\kappa_n(s)}^{x,N,n},x_{\kappa_n(s)}^{k,N,n},\tilmukappa)}

\newcommand{\sigmuptame}{(\widehat{\mathfrak J}_\mu^{\gamma(z)} \sigma^{u\ell})_{tam}^{\displaystyle{n}}(x_{\kappa_n(s)}^{i,N,n}, \mu_{\kappa_n(s)}^{x,N,n},x_{\kappa_n(s)}^{k,N,n},\tilmukappa)}

\newcommand{\hatlamda}{(\widehat{\Lambda}^{u\ell})_{tam}^{\displaystyle{n}}(s,x_{\kappa_n(s)}^{i,N,n},\mu^{x,N,n}_{\kappa_n(s)})}

\newcommand{\hatlamdax}{(\widehat{\Lambda}^{u\ell})_{tam}^{\displaystyle{n}}(s,x,\mu)}

\newcommand{\sigone}{(\widehat\sigma_1^{u\ell})_{tam}^{\displaystyle{n}}\big(\kappa_n(s),s,x_{\kappa_n(s)}^{i,N,n},\mu^{x,N,n}_{\kappa_n(s)}\big)}

\newcommand{\sigonelone}{(\widehat\sigma_1^{\ell_1})_{tam}^{\displaystyle{n}}\big(\kappa_n(s),s,x_{\kappa_n(s)}^{i,N,n},\mu^{x,N,n}_{\kappa_n(s)}\big)}

\newcommand{\sigtwo}{(\widehat\sigma_2^{u\ell})_{tam}^{\displaystyle{n}}\big(\kappa_n(s),s,x_{\kappa_n(s)}^{i,N,n},\mu^{x,N,n}_{\kappa_n(s)}\big)}

\newcommand{\sigtwolone}{(\widehat\sigma_2^{\ell_1})_{tam}^{\displaystyle{n}}\big(\kappa_n(s),s,x_{\kappa_n(s)}^{i,N,n},\mu^{x,N,n}_{\kappa_n(s)}\big)}

\newcommand{\sigthree}{(\widehat\sigma_2^{u\ell})_{tam}^{\displaystyle{n}}\big(\kappa_n(s),s,x_{\kappa_n(s)}^{i,N,n},\mu^{x,N,n}_{\kappa_n(s)}\big)}

\newcommand{\sigonex}{(\widehat\sigma_1^{u\ell})_{tam}^{\displaystyle{n}}(s,x,\mu)}

\newcommand{\sigtwox}{(\widehat\sigma_2^{u\ell})_{tam}^{\displaystyle{n}}(s,x,\mu)}

\newcommand{\sigthreex}{(\widehat\sigma_2^{u\ell})_{tam}^{\displaystyle{n}}(s,x,\mu)}

\newcommand{\sigq}{(\widehat\sigma_{\mathfrak{q}}^{u\ell})_{tam}^{\displaystyle{n}}\big(\kappa_n(s),s,x_{\kappa_n(s)}^{i,N,n},\mu^{x,N,n}_{\kappa_n(s)}\big)}

\newcommand{\sigqr}{(\widehat\sigma_{\mathfrak{q}}^{u\ell})_{tam}^{\displaystyle{n}}\big(\kappa_n(r),r,x_{\kappa_n(r)}^{i,N,n},\mu^{x,N,n}_{\kappa_n(r)}\big)}

\newcommand{\sigqlone}{(\widehat\sigma_{\mathfrak{q}}^{\ell_1})_{tam}^{\displaystyle{n}}\big(\kappa_n(s),s,x_{\kappa_n(s)}^{i,N,n},\mu^{x,N,n}_{\kappa_n(s)}\big)}

\newcommand{\sigqlonek}{(\widehat\sigma_{\mathfrak{q}}^{\ell_1})_{tam}^{\displaystyle{n}}\big(\kappa_n(s),s,x_{\kappa_n(s)}^{k,N,n},\mu^{x,N,n}_{\kappa_n(s)}\big)}

\newcommand{\sigqloner}{(\widehat\sigma_{\mathfrak{q}}^{\ell_1})_{tam}^{\displaystyle{n}}\big(\kappa_n(r),r,x_{\kappa_n(r)}^{i,N,n},\mu^{x,N,n}_{\kappa_n(r)}\big)}

\newcommand{\sigqlonekr}{(\widehat\sigma_{\mathfrak{q}}^{\ell_1})_{tam}^{\displaystyle{n}}\big(\kappa_n(r),r,x_{\kappa_n(r)}^{k,N,n},\mu^{x,N,n}_{\kappa_n(r)}\big)}

\newcommand{\Sig}{\sigma(x_{\kappa_n(s)}^{i,N,n},\mu^{x,N,n}_{\kappa_n(s)})}

\newcommand{\Sigtame}{(\widehat\sigma)_{tam}^{\displaystyle{n}}\big(x_{\kappa_n(s)}^{i,N,n},\mu^{x,N,n}_{\kappa_n(s)}\big)}

\newcommand{\Sigtamek}{(\widehat\sigma)_{tam}^{\displaystyle{n}}\big(x_{\kappa_n(s)}^{k,N,n},\mu^{x,N,n}_{\kappa_n(s)}\big)}

\newcommand{\Sigtamelone}{(\widehat\sigma^{\ell_1})_{tam}^{\displaystyle{n}}\big(x_{\kappa_n(s)}^{i,N,n},\mu^{x,N,n}_{\kappa_n(s)}\big)}

\newcommand{\Sigtameloner}{(\widehat\sigma^{\ell_1})_{tam}^{\displaystyle{n}}\big(x_{\kappa_n(r)}^{i,N,n},\mu^{x,N,n}_{\kappa_n(r)}\big)}

\newcommand{\Siglone}{\sigma^{\ell_1}\big(x_{\kappa_n(s)}^{i,N,n},\mu^{x,N,n}_{\kappa_n(s)}\big)}

\newcommand{\Sigloner}{\sigma^{\ell_1}\big(x_{\kappa_n(r)}^{i,N,n},\mu^{x,N,n}_{\kappa_n(r)}\big)}

\newcommand{\Siglonek}{\sigma^{\ell_1}\big(x_{\kappa_n(s)}^{k,N,n},\mu^{x,N,n}_{\kappa_n(s)}\big)}

\newcommand{\Siglonekr}{\sigma^{\ell_1}\big(x_{\kappa_n(r)}^{k,N,n},\mu^{x,N,n}_{\kappa_n(r)}\big)}

\newcommand{\Sigtamelonek}{(\widehat\sigma^{\ell_1})_{tam}^{\displaystyle{n}}\big(x_{\kappa_n(s)}^{k,N,n},\mu^{x,N,n}_{\kappa_n(s)}\big)}

\newcommand{\Sigtamelonekr}{(\widehat\sigma^{\ell_1})_{tam}^{\displaystyle{n}}\big(x_{\kappa_n(r)}^{k,N,n},\mu^{x,N,n}_{\kappa_n(r)}\big)}

\newcommand{\Sigtametrans}{(\widehat\sigma)_{tam}^{{\displaystyle{n}}*}\big(x_{\kappa_n(s)}^{i,N,n},\mu^{x,N,n}_{\kappa_n(s)}\big)}

\newcommand{\Sigtametransk}{(\widehat\sigma)_{tam}^{{\displaystyle{n}}*}\big(x_{\kappa_n(s)}^{k,N,n},\mu^{x,N,n}_{\kappa_n(s)}\big)}

\newcommand{\Tildesig}{(\widehat{\Lambda})_{tam}^{\displaystyle{n}}(s,x_{\kappa_n(s)}^{i,N,n},\mu^{x,N,n}_{\kappa_n(s)})}

\newcommand{\Sigone}{(\widehat\sigma_1)_{tam}^{\displaystyle{n}}\big(\kappa_n(s),s,x_{\kappa_n(s)}^{i,N,n},\mu^{x,N,n}_{\kappa_n(s)}\big)}

\newcommand{\Sigtwo}{(\widehat\sigma_2)_{tam}^{\displaystyle{n}}\big(\kappa_n(s),s,x_{\kappa_n(s)}^{i,N,n},\mu^{x,N,n}_{\kappa_n(s)}\big)}

\newcommand{\Sigthree}{(\widehat\sigma_2)_{tam}^{\displaystyle{n}}\big(\kappa_n(s),s,x_{\kappa_n(s)}^{i,N,n},\mu^{x,N,n}_{\kappa_n(s)}\big)}

\newcommand{\Sigq}{(\widehat\sigma_{\mathfrak{q}})_{tam}^{\displaystyle{n}}\big(\kappa_n(s),s,x_{\kappa_n(s)}^{i,N,n},\mu^{x,N,n}_{\kappa_n(s)}\big)}

\newcommand{\Sigqk}{(\widehat\sigma_{\mathfrak{q}})_{tam}^{\displaystyle{n}}\big(\kappa_n(s),s,x_{\kappa_n(s)}^{k,N,n},\mu^{x,N,n}_{\kappa_n(s)}\big)}
%%%%%%%%%%%%%%%%%%%%%%%%%%%%%%%%%%%%%%%%%%%%%%%%%%%%%%%

%FOR jump COEFFICIENT, \gamma     
%%%%%%%%%%%%%%%%%%%%%%%%%%%%%%%%%%%%%%%%%%%%%%%%%%%%%%%
\newcommand{\gam}{\gamma^{u}\big(x_{\kappa_n(s)}^{i,N,n},\mu^{x,N,n}_{\kappa_n(s)}, z\big)}

\newcommand{\gamzbar}{\gamma^{u}\big(x_{\kappa_n(s)}^{i,N,n},\mu^{x,N,n}_{\kappa_n(s)}, \bar z\big)}

\newcommand{\gammaonetame}{(\widehat{\gamma}_1)_{tam}^{\displaystyle n}}

\newcommand{\gammatwotame}{(\widehat{\gamma}_2)_{tam}^{\displaystyle n}}

\newcommand{\gammaonetameu}{(\widehat{\gamma}_1^{u})_{tam}^{\displaystyle n}}

\newcommand{\gammatwotameu}{(\widehat{\gamma}_2^{u})_{tam}^{\displaystyle n}}

\newcommand{\gamtame}{(\widehat\gamma^{u})_{tam}^{\displaystyle{n}}\big(x_{\kappa_n(s)}^{i,N,n},\mu^{x,N,n}_{\kappa_n(s)},\bar z\big)}

\newcommand{\gamtamer}{(\widehat\gamma^{u})_{tam}^{\displaystyle{n}}\big(x_{\kappa_n(r)}^{i,N,n},\mu^{x,N,n}_{\kappa_n(r)},\bar z\big)}

\newcommand{\gamtamez}{(\widehat\gamma^{u})_{tam}^{\displaystyle{n}}\big(x_{\kappa_n(s)}^{i,N,n},\mu^{x,N,n}_{\kappa_n(s)},z\big)}

\newcommand{\gamtamex}{(\widehat\gamma^{u})_{tam}^{\displaystyle{n}}(x,\mu,\bar z)}

\newcommand{\gamtamexz}{(\widehat\gamma)_{tam}^{\displaystyle{n}}(x,\mu,z)}

\newcommand{\gamxsig}{\mathfrak{D}_x^{\sigma^{\ell_1}}\gamma^u\big(x_{\kappa_n(s)}^{i,N,n}, \mu_{\kappa_n(s)}^{x,N,n},\bar z\big)}

\newcommand{\gamzxsig}{\mathfrak{D}_x^{\sigma^{\ell_1}}\gamma^u\big(x_{\kappa_n(s)}^{i,N,n}, \mu_{\kappa_n(s)}^{x,N,n},z\big)}

\newcommand{\gamzxsigr}{\mathfrak{D}_x^{\sigma^{\ell_1}}\gamma^u\big(x_{\kappa_n(r)}^{i,N,n}, \mu_{\kappa_n(r)}^{x,N,n},z\big)}

\newcommand{\gamxsigtame}{(\widehat{\mathfrak{D}}_x^{\sigma^{\ell_1}}\gamma^u)_{tam}^{\displaystyle{n}}\big(x_{\kappa_n(s)}^{i,N,n}, \mu_{\kappa_n(s)}^{x,N,n},\bar z\big)}

\newcommand{\gamzxsigtame}{(\widehat{\mathfrak{D}}_x^{\sigma^{\ell_1}}\gamma^u)_{tam}^{\displaystyle{n}}\big(x_{\kappa_n(s)}^{i,N,n}, \mu_{\kappa_n(s)}^{x,N,n},z\big)}

\newcommand{\gamzxsigtamer}{(\widehat{\mathfrak{D}}_x^{\sigma^{\ell_1}}\gamma^u)_{tam}^{\displaystyle{n}}\big(x_{\kappa_n(r)}^{i,N,n}, \mu_{\kappa_n(r)}^{x,N,n},z\big)}

\newcommand{\gamxsigtamex}{(\widehat{\mathfrak{D}}_x^{\sigma^{\ell_1}}\gamma^u)_{tam}^{\displaystyle{n}}(x, \mu,\bar z)}

\newcommand{\gamxsigtamexz}{(\widehat{\mathfrak{D}}_x^{\sigma^{\ell_1}}\gamma^u)_{tam}^{\displaystyle{n}}(x, \mu,z)}

\newcommand{\gammusig}{\mathfrak{D}_\mu^{\sigma^{\ell_1}}\gamma^u\big(x_{\kappa_n(s)}^{i,N,n},\mu_{\kappa_n(s)}^{x,N,n},x_{\kappa_n(s)}^{k,N,n},\bar z\big)}

\newcommand{\gamzmusig}{\mathfrak{D}_\mu^{\sigma^{\ell_1}}\gamma^u\big(x_{\kappa_n(s)}^{i,N,n},\mu_{\kappa_n(s)}^{x,N,n},x_{\kappa_n(s)}^{k,N,n},z\big)}

\newcommand{\gamzmusigr}{\mathfrak{D}_\mu^{\sigma^{\ell_1}}\gamma^u\big(x_{\kappa_n(r)}^{i,N,n},\mu_{\kappa_n(r)}^{x,N,n},x_{\kappa_n(r)}^{k,N,n},z\big)}

\newcommand{\gammusigtame}{(\widehat{\mathfrak{D}}_\mu^{\sigma^{\ell_1}}\gamma^u)_{tam}^{\displaystyle{n}}\big(x_{\kappa_n(s)}^{i,N,n},\mu_{\kappa_n(s)}^{x,N,n},x_{\kappa_n(s)}^{k,N,n},\bar z\big)}

\newcommand{\gamzmusigtame}{(\widehat{\mathfrak{D}}_\mu^{\sigma^{\ell_1}}\gamma^u)_{tam}^{\displaystyle{n}}\big(x_{\kappa_n(s)}^{i,N,n},\mu_{\kappa_n(s)}^{x,N,n},x_{\kappa_n(s)}^{k,N,n}, z\big)}

\newcommand{\gamzmusigtamer}{(\widehat{\mathfrak{D}}_\mu^{\sigma^{\ell_1}}\gamma^u)_{tam}^{\displaystyle{n}}\big(x_{\kappa_n(r)}^{i,N,n},\mu_{\kappa_n(r)}^{x,N,n},x_{\kappa_n(r)}^{k,N,n}, z\big)}

\newcommand{\gammusigtamex}{(\widehat{\mathfrak{D}}_\mu^{\sigma^{\ell_1}}\gamma^u)_{tam}^{\displaystyle{n}}(x,\mu,y,\bar z)}

\newcommand{\gammusigtamexz}{(\widehat{\mathfrak{D}}_\mu^{\sigma^{\ell_1}}\gamma^u)_{tam}^{\displaystyle{n}}(x,\mu,y, z)}

\newcommand{\gamxgam}{\mathfrak{D}_x^{\gamma(z)}\gamma^u\big(x_{\kappa_n(s)}^{i,N,n}, \mu_{\kappa_n(s)}^{x,N,n},\bar z\big)}

\newcommand{\gamxgamtame}{(\widehat{\mathfrak{D}}_x^{\gamma(z)}\gamma^u)_{tam}^{\displaystyle{n}}\big(x_{\kappa_n(s)}^{i,N,n}, \mu_{\kappa_n(s)}^{x,N,n},\bar z\big)}

\newcommand{\gamxgamtamex}{(\widehat{\mathfrak{D}}_x^{\gamma(z)}\gamma^u)_{tam}^{\displaystyle{n}}\big(x, \mu,\bar z\big)}

\newcommand{\gammugam}{\mathfrak{D}_\mu^{\gamma(z)}\gamma^u\big(x_{\kappa_n(s)}^{i,N,n},\mu_{\kappa_n(s)}^{x,N,n},x_{\kappa_n(s)}^{k,N,n},\bar z\big)}

\newcommand{\gammugamtame}{(\widehat{\mathfrak{D}}_\mu^{\gamma(z)}\gamma^u)_{tam}^{\displaystyle{n}}\big(x_{\kappa_n(s)}^{i,N,n},\mu_{\kappa_n(s)}^{x,N,n},x_{\kappa_n(s)}^{k,N,n},\bar z\big)}

\newcommand{\gammugamtamex}{(\widehat{\mathfrak{D}}_\mu^{\gamma(z)}\gamma^u)_{tam}^{\displaystyle{n}}(x,\mu,y,\bar z)}

\newcommand{\gamxp}{ \mathfrak J_x^{\gamma(z)}\gamma^{u} (x_{\kappa_n(s)}^{i,N,n},\mu_{\kappa_n(s)}^{x,N,n}, \bar z)}

\newcommand{\gamxptame}{ (\widehat{\mathfrak J}_x^{\gamma(z)}\gamma^{u})_{tam}^{\displaystyle{n}} (x_{\kappa_n(s)}^{i,N,n},\mu_{\kappa_n(s)}^{x,N,n}, \bar z)}

\newcommand{\gammup}{\mathfrak J_\mu^{\gamma(z)} \gamma^{u} (x_{\kappa_n(s)}^{i,N,n},\mu_{\kappa_n(s)}^{x,N,n},x_{\kappa_n(s)}^{k,N,n},\tilmukappa, \bar z)}

\newcommand{\gammuptame}{(\widehat{\mathfrak J}_\mu^{\gamma(z)} \gamma^{u})_{tam}^{\displaystyle{n}} (x_{\kappa_n(s)}^{i,N,n},\mu_{\kappa_n(s)}^{x,N,n},x_{\kappa_n(s)}^{k,N,n},\tilmukappa, \bar z)}

\newcommand{\hatgamma}{(\widehat\Gamma^{u})_{tam}^{\displaystyle{n}}(s,x_{\kappa_n(s)}^{i,N,n},\mu^{x,N,n}_{\kappa_n(s)},z)}

\newcommand{\hatgammax}{(\widehat\Gamma^{u})_{tam}^{\displaystyle{n}}(s,x,\mu,z)}

\newcommand{\gamone}{(\widehat\gamma_1^u)_{tam}^{\displaystyle{n}}\big(\kappa_n(s),s,x_{\kappa_n(s)}^{i,N,n},\mu^{x,N,n}_{\kappa_n(s)},\bar z\big)}

\newcommand{\gamtwo}{(\widehat\gamma_2^{u})_{tam}^{\displaystyle{n}}\big(\kappa_n(s),s,x_{\kappa_n(s)}^{i,N,n},\mu^{x,N,n}_{\kappa_n(s)},\bar z\big)}

\newcommand{\gamthree}{(\widehat\gamma_2^{u})_{tam}^{\displaystyle{n}}(\kappa_n(s),s,x_{\kappa_n(s)}^{i,N,n},\mu^{x,N,n}_{\kappa_n(s)},\bar z\big)}

\newcommand{\gamonex}{(\widehat\gamma_1^u)_{tam}^{\displaystyle{n}}(s,x,\mu,\bar z)}

\newcommand{\gamtwox}{(\widehat\gamma_2^{u})_{tam}^{\displaystyle{n}}(s,x,\mu,\bar z)}

\newcommand{\gamthreex}{(\widehat\gamma_2^{u})_{tam}^{\displaystyle{n}}(s,x,\mu,\bar z)}

\newcommand{\gamq}{(\widehat\gamma_{\mathfrak{q}}^{u})_{tam}^{\displaystyle{n}}\big(\kappa_n(s),s,x_{\kappa_n(s)}^{i,N,n},\mu^{x,N,n}_{\kappa_n(s)},\bar z\big)}

\newcommand{\gamqr}{(\widehat\gamma_{\mathfrak{q}}^{u})_{tam}^{\displaystyle{n}}\big(\kappa_n(r),r,x_{\kappa_n(r)}^{i,N,n},\mu^{x,N,n}_{\kappa_n(r)},\bar z\big)}

\newcommand{\Gam}{\gamma(x_{\kappa_n(s)}^{i,N,n},\mu^{x,N,n}_{\kappa_n(s)},\bar z)}

\newcommand{\Gamz}{\gamma(x_{\kappa_n(s)}^{i,N,n},\mu^{x,N,n}_{\kappa_n(s)},z)}

\newcommand{\Gamk}{\gamma\big(x_{\kappa_n(s)}^{k,N,n},\mu^{x,N,n}_{\kappa_n(s)}, z\big)}

\newcommand{\Gamtame}{(\widehat\gamma)_{tam}^{\displaystyle{n}}\big(x_{\kappa_n(s)}^{i,N,n},\mu^{x,N,n}_{\kappa_n(s)}, \bar z\big)}

\newcommand{\Gamtamer}{(\widehat\gamma)_{tam}^{\displaystyle{n}}\big(x_{\kappa_n(r)}^{i,N,n},\mu^{x,N,n}_{\kappa_n(r)}, \bar z\big)}

\newcommand{\Gamztame}{(\widehat\gamma)_{tam}^{\displaystyle{n}}\big(x_{\kappa_n(s)}^{i,N,n},\mu^{x,N,n}_{\kappa_n(s)}, z\big)}

\newcommand{\Gamtamej}{(\widehat\gamma)_{tam}^{\displaystyle{n}}\big(x_{\kappa_n(s)}^{j,N,n},\mu^{x,N,n}_{\kappa_n(s)},  z\big)}

\newcommand{\Gamtamejhat}{(\widehat\gamma)_{tam}^{\displaystyle{n}}\big(x_{\kappa_n(s)}^{\widehat j,N,n},\mu^{x,N,n}_{\kappa_n(s)},  z\big)}

\newcommand{\Gamtamek}{(\widehat\gamma)_{tam}^{\displaystyle{n}}\big(x_{\kappa_n(s)}^{k,N,n},\mu^{x,N,n}_{\kappa_n(s)}, z\big)}

\newcommand{\Gamtamekzbar}{(\widehat\gamma)_{tam}^{\displaystyle{n}}\big(x_{\kappa_n(s)}^{k,N,n},\mu^{x,N,n}_{\kappa_n(s)}, \bar z\big)}

\newcommand{\Gamtamekzbarr}{(\widehat\gamma)_{tam}^{\displaystyle{n}}\big(x_{\kappa_n(r)}^{k,N,n},\mu^{x,N,n}_{\kappa_n(r)}, \bar z\big)}

\newcommand{\Gamtamei}{(\widehat\gamma)_{tam}^{\displaystyle{n}}\big(x_{\kappa_n(s)}^{i,N,n},\mu^{x,N,n}_{\kappa_n(s)}, z\big)}

\newcommand{\Gamtameizbar}{(\widehat\gamma)_{tam}^{\displaystyle{n}}\big(x_{\kappa_n(s)}^{i,N,n},\mu^{x,N,n}_{\kappa_n(s)}, \bar z\big)}

\newcommand{\Tildegam}{(\widehat\Gamma)_{tam}^{\displaystyle{n}}(s,x_{\kappa_n(s)}^{i,N,n},\mu^{x,N,n}_{\kappa_n(s)},z)}

\newcommand{\Gamone}{(\widehat\gamma_1)_{tam}^{\displaystyle{n}}\big(\kappa_n(s),s,x_{\kappa_n(s)}^{i,N,n},\mu^{x,N,n}_{\kappa_n(s)},\bar z\big)}

\newcommand{\Gamtwo}{(\widehat\gamma_2)_{tam}^{\displaystyle{n}}\big(\kappa_n(s),s,x_{\kappa_n(s)}^{i,N,n},\mu^{x,N,n}_{\kappa_n(s)},\bar z\big)}

\newcommand{\Gamthree}{(\widehat\gamma_2)_{tam}^{\displaystyle{n}}(\kappa_n(s),s,x_{\kappa_n(s)}^{i,N,n},\mu^{x,N,n}_{\kappa_n(s)},\bar z\big)}

\newcommand{\Gamq}{(\widehat\gamma_{\mathfrak{q}})_{tam}^{\displaystyle{n}}\big(\kappa_n(s),s,x_{\kappa_n(s)}^{i,N,n},\mu^{x,N,n}_{\kappa_n(s)},\bar z\big)}

\newcommand{\Gamqr}{(\widehat\gamma_{\mathfrak{q}})_{tam}^{\displaystyle{n}}\big(\kappa_n(r),r,x_{\kappa_n(r)}^{i,N,n},\mu^{x,N,n}_{\kappa_n(r)},\bar z\big)}

\newcommand{\Gamqj}{(\widehat\gamma_{\mathfrak{q}})_{tam}^{\displaystyle{n}}\big(\kappa_n(s),s,x_{\kappa_n(s)}^{j,N,n},\mu^{x,N,n}_{\kappa_n(s)},\bar z\big)}

\newcommand{\Gamqk}{(\widehat\gamma_{\mathfrak{q}})_{tam}^{\displaystyle{n}}\big(\kappa_n(s),s,x_{\kappa_n(s)}^{k,N,n},\mu^{x,N,n}_{\kappa_n(s)},\bar z\big)}

\newcommand{\Gamqkr}{(\widehat\gamma_{\mathfrak{q}})_{tam}^{\displaystyle{n}}\big(\kappa_n(r),r,x_{\kappa_n(r)}^{k,N,n},\mu^{x,N,n}_{\kappa_n(r)},\bar z\big)}
%%%%%%%%%%%%%%%%%%%%%%%%%%%%%%%%%%%%%%%%%%%%%%%%%%%%%%%

%% ADDITIONAL COMMAND
%%%%%%%%%%%%%%%%%%%%%%%%%%%%%%%%%%%%%%%%%%%%%%%%%%%%%%

\newcommand{\empiri}{\frac{1}{N}\sum_{j=1}^{N}\delta_{x_{r}^{j,N,n}+1_{\{j=k\}}\Gammatamejr}}

\newcommand{\empirical}{\frac{1}{N}\sum_{j=1}^{N}\delta_{x_{\kappa_n(s)}^{j,N,n}+1_{\{j=k\}}\gamma(x_{\kappa_n(s)}^{j,N,n},\,\mu_{\kappa_n(s)}^{x,N,n},\,z)}}

\newcommand{\empiricaltame}{\frac{1}{N}\sum_{j=1}^{N}\delta_{x_{\kappa_n(s)}^{j,N,n}+1_{\{j=k\}}\Gamtamej}}

\newcommand{\empiricaltamei}{\frac{1}{N}\sum_{j=1}^{N}\delta_{x_{\kappa_n(s)}^{j,N,n}+1_{\{j=i\}}\Gamtamej}}

\newcommand{\dif}{x_{\kappa_n(s)}^{j,N,n}+\theta\big\{x_{\kappa_n(s)}^{j,N,n}+1_{\{j=k\}}\gamma(x_{\kappa_n(s)}^{j,N,n},\,\mu_{\kappa_n(s)}^{x,N,n},\,z)-x_{\kappa_n(s)}^{j,N,n}\big\}}

\newcommand{\diftame}{x_{\kappa_n(s)}^{j,N,n}+\theta\big\{x_{\kappa_n(s)}^{j,N,n}+1_{\{j=k\}}\Gamtamej-x_{\kappa_n(s)}^{j,N,n}\big\}}

\newcommand{\difone}{x_{\kappa_n(s)}^{j,N,n}+1_{\{j=k\}}\theta\gamma(x_{\kappa_n(s)}^{j,N,n},\,\mu_{\kappa_n(s)}^{x,N,n},\,z)}

\newcommand{\difonetame}{x_{\kappa_n(s)}^{j,N,n}+1_{\{j=k\}}\theta\Gamtamej}

\newcommand{\empiricaltheta}{\frac{1}{N}\sum_{j=1}^{N}\delta_{\dif}}

\newcommand{\empiricalthetatame}{\frac{1}{N}\sum_{j=1}^{N}\delta_{\diftame}}

\newcommand{\empiricalthetaone}{\frac{1}{N}\sum_{j=1}^{N}\delta_{\difone}}

\newcommand{\empiricalthetaonetame}{\frac{1}{N}\sum_{j=1}^{N}\delta_{x_{\kappa_n(s)}^{j,N,n}+1_{\{j=k\}}\theta\Gamtamej}}

\newcommand{\tilmur}{\mu_{r-}^{x+\gamma(k,z),N}}

\newcommand{\tilmukappa}{\mu_{\kappa_{n}(s)}^{x+\gamma(k,z),N,n}}

\newcommand{\tilmukappaX}{\mu_{\kappa_{n}(s)}^{X+\gamma(k,z),N,n}}

\newcommand{\tilmukappai}{\mu_{\kappa_{n}(s)}^{x+\gamma(i,z),N,n}}

\newcommand{\tilmukappatame}{\mu_{\kappa_{n}(s)}^{x+(\widehat{\gamma})_{tam}^n(k,z),N,n}}

\newcommand{\tilmukappatamezbar}{\mu_{\kappa_{n}(s)}^{x+(\widehat{\gamma})_{tam}^n(k,\bar z),N,n}}

\newcommand{\tilmukappatamezbarr}{\mu_{\kappa_{n}(r)}^{x+(\widehat{\gamma})_{tam}^n(k,\bar z),N,n}}

\newcommand{\tilpikappaGammatame}{\pi_{r,\kappa_{n}(r)}^{x+(\widehat{\Gamma})_{tam}^n(k,\bar z),N,n}}

\newcommand{\tilpikappaGammatameve}{\pi_{r-,\kappa_{n}(r)}^{x+(\widehat{\Gamma})_{tam}^n(k,\bar z),N,n}}

\newcommand{\tilpikappaGammatamei}{\pi_{r,\kappa_{n}(r)}^{x+(\widehat{\Gamma})_{tam}^n(i,\bar z),N,n}}

\newcommand{\tilpikappaGammatameive}{\pi_{r-,\kappa_{n}(r)}^{x+(\widehat{\Gamma})_{tam}^n(i,\bar z),N,n}}

\newcommand{\tilmukappatamei}{\mu_{\kappa_{n}(s)}^{x+(\widehat{\gamma})_{tam}^n(i,z),N,n}}

\newcommand{\tilmukappatameizbar}{\mu_{\kappa_{n}(s)}^{x+(\widehat{\gamma})_{tam}^n(i,\bar z),N,n}}

\newcommand{\tilmukappatameizbarr}{\mu_{\kappa_{n}(r)}^{x+(\widehat{\gamma})_{tam}^n(i,\bar z),N,n}}

%%%%%%%%%%%%%%%%%%%%%%%%%%%%%%%%%%%%%%%%%%%%%%%%%%%%%%%

\newcommand{\Lambdatame}{(\widehat{\Lambda}^{\ell_1})_{tam}^{\displaystyle n}\big(\kappa_{n}(r),r,x_{\kappa_{n}(r)}^{i,N,n},\mu_{\kappa_{n}(r)}^{x,N,n}\big)}

\newcommand{\Lambdatamer}{(\widehat{\Lambda}^{\ell_1})_{tam}^{\displaystyle n}\big(\kappa_{n}(r),r,x_{\kappa_{n}(r)}^{i,N,n},\mu_{\kappa_{n}(r)}^{x,N,n}\big)}

\newcommand{\Lambdatam}{(\widehat{\Lambda})_{tam}^{\displaystyle n}\big(\kappa_{n}(r),r,x_{\kappa_{n}(r)}^{i,N,n},\mu_{\kappa_{n}(r)}^{x,N,n}\big)}

\newcommand{\Lambdatams}{(\widehat{\Lambda})_{tam}^{\displaystyle n}\big(\kappa_{n}(s),s,x_{\kappa_{n}(s)}^{i,N,n},\mu_{\kappa_{n}(s)}^{x,N,n}\big)}

\newcommand{\Lambdatamr}{(\widehat{\Lambda})_{tam}^{\displaystyle n}\big(\kappa_{n}(r),r,x_{\kappa_{n}(r)}^{i,N,n},\mu_{\kappa_{n}(r)}^{x,N,n}\big)}

\newcommand{\Lambdatamru}{(\widehat{\Lambda}^{(u)})_{tam}^{\displaystyle n}\big(\kappa_{n}(r),r,x_{\kappa_{n}(r)}^{i,N,n},\mu_{\kappa_{n}(r)}^{x,N,n}\big)}

\newcommand{\Lambdatamsu}{(\widehat{\Lambda}^{(u)})_{tam}^{\displaystyle n}\big(\kappa_{n}(s),s,x_{\kappa_{n}(s)}^{i,N,n},\mu_{\kappa_{n}(s)}^{x,N,n}\big)}

\newcommand{\Lambdatamtrans}{(\widehat{\Lambda}^*)_{tam}^{{\displaystyle n}}\big(\kappa_{n}(r),r,x_{\kappa_{n}(r)}^{i,N,n},\mu_{\kappa_{n}(r)}^{x,N,n}\big)}

\newcommand{\Lambdatamtransr}{(\widehat{\Lambda}^*)_{tam}^{{\displaystyle n}}\big(\kappa_{n}(r),r,x_{\kappa_{n}(r)}^{i,N,n},\mu_{\kappa_{n}(r)}^{x,N,n}\big)}

\newcommand{\Lambdatameu}{(\widehat{\Lambda}^{u\ell})_{tam}^{\displaystyle n}\big(\kappa_{n}(r),r,x_{\kappa_{n}(r)}^{i,N,n},\mu_{\kappa_{n}(r)}^{x,N,n}\big)}

\newcommand{\Lambdatameus}{(\widehat{\Lambda}^{u\ell})_{tam}^{\displaystyle n}\big(\kappa_{n}(s),s,x_{\kappa_{n}(s)}^{i,N,n},\mu_{\kappa_{n}(s)}^{x,N,n}\big)}

\newcommand{\Psitameu}{(\widehat{\Psi}^{u})_{tam}^{\displaystyle n}\big(\kappa_{n}(r),r,x_{\kappa_{n}(s)}^{i,N,n},\mu_{\kappa_{n}(r)}^{x,N,n}\big)}

\newcommand{\Psitame}{(\widehat{\Psi})_{tam}^{\displaystyle n}\big(\kappa_{n}(s),s,x_{\kappa_{n}(s)}^{i,N,n},\mu_{\kappa_{n}(s)}^{x,N,n}\big)}

\newcommand{\Psitamek}{(\widehat{\Psi})_{tam}^{\displaystyle n}\big(\kappa_{n}(s),s,x_{\kappa_{n}(s)}^{k,N,n},\mu_{\kappa_{n}(s)}^{x,N,n}\big)}

\newcommand{\Lambdatamek}{(\widehat{\Lambda}^{\ell_1})_{tam}^{\displaystyle n}\big(\kappa_{n}(r),r,x_{\kappa_{n}(r)}^{k,N,n},\mu_{\kappa_{n}(r)}^{x,N,n}\big)}

\newcommand{\Lambdatamekr}{(\widehat{\Lambda}^{\ell_1})_{tam}^{\displaystyle n}\big(\kappa_{n}(r),r,x_{\kappa_{n}(r)}^{k,N,n},\mu_{\kappa_{n}(r)}^{x,N,n}\big)}

\newcommand{\Lambdatamk}{(\widehat{\Lambda})_{tam}^{\displaystyle n}\big(\kappa_{n}(r),r,x_{\kappa_{n}(r)}^{k,N,n},\mu_{\kappa_{n}(r)}^{x,N,n}\big)}

\newcommand{\Lambdatamkr}{(\widehat{\Lambda})_{tam}^{\displaystyle n}\big(\kappa_{n}(r),r,x_{\kappa_{n}(r)}^{k,N,n},\mu_{\kappa_{n}(r)}^{x,N,n}\big)}

\newcommand{\Lambdatamktrans}{(\widehat{\Lambda}^*)_{tam}^{{\displaystyle n}}\big(\kappa_{n}(r),r,x_{\kappa_{n}(r)}^{k,N,n},\mu_{\kappa_{n}(r)}^{x,N,n}\big)}

\newcommand{\Lambdatamktransr}{(\widehat{\Lambda}^*)_{tam}^{{\displaystyle n}}\big(\kappa_{n}(r),r,x_{\kappa_{n}(r)}^{k,N,n},\mu_{\kappa_{n}(r)}^{x,N,n}\big)}

\newcommand{\Gammatame}{(\widehat{\Gamma})_{tam}^{\displaystyle n}\big(\kappa_{n}(r),r,x_{\kappa_{n}(r)}^{i,N,n},\mu_{\kappa_{n}(r)}^{x,N,n},\bar z\big)}

\newcommand{\Gammatames}{(\widehat{\Gamma})_{tam}^{\displaystyle n}\big(\kappa_{n}(s),s,x_{\kappa_{n}(s)}^{i,N,n},\mu_{\kappa_{n}(s)}^{x,N,n},\bar z\big)}

\newcommand{\Gammaztame}{(\widehat{\Gamma})_{tam}^{\displaystyle n}\big(\kappa_{n}(r),r,x_{\kappa_{n}(r)}^{i,N,n},\mu_{\kappa_{n}(r)}^{x,N,n},z\big)}

\newcommand{\Gammaztames}{(\widehat{\Gamma})_{tam}^{\displaystyle n}\big(\kappa_{n}(s),s,x_{\kappa_{n}(s)}^{i,N,n},\mu_{\kappa_{n}(s)}^{x,N,n},z\big)}

\newcommand{\Gammatamer}{(\widehat{\Gamma})_{tam}^{\displaystyle n}\big(\kappa_{n}(r),r,x_{\kappa_{n}(r)}^{i,N,n},\mu_{\kappa_{n}(r)}^{x,N,n},\bar z\big)}

\newcommand{\Gammatameru}{(\widehat{\Gamma}^u)_{tam}^{\displaystyle n}\big(\kappa_{n}(r),r,x_{\kappa_{n}(r)}^{i,N,n},\mu_{\kappa_{n}(r)}^{x,N,n},\bar z\big)}

\newcommand{\Gammaztamer}{(\widehat{\Gamma})_{tam}^{\displaystyle n}\big(\kappa_{n}(r),r,x_{\kappa_{n}(r)}^{i,N,n},\mu_{\kappa_{n}(r)}^{x,N,n},z\big)}

\newcommand{\Gammatameu}{(\widehat{\Gamma}^u)_{tam}^{\displaystyle n}\big(\kappa_{n}(r),r,x_{\kappa_{n}(r)}^{i,N,n},\mu_{\kappa_{n}(r)}^{x,N,n},\bar z\big)}

 \newcommand{\Gammatameus}{(\widehat{\Gamma}^u)_{tam}^{\displaystyle n}\big(\kappa_{n}(s),s,x_{\kappa_{n}(s)}^{i,N,n},\mu_{\kappa_{n}(s)}^{x,N,n},\bar z\big)}
 
\newcommand{\Gammaztameu}{(\widehat{\Gamma}^u)_{tam}^{\displaystyle n}\big(\kappa_{n}(s),s,x_{\kappa_{n}(s)}^{i,N,n},\mu_{\kappa_{n}(s)}^{x,N,n}, z\big)}

\newcommand{\Gammaztameur}{(\widehat{\Gamma}^u)_{tam}^{\displaystyle n}\big(\kappa_{n}(r),r,x_{\kappa_{n}(r)}^{i,N,n},\mu_{\kappa_{n}(r)}^{x,N,n}, z\big)}

\newcommand{\Gammatamej}{(\widehat{\Gamma})_{tam}^{\displaystyle n}\big(\kappa_{n}(r),r,x_{\kappa_{n}(r)}^{ j,N,n},\mu_{\kappa_{n}(r)}^{x,N,n},\bar z\big)}

\newcommand{\Gammaztamej}{(\widehat{\Gamma})_{tam}^{\displaystyle n}\big(\kappa_{n}(r),r,x_{\kappa_{n}(r)}^{ j,N,n},\mu_{\kappa_{n}(r)}^{x,N,n},z\big)}

\newcommand{\Gammatamejanother}{(\widehat{\Gamma})_{tam}^{n}\big(\kappa_{n}(r),r,x_{\kappa_{n}(r)}^{ j,N,n},\mu_{\kappa_{n}(r)}^{x,N,n},\bar z\big)}

\newcommand{\Gammaztamejanother}{(\widehat{\Gamma})_{tam}^{n}\big(\kappa_{n}(r),r,x_{\kappa_{n}(r)}^{ j,N,n},\mu_{\kappa_{n}(r)}^{x,N,n},z\big)}

\newcommand{\Gammatamejhat}{(\widehat{\Gamma})_{tam}^{\displaystyle n}(\kappa_{n}(r),r,x_{\kappa_{n}(r)}^{\widehat j,N,n},\mu_{\kappa_{n}(r)}^{x,N,n},\bar z)}

\newcommand{\Gammatamek}{(\widehat{\Gamma})_{tam}^{\displaystyle n}(\kappa_{n}(r),r,x_{\kappa_{n}(r)}^{k,N,n},\mu_{\kappa_{n}(r)}^{x,N,n},\bar z)}

\newcommand{\Gammaztamek}{(\widehat{\Gamma})_{tam}^{\displaystyle n}(\kappa_{n}(r),r,x_{\kappa_{n}(r)}^{k,N,n},\mu_{\kappa_{n}(r)}^{x,N,n}, z)}

\newcommand{\Gammatamekr}{(\widehat{\Gamma})_{tam}^{\displaystyle n}(\kappa_{n}(r),r,x_{\kappa_{n}(r)}^{k,N,n},\mu_{\kappa_{n}(r)}^{x,N,n},\bar z)}

\newcommand{\Gammaztamekr}{(\widehat{\Gamma})_{tam}^{\displaystyle n}(\kappa_{n}(r),r,x_{\kappa_{n}(r)}^{k,N,n},\mu_{\kappa_{n}(r)}^{x,N,n},z)}

\newcommand{\Gammatamejr}{(\widehat{\Gamma})_{tam}^{\displaystyle n}(\kappa_{n}(r),\,r,\,x_{\kappa_{n}(r)}^{j,N,n},\,\mu_{\kappa_{n}(r)}^{x,N,n},\,\bar z)}

\newcommand{\Gammatamejs}{(\widehat{\Gamma})_{tam}^{n}(\kappa_{n}(s),\,s,\,x_{\kappa_{n}(s)}^{j,N,n},\,\mu_{\kappa_{n}(s)}^{x,N,n},\,\bar z)}

\newcommand{\Gammaztamejr}{(\widehat{\Gamma})_{tam}^{\displaystyle n}(\kappa_{n}(r),\,r,\,x_{\kappa_{n}(r)}^{j,N,n},\,\mu_{\kappa_{n}(r)}^{x,N,n},\,z)}

\newcommand{\Gammaztamejs}{(\widehat{\Gamma})_{tam}^{n}(\kappa_{n}(s),\,s,\,x_{\kappa_{n}(s)}^{j,N,n},\,\mu_{\kappa_{n}(s)}^{x,N,n},\, z)}

%________________For Examples_____________________
\newcommand{\f}{f(x,u,\mu)}

\newcommand{\fbar}{f(\bar x,\bar u,\bar\mu)}

\newcommand{\g}{g(x,u,v)}

\newcommand{\gbar}{g(\bar x,\bar u,\bar v)}

\newcommand{\uu}{h(x,\mu)}

\newcommand{\ubar}{h(\bar x, \bar \mu)}

%%%%%%%%%%%%%%%%%% DERIVATIVES %%%%%%%%%%%%%%%%%%%%

\newcommand{\Dzf}{\partial_x \f}

\newcommand{\Dzfbar}{\partial_x \fbar}

\newcommand{\Dmuf}{\partial_\mu f(x,u,\mu,v)}

\newcommand{\Dmufbar}{\partial_\mu f(\bar x,\bar u,\bar \mu,\bar v)}

%-------------------------------------------------

\newcommand{\Dzg}{\partial_x \g}

\newcommand{\Dzgbar}{\partial_x \gbar}

\newcommand{\Dxg}{\partial_u \g}

\newcommand{\Dxgbar}{\partial_u \gbar}

\newcommand{\Dyg}{\partial_v \g}

\newcommand{\Dygbar}{\partial_v \gbar}

%-------------------------------------------------

\newcommand{\Dzu}{\partial_x \uu}

\newcommand{\Dzubar}{\partial_x \ubar}

\newcommand{\Dmuu}{\partial_\mu h(x,\mu,u)}

\newcommand{\Dmuubar}{\partial_\mu h(\bar x,\bar \mu,\bar u)}

%-------------------------Taming-----------------------
\newcommand{\tame}{1+n^{-1}|x|^{3\eta}}
\newcommand{\tamegam}{1+n^{-1}|x|^{2\eta \bar q}}
\newcommand{\tamedervsig}{1+n^{-1}|x|^{6\eta}}
\newcommand{\tamedervgam}{1+n^{-1}|x|^{4\eta}}
\newcommand{\tamedervmugam}{1+n^{-1}|x|^{2\eta}}

\newcommand{\tamekappa}{1+n^{-1}|x_{\kappa_{n}(s)}^{i,N,n}|^{3\eta}}
\newcommand{\tamegamkappa}{1+n^{-1}|x_{\kappa_{n}(s)}^{i,N,n}|^{2\eta\bar q}}
\newcommand{\tamedervsigkappa}{1+n^{-1}|x_{\kappa_{n}(s)}^{i,N,n}|^{6\eta}}
\newcommand{\tamedervgamkappa}{1+n^{-1}|x_{\kappa_{n}(s)}^{i,N,n}|^{4\eta}}
\newcommand{\tamedervmugamkappa}{1+n^{-1}|x_{\kappa_{n}(s)}^{i,N,n}|^{2\eta}}

\newcommand{\bbx}{b(x,\mu)}
\newcommand{\sigx}{\sigma(x,\mu)}
\newcommand{\gamxz}{\gamma(x,\mu,z)}
\newcommand{\sigxsigx}{{\mathfrak{D}}_x^{\sigma^{\ell_1}}\sigma^{u\ell}(x, \mu)}
\newcommand{\sigmusigx}{{\mathfrak{D}}_\mu^{\sigma^{\ell_1}}\sigma^{u\ell}(x,\mu,y)}
\newcommand{\gamxsigxz}{{\mathfrak{D}}_x^{\sigma^{\ell_1}}\gamma^u(x, \mu,z)}
\newcommand{\gammusigxz}{{\mathfrak{D}}_\mu^{\sigma^{\ell_1}}\gamma^u(x,\mu,y, z)}

\newcommand{\gammaexample}{x(1+x^2)^{1/\bar p}}

\newcommand{\gammaexampleq}{x(1+x^2)^{1/q}}

%% file: main.bbl
\begin{thebibliography}{10}
	\bibitem{agarwal2023}
	Agarwal, A.,  Amato, A., dos Reis, G. and Pagliarani, S.:
	Numerical approximation of McKean-Vlasov SDEs via stochastic gradient descent,  (2023), \url{https://arxiv.org/abs/2310.13579} 
	
	\bibitem{Baladron2012}
	Baladron, J., Fasoli, D., Faugeras, O. and Touboul, J.:  Mean-field description and propagation of chaos in networks of Hodgkin--Huxley and FitzHugh--Nagumo neurons. \emph{J. Math. Neurosci.} \textbf{2}, (2012), 1--50. 
	
	\bibitem{bao2021-I}
	Bao, J., and Huang, X.: Approximations of McKean--Vlasov stochastic differential equations with irregular coefficients. \emph{J. Theoret. Probab.} \textbf{35}(2), (2022),  1187--1215.
	
	\bibitem{bao2021-II}
	Bao, J., Reisinger, C., Ren, P., and Stockinger, W.:  First order convergence of Milstein schemes for McKean--Vlasov equations and interacting particle systems. \emph{Proc. R. Soc. A.} \textbf{477}(2245), (2021). 
	
	\bibitem{bao2021-III}
	Bao, J.,  Reisinger, C., Ren, P., and Stockinger, W.: 
	Milstein schemes and antithetic multi-level Monte Carlo sampling for delay McKean--Vlasov equations and interacting particle systems, (2023), \url{https://arxiv.org/abs/2005.01165}
	
	\bibitem{Biswas2022}
	Biswas, S., Kumar, C., Neelima,  Reis, G. dos and Reisinger, C.: An explicit Milstein-type scheme for interacting particle systems and McKean--Vlasov SDEs with common noise and non-differentiable drift coefficients. \emph{The Ann.  Appl. Probab.} \textbf{34}(2), (2024), 2326--2363.  
	
	\bibitem{Bolley2011} 
	Bolley, F., Ca\~{n}izo, J. A. and Carrillo, J. A.: Stochastic mean-field limit: non-Lipschitz forces and swarming. \emph{Math. Models Methods Appl. Sci.} \textbf{21}(11),  (2011), 2179--2210. 
 
	\bibitem{Bossy2015}
	Bossy, M., Faugeras, O. and Talay, D.: Clarification and complement to ``mean-field description and propagation of chaos in networks of Hodgkin--Huxley and FitzHugh--Nagumo neurons''. \emph{J. Math. Neurosci.}  \textbf{5}, (2015), 1--23. 
	
	\bibitem{Buckdahn2017}
	Buckdahn, R., Li, J., Peng, S. and Rainer, C.: Mean-field stochastic differential equations and associated PDEs.  \emph{Ann. Probab.} \textbf{45}(2), (2017), 824--878. 
	
	\bibitem{Cardaliaquet2013}
	Cardaliaguet, P.: Notes on mean-field games, notes from P. L. Lions lectures at Coll{\`e}ge de France. (2013), \url{https://www.ceremade.dauphine.fr/~cardaliaguet/MFG20130420.pdf}. 
	
	\bibitem{Carmona2018-I}
	Carmona, R. and  Delarue, F.: Probabilistic theory of mean field games with applications I: Mean field FBSDEs, control, and games. \emph{Springer, Cham} \textbf{83},  (2018).
	
	\bibitem{Carmona2018-II}
	Carmona, R. and  Delarue, F.: Probabilistic theory of mean field games with applications II: Mean field games with common noise and master equations. \emph{Springer, Cham} \textbf{84},   (2018). 
	
	\bibitem{Carrillo2010}
	Carrillo, J.~A.,  Fornasier, M., Toscani, G.~and Vecil, F.:
	{Particle, kinetic, and hydrodynamic models of swarming}: {Mathematical modeling of collective behavior in socio-economic and life sciences.} 
 \emph{Birkh\"auser Boston, Boston, MA} (2010), 297--336. 
	
	\bibitem{Chen2021}
	Chen, X.  and dos Reis, G.: A flexible split‐step scheme for solving McKean--Vlasov stochastic differential equations. {Appl. Math. Comput.} \textbf{427},  (2022). 

\bibitem{Chen2024}
 Chen, X. and dos Reis, G.:   Euler simulation of interacting particle systems and McKean–Vlasov SDEs with fully super-linear growth drifts in space and interaction. \emph{IMA J. Numer. Anal.}  \textbf{44}(2),   (2024), 751--796. 
	
	\bibitem{Chen2023}
	Chen, X.,  dos Reis, G. and  Stockinger, W.: 
	Wellposedness, exponential ergodicity and numerical approximation of fully super-linear McKean--Vlasov SDEs and associated particle systems, (2023),  \url{https://arxiv.org/abs/2302.05133}

\bibitem{Chen2019}
Chen,  Z.,  Gan, S. and  Wang, X.: Mean-square approximations of L\'evy noise driven SDEs with super-linearly growing diffusion and jump coefficients. \emph{Discrete Contin. Dyn. Syst. Ser. B} \textbf{24}(8),  (2019), 4513--4545. 
	
	\bibitem{Chiara2017}
	Chiara, B., Campi, L. and Di Perso, L.: Mean field games with controlled jump--diffusion dynamics: Existence results and an illiquid interbank market model. \emph{Stochastic Process. Appl.} \textbf{130}(11), (2020), 6927--6964. 
	
	\bibitem{Dareiotis2016}
	Dareiotis, K., Kumar, C. and Sabanis, S.: On tamed Euler approximations of SDEs driven by L\'evy noise with applications to delay equations. \emph{SIAM J. Numer. Anal.} \textbf{54}(3),  (2016), 1840--1872. 
	
	\bibitem{Delarue2015}
	Delarue, F., Inglis, J., Rubenthaler, S. and Tanré, E.: Global solvability of a networked integrate-and-fire model of McKean–Vlasov type. \emph{Ann.  Appl. Probab.} \textbf{25}(4),  (2015), 2096--2133. 
	
	\bibitem{Reis2019}
	dos Reis, G., Engelhardt, S. and Smith G.: Simulation of McKean--Vlasov SDEs with super-linear growth. \emph{IMA J. Numer. Anal.} \textbf{42}(1), (2022), 874--922.  
	
	\bibitem{Dreyer2011} 
	Dreyer, W., Gaber\v{s}\v{c}ek, M., Guhlke, C., Huth, R. and Jamnik, J.: Phase transition in a rechargeable lithium battery. \emph{Eur. J. Appl. Math.} \textbf{22}(3),  (2011), 267--290. 
	
	\bibitem{Goddard2012}
	Goddard, B.~D., Pavliotis, G.~A. and Kalliadasis, S.: 
	The overdamped limit of dynamic density functional theory: rigorous results. 
	\emph{Multiscale Model. Simul.} \textbf{10}(2),  (2012), 633--663. 
	
	\bibitem{Guhlke2018}  
	Guhlke, C., Gajewski, P., Maurelli, M., Friz, P. K. and Dreyer, W.: Stochastic many-particle model for LFP electrodes, \emph{Contin. Mech. Thermodyn.} \textbf{30}(3),  (2018), 593--628. 
	
	\bibitem{pham2023}
	Guoa,X.,  Pham, H. and  Wei, X.:
	It\^o's formula for flows of measures on semimartingales, 
	\emph{Stochastic Process. Appl.} \textbf{159},  (2023), 350--390. 
	
	\bibitem{Hao2016}
	Hao, T. and Li, J.:  Mean-field SDEs with jumps and nonlocal integral-PDEs. \emph{NoDEA Nonlinear Differential Equations Appl.} \textbf{23}(2), (2016), 1--51. 
	
	\bibitem{Holm2006}
	Holm, D.~D. and Putkaradze, V.:
	Formation of clumps and patches in self-aggregation of finite-size particles. 
	\emph{Phys. D} \textbf{220}(2),  (2006), 183--196. 
	
	\bibitem{hutzenthaler2015}
	Hutzenthaler, M. and Jentzen, A.: Numerical approximations of stochastic differential
	equations with non globally Lipschitz continuous coefficients. \emph{Mem. Amer. Math. Soc.} \textbf{236}(1112), (2015). 
	
	\bibitem{hutzenthaler2012}
	Hutzenthaler, M., Jentzen,  A.  and  Kloeden, P. E.: Strong convergence of an explicit numerical method for SDEs with nonglobally Lipschitz continuous coefficients. \emph{Ann. Appl. Probab.} \textbf{22}(4), (2012), 1611--1641. 
	
	\bibitem{Jin2020}
	Jin, S., Li, L.~and Liu, J.~G.: 
	Random batch methods (RBM) for interacting particle systems. 
	\emph{J. Comput. Phys.} \textbf{400}, (2020).
	
	\bibitem{Kolokolnikov2013}
	Kolokolnikov, T., Carrillo, J.~A., Bertozzi, A., Fetecau, R.~and Lewis, M.: 
	Emergent behaviour in multi-particle systems with non-local interactions. 
	\emph{Phys. D} \textbf{260}, (2013), 1--4. 
	
	\bibitem{Kumar2021a}
	Kumar, C.: On Milstein-type scheme for SDE driven by Lévy noise with super-linear coefficients.  \emph{Discrete Contin. Dyn. Syst. Ser. B} \textbf{26}(3),  (2021),  1405--1446.

\bibitem{Kumar2021}
	Kumar, C. and Neelima: On Explicit Milstein-type Scheme for
	McKean--Vlasov stochastic differential equations with superlinear 
	drift coefficient. \emph{Electron. J. Probab.} \textbf{26},  (2021), 1--32. 
	
	\bibitem{C.Kumar2021}
	Kumar, C., Neelima, Reisinger, C. and   Stockinger, W.: Well-posedness and tamed schemes for McKean–Vlasov equations with common noise. \emph{Ann. Appl. Probab.} \textbf{32}(5), (2022), 3283--3330. 
 
\bibitem{Kumar2017a}
	Kumar, C. and  Sabanis, S.:
On explicit approximations for Lévy driven SDEs with
super-linear diffusion coefficients, \emph{Electron. J. Probab.} \textbf{22}(73), (2017), 1–19.

	\bibitem{Kumar2017}
	Kumar, C. and  Sabanis, S.: On Tamed Milstein Schemes of SDEs Driven by L\'evy Noise. \emph{Discrete Contin. Dyn. Syst. Ser. B} \textbf{22}(2),  (2017), 421--463. 
 
	\bibitem{Li2018}
	Li, J.: Mean-field forward and backward SDEs with jumps and
	associated nonlocal quasi-linear integral-PDEs. 
	\emph{Stochastic Process. Appl.} \textbf{128}(9),  (2018), 3118--3180. 
	
	\bibitem{Li2022}
	Li, Y., Mao, X., Song, Q., Wu, F. and Yin, G.: Strong convergence of Euler--Maruyama schemes for McKean--Vlasov stochastic differential equations under local Lipschitz conditions of state variables. \emph{IMA J. Numer. Anal.} \textbf{00}(2), (2021), 1001--1035.
	
	\bibitem{Mehri2020}
	Mehri, S., Scheutzow, M., Stannat, W. and Zangeneh, B. Z.: Propagation of chaos for stochastic spatially structured neuronal networks with delay driven by jump diffusions. \emph{Ann. Appl. Probab.} \textbf{30}(1),  (2020), 175--207. 
	
	\bibitem{Mikulevicius2012}
	Mikulevicius, R. and Pragarauskas, H.: 
	On $\mathcal L_p$-estimates of some singular integrals related to
	jump processes.
	\emph{SIAM J. Math. Anal.} 
	\textbf{44}(4), (2012), 2305--2328. 
	
	\bibitem{Neelima2020}
	Neelima, Biswas, S., Kumar, C.,  dos Reis, G. and Reisinger, C.: Well-posedness and tamed Euler schemes for McKean-Vlasov equations driven by L\'evy noise, (2020), \url{https://arxiv.org/abs/2010.08585}
	
	\bibitem{Platen2010}
	Platen, E. and Bruti-Liberati, N.:  Numerical solution of stochastic differential equations with jumps in finance: Stochastic Modelling and Applied Probability. \emph{Springer-Verlag, Berlin} \textbf{64}, (2010). 
	
	
	
	\bibitem{Reisinger2022}
	Reisinger, C. and Wolfgang, S.:
	An adaptive Euler--Maruyama scheme for McKean--Vlasov SDEs with super-linear growth and application to the mean-field FitzHugh--Nagumo model. \emph{J. Comput. Appl. Math.} 
	\textbf{400}, (2022).
	
	\bibitem{Salkeld2019}
	Salkeld, W., dos Reis, G.   and Tugaut, J.: 
	Freidlin-Wentzell LDPs in path space for McKean--Vlasov equations and the Functional Iterated Logarithm Law. \emph{Ann. Appl. Probab.} \textbf{29}(3),  (2019), 1487--1540. 
	
	\bibitem{Shen2023}
	Shen, G., Xiang, J. and Wu, J. L.:  Stochastic averaging principle and stability for multi-valued McKean-Vlasov stochastic differential equations with jumps,  (2023), \url{https://arxiv.org/abs/2308.02195}
	
	\bibitem{Situ2006}
	Situ, R.: Theory of stochastic differential equations with jumps and applications: Mathematical and analytical techniques with applications to engineering. \emph{Springer, New York} (2005).
	
	\bibitem{Ullner2018}
	Ullner, E., Politi, A. and Torcini, A.: Ubiquity of collective irregular dynamics in balanced networks of spiking neurons. \emph{Chaos} \textbf{28}(8), (2018).  
	
\end{thebibliography}
